\documentclass[a4paper]{article}
\usepackage[english]{babel}
\usepackage[utf8x]{inputenc}

\usepackage{booktabs}
\usepackage{tabu}
\usepackage[T1]{fontenc}
\usepackage{bigints}
\usepackage{newtxtext,newtxmath}

\usepackage[a4paper,top=3cm,bottom=2cm,left=3cm,right=3cm,marginparwidth=1.75cm]{geometry}

\usepackage{amsmath}

\usepackage{mathrsfs}
\usepackage{accents}
\usepackage{graphicx}
\usepackage{bigints}
\usepackage[colorinlistoftodos]{todonotes}
\usepackage[colorlinks=true, allcolors=black]
{hyperref}
\usepackage{amssymb}
\usepackage{amsthm}
\usepackage{dsfont}
\numberwithin{equation}{section}
\newtheorem{prop}{Proposition}
\newtheorem*{theorem*}{Theorem 1.A}
\newtheorem*{proof of theorem 1*}{Proof of Theorem 1}
\newtheorem{thm}{Theorem}
\newtheorem*{thmm}{Theorem}
\newtheorem*{ack}{Acknowledgments}
\newtheorem{lem}{Lemma}
\newtheorem*{lem*}{Lemma}
\newtheorem*{VDCL}{Van der Corput's Lemma}
\newtheorem*{VL}{Vaaler's Lemma}
\newtheorem*{HI}{Hilbert’s inequality}
\newtheorem{cor}{Corollary}
\newtheorem*{mydef}{Definition}
\numberwithin{lem}{section}
\numberwithin{prop}{section}
\numberwithin{cor}{section}
\newcommand{\RomanNumeralCaps}[1]
    {\MakeUppercase{\romannumeral #1}}
\title{\textbf{\mdseries{On an Analogue Of the Gauss Circle Problem For the Heisenberg Groups}}}
\author{\upshape{YOAV A. GATH}}
\date{}

\begin{document}

\maketitle
\begin{abstract}
We consider the problem of estimating the error term $\mathcal{E}_{q}(x)=\big|\mathbb{Z}^{2q+1}\cap\delta_{x}\mathcal{B}\big|-\textit{vol}\big(\mathcal{B}\big)x^{2q+2}$ which occurs in the counting of lattice points in Heisenberg dilates of the Cygan-Kor{\'a}nyi ball:
\begin{equation*}
\begin{split}
&\mathcal{B}=\big\{(v,w)\in\mathbb{R}^{2q}\times\mathbb{R}:\mathcal{N}(v,w)\leq1\big\}\quad;\quad\delta_{x}\big(v,w\big)=\big(xv,x^{2}w\big)
\end{split}
\end{equation*}
where $\mathcal{N}(v,w)=\big(|v|_{2}^{4}+w^{2}\big)^{1/4}$ is the Cygan-Kor{\'a}nyi norm, and $|\cdot|_{2}$ denotes the Euclidean norm. This lattice point counting problem arises naturally in the context of the Heisenberg groups, and may be viewed as a non-commutative analogue of the classical lattice point counting problem for Euclidean balls. In a previous paper, we have shown that the exponent in the upper bound $|\mathcal{E}_{1}(x)|\ll x^{2}\log{x}$ obtained by Garg, Nevo \& Taylor is best possible, thereby solving the problem for $q=1$. In the higher dimensional case, the behavior of the error term is of an entirely different nature, and is closely related both in shape and form to the error term in the Gauss circle problem as soon as $q\geq3$, while $q=2$ marks somewhat of an intermediate point.\\ In the present paper, we shall prove three type of results regarding the order of magnitude of $\mathcal{E}_{q}(x)$, which are valid for any $q\geq3$. An upper bound estimate of the form $|\mathcal{E}_{q}(x)|\ll x^{2q-2/3}$ ; A sharp second moment estimate, which shows that $\mathcal{E}_{q}(x)$ has order of magnitude $x^{2q-1}$ in mean-square ; And an $\Omega$-estimate of the form $\mathcal{E}_{q}(x)=\Omega\big(x^{2q-1}\big(\log{x}\big)^{1/4}\big(\log{\log{x}}\big)^{1/8}\big)$. Consequently, we obtain the lower bound $\kappa_{q}=\sup\big\{\alpha>0:\big|\mathcal{E}_{q}(x)\big|\ll x^{2q+2-\alpha}\big\}\geq\frac{8}{3}$ for $q\geq3$, and conjecture that $\kappa_{q}=3$. .
\end{abstract}
\maketitle
 
\tableofcontents
\section{Introduction, notation and statement of results}
\subsection{Lattice points in Euclidean balls - Motivation and general principles}
A classical problem in analytic number theory concerns the counting of lattice points in $n$-dimensional balls. To motivate the set up for the lattice point counting problem which we shall consider in the present paper, we shall state the problem in the following form. We view the Abelian group $\mathbb{E}_{n}=\big(\mathbb{R}^{n},+\big)$ as an homogeneous group, where the dilation group $\big\{\uplambda_{x}\big\}_{x\in\mathbb{R}_{+}}$ is given by the familiar Euclidean dilations $\uplambda_{x}u=xu$. Let $|\cdot|_{2}$ denote the Euclidean norm, which is the \textbf{canonical} sub-additive homogeneous group norm on $\mathbb{E}_{n}$, and write $\mathscr{O}=\big\{u\in\mathbb{R}^{n}:|u|_{2}\leq1\big\}$ for the unit norm ball. It is well known that the number of lattice points in the Euclidean dilated body $\uplambda_{x}\mathscr{O}$ is asymptotic to $\textit{vol}\big(\mathscr{O}\big)x^{n}$, where $\textit{vol}(\cdot)$ is the $n$-dimensional volume. We thus define the error term to be
\begin{equation*}
\mathcal{E}^{\mathbb{E}}_{n}(x)=\big|\mathbb{Z}^{n}\cap\uplambda_{x}\mathscr{O}\big|-\textit{vol}\big(\mathscr{O}\big)x^{n}
\end{equation*}
with the aim of obtaining upper/lower bounds for:
\begin{equation*}
\text{ }\qquad\varkappa_{n}=\sup\big\{\alpha>0:\big|\mathcal{E}^{\mathbb{E}}_{n}(x)\big|\ll x^{n-\alpha}\big\}\,.
\end{equation*}
For $n\geq4$ one has $\varkappa_{n}=2$, which follows from classical results on representation of integers by quadratic forms, and so the problem is settled in the higher dimensional case. The determination of $\varkappa_{n}$ for $n=2,3$ are amongst the most famous open problems in analytic number theory, where a solution to this problem, in either case, would constitute a landmark achievement. For $n=3$ one has $\varkappa_{3}\geq\frac{27}{16}$ due to Heath-Brown \cite{heath1999lattice}, and it is conjectured that $\varkappa_{3}=2$.\\\\
Of most relevance to us is the case $n=2$, the so called Gauss circle problem, and so we shall elaborate more. Gauss gave the first result $\varkappa_{2}\geq1$. This lower bound has been improved many times over. $\varkappa_{2}\geq\frac{4}{3}$ Vorono\"{i} \cite{voronoi1904fonction} and Sierpi\'{n}ski \cite{sierpinski1906pewnem} ; $\varkappa_{2}\geq\frac{67}{50}$ Van der Corput \cite{van1923neue} ; $\varkappa_{2}\geq\frac{580}{429}$ Kolesnik \cite{kolesnik1985method} ; $\varkappa_{2}\geq\frac{15}{11}$ Iwaniec and Mozzochi \cite{iwaniec1988divisor} ; $\varkappa_{2}\geq\frac{285}{208}$ Huxley \cite{huxley2003exponential} ; $\varkappa_{2}\geq\frac{1131}{824}$ Bourgain and Watt \cite{bourgain2017mean}. It is conjectured that $\varkappa_{2}=\frac{3}{2}$, which is supported by the second moment estimate \cite{lau2009mean}:
\begin{equation}\label{eq:1.1}
\begin{split}
&\frac{1}{X}\bigintssss\limits_{ X}^{2X}\Big(\mathcal{E}^{\mathbb{E}}_{2}(x)\Big)^{2}\textit{d}\sim C_{2}X\qquad\textit{as }X\to\infty\quad;\qquad C_{2}=\frac{3}{4\pi^{2}}\sum_{m=1}^{\infty}\frac{r^{2}_{2}(m)}{m^{3/2}}\,.
\end{split}
\end{equation}
Moreover, $\mathcal{E}^{\mathbb{E}}_{2}(x)$ can be abnormally large at times \cite{soundararajan2003omega}:
\begin{equation}\label{eq:1.2}
\mathcal{E}^{\mathbb{E}}_{2}(x)=\Omega\Big(x^{1/2}\big(\log{x}\big)^{1/4}F(x)\Big)\qquad;\qquad F(x)=\big(\log{\log{x}}\big)^{\frac{3}{4}(2^{1/3}-1)}\big(\log{\log{\log{x}}}\big)^{-5/8}\,.
\end{equation}
We now proceed to describe the lattice point counting problem that we shall consider in the present paper, where we shall obtain analogues results for the three type of estimates listed above. A pointwise bound corresponding to that of Vorono\"{i} and Sierpi\'{n}ski ; A second moment estimate with a power saving corresponding to \eqref{eq:1.1} : And an $\Omega$-estimate analogues to \eqref{eq:1.2} with $F(x)$ replaced by $\big(\log{\log{x}}\big)^{1/8}$.
\subsection{The lattice point counting problem on the Heisenberg groups}
Let $q\geq1$ be an integer. We endow the space $\mathbb{R}^{2q+1}\equiv\mathbb{R}^{2q}\times\mathbb{R}$ with the following homogeneous structure. The group law is defined by:
\begin{equation*}
\begin{split}
&\big(v,w\big)\ast\big(v',w'\big)=\big(v+v',w+w'+2\langle\text{J}v,v'\rangle\big)\quad;\quad\text{J}=\begin{pmatrix}0&I_{q}\\-I_{q}& 0\end{pmatrix}
\end{split}
\end{equation*}
where $\langle\,,\, \rangle$ stands for standard inner product on $\mathbb{R}^{q}$. One can check that the identity element is $0\in\mathbb{R}^{2q+1}$, and that $\big(v,w\big)^{-1}=\big(-v,-w\big)$. We shall write $\mathds{H}_{q}=\big(\mathbb{R}^{2q+1},\ast\big)$, and refer to this group as the $q$-th Heisenberg group. Note that $\mathbb{H}_{q}$ is a 2-step nilpotent group with a 1-dimensional center:
\begin{equation*}
\big[\big(v,w\big),\big(v',w'\big)\big]=\big(0,0,4\langle\text{J}v,v'\rangle\big)\,.
\end{equation*}
The Heisenberg dilations are given by:
\begin{equation*}
\delta_{x}\big(v,w\big)=\big(xv,x^{2}w\big)\quad;\quad x\in\mathbb{R}_{+}\,.
\end{equation*}
It is easily checked that:
\begin{equation*}
\begin{split}
&\delta_{x}\big((v,w)\ast(v',w')\big)=\delta_{x}\big(v,w\big)\ast\delta_{x}\big(v',w'\big)\\
&\delta_{x'}\big(\delta_{x}(v,w)\big)=\delta_{x'x}\big(v,w\big)
\end{split}
\end{equation*}
hence $\big\{  \delta_{x}:x\in\mathbb{R}_{+}\big\}$ forms a group of automorphisms of $\mathbb{H}_{q}$. We shall refer to this group as the dilation group. Define the family of Heisenberg norms:
\begin{equation*}
\mathcal{N}_{\alpha}\big(v,w\big)=\Big(|v|_{2}^{\alpha}+|w|^{\alpha/2}\Big)^{1/\alpha}\quad;\quad\alpha>0\,,\quad|\cdot|_{2}=\text{Euclidean norm}
\end{equation*}
and set:
\begin{equation*}
\mathcal{B}^{\alpha}=\big\{(v,w)\in\mathbb{R}^{2q}\times\mathbb{R}:\mathcal{N}_{\alpha}(v,w)\leq1\big\}\,.
\end{equation*}
Note that
\begin{equation*}
\mathcal{N}_{\alpha}\circ\delta_{x}=x\mathcal{N}_{\alpha}
\end{equation*}
Hence:
\begin{equation*}
\big|\mathbb{Z}^{2q+1}\cap\delta_{x}\mathcal{B}^{\alpha}\big|=\big|\big\{(z,z')\in\mathbb{Z}^{2q}\times\mathbb{Z}:\mathcal{N}_{\alpha}(z,z')\leq x\big\}\big|\,.
\end{equation*}
\text{ }\\
This natural family includes the \textbf{canonical} Cygan-Kor{\'a}nyi norm, corresponding to $\alpha=4$. This norm was considered
by Cygan \cite{cygan1978wiener}, and Kor{\'a}nyi \cite{koranyi1985geometric}. Cygan \cite{cygan1981subadditivity} has shown that this norm is sub-additive, in the sense that:
\begin{equation*}
\mathcal{N}_{4}\big((v,w)\ast(v',w')\big)\leq \mathcal{N}_{4}(v,w)+\mathcal{N}_{4}(v',w')\,.
\end{equation*}
In fact, $\mathcal{N}_{\alpha}$ is sub-additive iff $\alpha\geq4$ which was proved by Popa \cite{popa2016heisenberg}. Consequently, $\mathcal{N}_{4}$ defines a left invariant homogeneous distance on $\mathbb{H}_{q}$. In addition, the Cygan-Kor{\'a}nyi norm appears in the expression defining the fundamental
solution of a natural sublaplacian on $\mathbb{H}_{q}$ and in other natural kernels, see \cite{stein2016harmonic} and \cite{cowling2010unitary}. It is for this canonical norm that we shall consider the lattice point counting problem on the $q$-th Heisenberg group, which is analogous to the problem in the Abelian case where one considers the canonical Euclidean norm.\\\\
For notational simplicity, we shall drop the subscript and write $\mathcal{N}=\mathcal{N}_{4}$, $\mathcal{B}=\mathcal{B}^{4}$. As we are going to consider this counting problem on $\mathbb{H}_{q}$, we shall indicate the dependence with respect to the parameter $q$.
\begin{mydef}
Let $q\geq1$ be an integer, $x>0$. Define:
\begin{equation}\label{eq:1.3}
\mathcal{E}_{q}(x)=\big|\mathbb{Z}^{2q+1}\cap\delta_{x}\mathcal{B}\big|-\textit{vol}\big(\mathcal{B}\big)x^{2q+2}
\end{equation}
and set:
\begin{equation}\label{eq:1.4}
\kappa_{q}=\sup\big\{\alpha>0:\big|\mathcal{E}_{q}(x)\big|\ll x^{2q+2-\alpha}\big\}
\end{equation}
\text{ }\\
where $\mathcal{B}=\big\{(v,w)\in\mathbb{R}^{2q}\times\mathbb{R}:\mathcal{N}(v,w)\leq1\big\}$, $\mathcal{N}$ denotes the Cygan-Kor{\'a}nyi norm and $\textit{vol}(\cdot)$ is the $(2q+1)$-dimensional volume.
\end{mydef}
\text{ }\\
Let us remark that unlike the problem for Euclidean balls, the Gaussian curvature of the enclosing surface $\partial\mathcal{B}$ vanishes at both the points of intersection of $\mathcal{B}$ with the $w$-axis, namely the north and south poles. In fact, all of the $2q$ principal curvatures vanish at these two points. This inherent difficulty has been dealt with by Garg, Nevo \& Taylor \cite{garg2015lattice}, in which various upper bound estimates have been established for \eqref{eq:1.3} in the case of general Heisenberg norm balls. Amongst their many results, they were able to establish the lower bound $\kappa_{q}\geq2$ for all $q\geq1$. More precisely, they prove:
\begin{thmm}[Garg, Nevo \& Taylor]
Let $q\geq1$ be an integer. Then:  
\begin{equation*}
\begin{split}
\big|\mathcal{E}_{q}(x)\big|\ll\left\{
        \begin{array}{ll}
            x^{2}\log{x}\text{ }&;\text{ }q=1\\\\
             x^{4}(\log{x})^{2/3}\text{ }&;\text{ }q=2\\\\
             x^{2q}&;\text{ }q\geq3\,.
        \end{array}
    \right.
\end{split}
\end{equation*}
\end{thmm}
\text{ }\\
The case of $q=1$ is particularly interesting, as the author was able to match the lower bound for $\kappa_{1}$ with the corresponding upper bound $\kappa_{1}\leq2$. In fact, we obtained the following more precise result \cite{gath2017best}:
\begin{thmm}[Gath]
Let $\mathcal{E}_{1}(x)$ be defined as above. Then:
\begin{equation*}
\limsup_{x\to\infty}\frac{\mathcal{E}_{1}(x)}{x^{2}}=\infty\qquad;\qquad\liminf_{x\to\infty}\frac{\mathcal{E}_{1}(x)}{x^{2}}=-\infty
\end{equation*}
and in particular:
\begin{equation*}
\kappa_{1}=\sup\big\{\alpha>0:\big|\mathcal{E}_{1}(x)\big|\ll x^{4-\alpha}\big\}=2\,.
\end{equation*}
\end{thmm}
\text{ }\\
Before we proceed to present our results, we mention one last point of similarity (and difference) between the lattice point counting problem for Euclidean balls $\uplambda_{x}\mathscr{O}$ in $\mathbb{E}_{n}$, and the one for Cygan-Kor{\'a}nyi norm balls $\delta_{x}\mathcal{B}$ in $\mathbb{H}_{q}$. Recall that in the Euclidean one has $\varkappa_{n}=2$ as soon as $n\geq4$, and it is conjectured that this is also true for $n=3$. For $n=2$ the conjecture is $\varkappa_{2}=\frac{3}{2}$, and thus we see a distinct behavior of $\varkappa_{n}$ depending on the dimension: $n=2$ or $n\geq3$. Furthermore, estimating $\mathcal{E}^{\mathbb{E}}_{n}(x)$ in mean-square reveals a finer distinction between $n=3$ and $n\geq4$. We shall encounter this exact dimension-dependence behavior for $\mathbb{H}_{q}$.\\ Indeed, Theorem 2 stated below leads to the conjectural value $\kappa_{q}=3$ for $q\geq3$. This should also be the conjectural value of $\kappa_{2}$, where the corresponding results will appear in a separate paper as this case requires a different treatment. Moreover, as in the Euclidean case, estimating $\mathcal{E}_{q}(x)$ in mean-square we find the same finer distinction between $q=2$ and $q\geq3$. Finally, we have $\kappa_{1}=2$ unconditionally. The point of difference is that once $q>1$, the lattice point counting problem on $\mathbb{H}_{q}$ becomes intractable.
\subsection{Statement of results}
\begin{thm}
Let $q\geq3$ be integer. Then:
\begin{equation}\label{eq:1.5}
\begin{split}
\big|\mathcal{E}_{q}(x)\big|\ll x^{2q-2/3}\,.
\end{split}
\end{equation}
\end{thm}
\text{ }\\
Consequently, one has the lower bound  $\kappa_{q}\geq\frac{8}{3}=2.666...$ for all $q\geq3$ which is a significant improvement to the one obtained by Garg, Nevo \& Taylor stated in $\S$1.2.\\
\textbf{Remark.} By using the enhanced version of the Bombieri-Iwaniec method due to Huxley \cite{huxley2003exponential}, we can improve this lower bound further to $\kappa_{q}\geq\frac{285}{104}=2.74038...$.\\\\
Our next result gives support to the conjecture that $\kappa_{q}=3$ for $q\geq3$.
\begin{thm}
Let $q\geq3$ be integer. Then:\\\\
\textit{(1)} For $q\equiv0\,(2)$
\begin{equation}\label{eq:1.6}
\begin{split}
\frac{1}{X}\bigintssss\limits_{ X}^{2X}\mathcal{E}^{2}_{q}(x)\textit{d}x=
\gamma_{q}\Bigg\{\,\,\underset{\,\,(d,2m)=1}{\sum_{d,m=1}^{\infty}}\frac{r^{2}_{2}\big(m,d;q\big)}{m^{3/2}d^{2q-3}}+2^{2q}\underset{d\equiv0(4)}{\underset{\,\,(d,m)=1}{\sum_{d,m=1}^{\infty}}}\frac{r^{2}_{2}\big(m,d;q\big)}{m^{3/2}d^{2q-3}}\Bigg\}X^{2(2q-1)}+O\Big( X^{2(2q-1)-1}\log^{2}{X}\Big)
\end{split}
\end{equation}
\textit{(2)} For $q\equiv1\,(2)$
\begin{equation}\label{eq:1.7}
\begin{split}
\frac{1}{X}\bigintssss\limits_{ X}^{2X}\mathcal{E}^{2}_{q}(x)\textit{d}x=\gamma_{q}\Bigg\{\,\,\underset{\,\,(d,2m)=1}{\sum_{d,m=1}^{\infty}}\frac{r^{2}_{2}\big(m,d;q\big)}{m^{3/2}d^{2q-3}}+2^{2q}\underset{d\equiv0(4)}{\underset{\,\,(d,m)=1}{\sum_{d,m=1}^{\infty}}}\frac{r^{2}_{2,\chi}\big(m,d;q\big)}{m^{3/2}d^{2q-3}}\Bigg\}X^{2(2q-1)}+O\Big( X^{2(2q-1)-1}\log^{2}{X}\Big)
\end{split}
\end{equation}
where
\begin{equation*}
\begin{split}
r_{2}\big(m,d;q\big)=\underset{b\equiv0(d)}{\sum_{a^{2}+b^{2}=m}}\bigg(\frac{|a|}{\sqrt{m}}\bigg)^{q-1}\qquad;\qquad r_{2,\chi}\big(m,d;q\big)=\underset{b\equiv0(d)}{\sum_{a^{2}+b^{2}=m}}\chi\big(|a|\big)\bigg(\frac{|a|}{\sqrt{m}}\bigg)^{q-1}
\end{split}
\end{equation*}
and $\gamma_{q}=\frac{c_{q}}{2}\Big(\frac{\pi^{q-1}}{2\Gamma(q)}\Big)^{2}$. The constant $c_{q}=\frac{2^{4q-1}-1}{4q-1}$ arises from integrating the function $x^{2(2q-1)}$ along the dyadic interval $[X,2X]$. Here $\chi$ denotes the non-trivial Dirichlet character (mod 4).
\end{thm}
\text{ }\\
\textbf{Remark.} It is possible to show that the bounds for the error term in \eqref{eq:1.6} and \eqref{eq:1.7} are sharp, where the only possible improvement is with respect to the exponent appearing in the logarithmic factor. For example, one may easily reduce the $\log^{2}{X}$ factor to $\log{X}$, and with additional effort this can be further reduced to $\log^{\alpha}{X}$ with some $0<\alpha<1$, as soon as $q\geq4$. The proof that these bounds are indeed sharp is a subject for a different paper.\\\\
The above estimates show that $\mathcal{E}_{q}(x)$ has order of magnitude $x^{2q-1}$ in mean-square. Our third and final result establishes the existence of an unbounded sequences of ”exceptional” $x$-values, for which $\mathcal{E}_{q}(x)$ can be abnormally large. More precisely, we prove: 
\begin{thm}
Let $q\geq3$ be integer. Then there are arbitrarily large values of $x$ for which:
\begin{equation}\label{eq:1.8}
\big|\mathcal{E}_{q}(x)\big|\geq\beta_{q}\, x^{2q-1}\big(\log{x}\big)^{1/4}\big(\log{\log{x}}\big)^{1/8}
\end{equation}
where $\beta_{q}>0$ is some constant depending on $q$.
\end{thm}
\text{ }\\
\textbf{Remark.} It would be interesting to know whether the factor $\big(\log{\log{x}}\big)^{1/8}$ in \eqref{eq:1.8} could be strengthened.
Unfortunately, the methods in \cite{soundararajan2003omega} leading to $F(x)$ in \eqref{eq:1.2} rely on a crucial positivity argument which does not hold in our case.
\subsection{Organization of the paper}
At this point, it is worth mentioning that our approach to the lattice point counting problem on $\mathbb{H}_{q}$ is conceptually different from the one in \cite{garg2015lattice}. Specifically, Garg, Nevo \& Taylor's approach is to dominate the lattice point count in $\delta_{x}\mathcal{B}$ from above and below by convolving (in the Euclidean sense) the characteristic function $\chi_{\mathcal{B}}$ of $\mathcal{B}$ against a certain bump function $\rho_{\epsilon}$, where $\epsilon>0$ is chosen in a way so as to optimize the end results. A key point in their approach is the fact that $\rho_{\epsilon}$ is defined using Heisenberg dilations. An application of the Euclidean Poisson summation formula is now imminent, and this would necessitate establishing spectral decay estimates for the Euclidean Fourier transform of $\chi_{\mathcal{B}}$. Now, as was pointed out in $\S$1.2, $\partial\mathcal{B}$ contains points of vanishing Gaussian curvature, rendering the above estimates much harder.\\
The spectral analysis lies at the heart of the work of Garg, Nevo \& Taylor, with the resulting decay estimates leading to the lower bound $\kappa_{q}\geq2$. While for $q=1$, where this lower bound was shown by the author to be tight, that is $\kappa_{1}=2$, this is no longer the case when $q>1$. This does not mean that the above approach can not yield stronger results in the higher dimensional case, where in fact, the author believes it most certainly can.\\
However, it would be preferable not to perform the lattice point count at once, i.e Poisson summation formula, but rather more gradually, by splitting (or in our case slicing) $\delta_{x}\mathcal{B}$, and then estimate the number of lattice point in each region using different analytic arguments. This is precisely what we shall do. While an application of Poisson summation formula quickly gives an approximate expression for $\mathcal{E}_{q}(x)$, in our case, it will require considerably more work, and the required approximate expression only emerges at the very end of third section of this paper $\S$3.4. As we shall see, this hard work will be well rewarded. The paper is organized as follows:\\\\
$\S$2. In this section we obtain an initial expression for $\mathcal{E}_{q}(x)$, which is achieved by employing a certain slicing argument for the lattice point count in $\delta_{x}\mathcal{B}$. This will require us to establish several results regarding weighted lattice points in Euclidean balls and lattice points in shrinking annuli. The latter case is treated in a straightforward manner by appealing to the Euler–Maclaurin summation formula, where the resulting error term is given in an adequate form. As we intend on proving Theorem 2 in its sharpest possible form, the estimates we shall obtain in the former case will be derived from Lemma 2.1, where a vast arsenal of tolls from analytic number theory will be used. Collecting the results, we arrive at this so called initial expression. At this point, we could have already presented a proof of Theorem 1, however, we have chosen to postpone it to a later stage.\\\\
$\S$3. We subject $\mathcal{E}_{q}(x)$ to a transformation process, whose end result is the desired approximate expression mentioned above. The process begins with an application of Vaaler's Lemma, which enables us to approximate the $\psi$-sums obtained in Proposition 2.3 by a certain type of trigonometric sums. In turn, these trigonometric sums will be transformed using a sharp form of the \textit{B}-process of Van der Corput due to Karatsuba and Korolev. We complete the transformation process in $\S$3.3, where the $\psi$-sums are estimated in several different ranges. Gathering the results, we obtain two approximate expressions for $\mathcal{E}_{q}(x)$ stated in $\S$3.4, which will be the starting point from which we shall embark upon the proofs of Theorem 2 and Theorem 3. We end this section with a proof of Theorem 1.\\\\
$\S$4 - $\S$5. The last two sections are devoted to the proofs of Theorem 2 and Theorem 3 respectively, where the reader may find a detailed description for the course of proof.
\begin{ack}
I would like to express my deepest gratitude to my advisor Prof. Amos Nevo for his support and guidance throughout the writing of this paper. %
\end{ack}
\subsection{Notation}
$\star$ Throughout this paper, $q\geq3$ is an arbitrary \textbf{fixed integer} .\\\\
The following notation will occur repeatedly in this paper. Note that for some of the notations below we have indicated their dependence on $q$, while for others we have chosen to suppress it.
\begin{equation*}
\begin{split}
&\RomanNumeralCaps{1}.\quad\mathfrak{f}(t)=\sqrt{1-t^{2}}\quad,\quad \mathfrak{g}(t)=\big(1-t^{2}\big)^{\frac{q-1}{2}}\quad,\quad \hat{\mathfrak{g}}(t)=t^{q-1}\qquad;\qquad t\in[0,1]\,.\\\\
&\RomanNumeralCaps{2}.\quad\psi(t)=t-[t]-1/2\quad,\quad\textit{\large{e}}\big(t\big)=\textit{\large{e}}^{2\pi it}\quad,\quad\parallel t \parallel=\text{\,min\,}\{|t-m|:m\in\mathbb{Z}\big\}\quad;\quad t\in\mathbb{R}\,.\\\\
&\RomanNumeralCaps{3}.\quad\Re(s)=\sigma\quad,\quad\Im(s)=t\quad;\quad s=\sigma+it\,\,,\,\,s,t\in\mathbb{R}\,.\\\\
&\RomanNumeralCaps{4}.\quad r_{k}(m)=\Big|\Big\{\big(a_{1},\ldots,a_{k}\big)\in\mathbb{Z}^{k}:a^{2}_{1}+\cdots+ a^{2}_{k}=m\Big\}\Big|\quad,\quad\varpi(m)=\sum_{d|m}1\quad;\quad m\in\mathbb{N},\,k\geq2\\\\
&\RomanNumeralCaps{5}.\quad r_{2}\big(m,d;q\big)=\underset{b\equiv0(d)}{\sum_{a^{2}+b^{2}=m}}\hat{\mathfrak{g}}\bigg(\frac{|a|}{\sqrt{m}}\bigg)\quad,\quad r_{2,\chi}\big(m,d;q\big)=\underset{b\equiv0(d)}{\sum_{a^{2}+b^{2}=m}}\chi\big(|a|\big)\hat{\mathfrak{g}}\bigg(\frac{|a|}{\sqrt{m}}\bigg)\quad;\quad m,d\in\mathbb{N},\, a,b\in\mathbb{Z}\,.\\\\
&\RomanNumeralCaps{6}.\quad\xi(d)=\mathds{1}_{d\equiv1(2)}+(-1)^{\frac{q}{2}+1}\mathds{1}_{d\equiv0(2)}+(-1)^{\frac{q}{2}}2^{q}\mathds{1}_{d\equiv0(4)}\quad;\quad d\in\mathbb{N}\,.\\\\
&\RomanNumeralCaps{7}.\quad\varrho_{q}=\frac{\pi^{q}}{(1-2^{-q})\Gamma(q)\zeta(q)}\quad,\quad\varrho_{\chi,q}=\frac{\pi^{q}}{2^{q-1}\Gamma(q)L(q,\chi)}\quad;\quad\zeta(s)=\sum_{n=1}^{\infty}\frac{1}{n^{s}}\,\,,\,\,L(s,\chi)=\sum_{n=1}^{\infty}\frac{\chi(n)}{n^{s}}\,\,;\,\,\Re(s)>1\,.\\\\
&\RomanNumeralCaps{8}.\quad\chi=\text{\,the non-trivial Dirichlet character (mod 4)}\,.\\\\
&\RomanNumeralCaps{9}.\quad m=\square\,\,,\,\, m\neq\square\Longleftrightarrow m \text{ is or is not equal to a square}\quad;\quad m\in\mathbb{N}\,.
\end{split}
\end{equation*}
\newpage
\section{Extracting the main term and an initial expression for $\mathcal{E}_{q}(x)$}
The method we shall use to count the number of lattice points in $\delta_{x}\mathcal{B}$, will be by slicing it with hyperplanes, and counting the number of lattice points in each sliced section separately. To do so, we shall first need to establish several results regarding weighted lattice points in Euclidean balls and lattice points in shrinking annuli. This will be done in the first two subsections, and the relevant results are given by Proposition 2.1 \& 2.2. In the third subsection we gather the results to obtain an initial expression for $\mathcal{E}_{q}(x)$.     
\subsection{Weighted lattice points in Euclidean balls}
We begin this subsection by establishing the following lemma, which will then be combined with Corollary 2.1 to prove Proposition 2.1. 
\begin{lem}
For $Y>0$ define:
\begin{equation*}
\mathscr{S}_{q}(Y)=\sum_{0\,\leq\,m\,\leq\,Y}r_{2q}(m)\bigg(1-\frac{m}{Y}\bigg)
\end{equation*}
and set
\begin{equation*}
\mathscr{E}_{q}(Y)=\mathscr{S}_{q}(Y)-\frac{\pi^{q}}{\Gamma(q+2)}\,Y^{q}\,.
\end{equation*}
Then:
\begin{equation}\label{eq:2.1}
\big|\mathscr{E}_{q}(Y)\big|\ll Y^{q-2}\log{Y}\,.
\end{equation}
\end{lem}
\begin{proof}
Let $Y$ be large, $\sqrt{2}\pi Y^{\frac{1}{2}}\leq T\leq Y$ a parameter to be chosen later, and set $\delta=\frac{1}{\log{Y}}$. Write $\upphi$ for the continuous function on $\mathbb{R}_{\text{}_{+}}$ defined by $\upphi(y)=1-y$ if $y\in[0,1]$, and $\upphi(y)=0$ otherwise. We have: 
\begin{equation}\label{eq:2.2}
\begin{split}
\mathscr{S}_{q}(Y)&=1+\sum_{m=1}^{\infty}r_{2q}(m)\upphi\bigg(\frac{m}{Y}\bigg)=1+\frac{1}{2\pi i}\bigintssss\limits_{q+\delta-i\infty}^{q+\delta+i\infty}\zeta_{\text{}_{2q}}(s)\breve{\upphi}(s)Y^{s}\textit{d}s=\\
&=1+\frac{1}{2\pi i}\bigintssss\limits_{q+\delta-iT}^{q+\delta+iT}\zeta_{\text{}_{2q}}(s)\breve{\upphi}(s)Y^{s}\textit{d}s+O\bigg(Y^{q}T^{-2}\sum_{m=1}^{\infty}\frac{r_{2q}(m)}{m^{q+\delta}}\bigg(1+\text{min}\bigg\{T,\frac{1}{|\log{\frac{Y}{m}}|}\bigg\}\bigg)\bigg)=\\
&=1+\frac{1}{2\pi i}\bigintssss\limits_{q+\delta-iT}^{q+\delta+iT}\zeta_{\text{}_{2q}}(s)\breve{\upphi}(s)Y^{s}\textit{d}s+O\bigg(Y^{q}T^{-2}\log{Y}\bigg)
\end{split}
\end{equation}
where $\breve{\upphi}(s)=\frac{\Gamma(s)}{\Gamma(s+2)}$ is the Mellin transform of $\upphi$, and:
\begin{equation*}
\zeta_{\text{}_{2q}}(s)=\sum_{m=1}^{\infty}\frac{r_{2q}(m)}{m^{s}}\qquad;\qquad\Re(s)>q\,.
\end{equation*}
The Zeta function $\zeta_{\text{}_{2q}}(s)$, initially defined for $\Re(s)>q$, admits an analytic continuation to the entire complex plane except at $s=q$ where it has a simple pole with residue $\frac{\pi^{q}}{\Gamma(q)}$, and satisfies the functional equation:
\begin{equation*}
\pi^{-s}\Gamma(s)\zeta_{\text{}_{2q}}(s)=\pi^{-(q-s)}\Gamma(q-s)\zeta_{\text{}_{2q}}(q-s)\,.
\end{equation*}
Now, $s(s-q)\zeta_{\text{}_{2q}}(s)\breve{\upphi}(s)Y^{s}$ is regular in the strip $-\delta\leq\Re(s)\leq q+\delta$, and by Stirling's asymptotic formula for the Gamma function (see \cite{iwaniec2004analytic}, A.4 (5.112))
\begin{equation}\label{eq:2.3}
\Gamma(s)=\bigg(\frac{2\pi}{s}\bigg)^{\frac{1}{2}}\bigg(\frac{s}{e}\bigg)^{s}\bigg(1+O_{\epsilon}\bigg(\frac{1}{|s|}\bigg)\bigg)\qquad;\qquad|\text{Arg}(s)|\leq\pi-\epsilon
\end{equation}
together with the functional equation, we obtain the bounds: 
\begin{equation}\label{eq:2.4}
\begin{split}
&\big|s(s-q)\zeta_{\text{}_{2q}}(s)\breve{\upphi}(s)Y^{s}\big|\ll\,Y^{q}\log{Y}\qquad\qquad\quad\,\,;\qquad\Re(s)=q+\delta\\\\
&\big|s(s-q)\zeta_{\text{}_{2q}}(s)\breve{\upphi}(s)Y^{s}\big|\ll\,|1+s|^{q+2\delta}\log{Y}\qquad;\qquad\Re(s)=-\delta\,.
\end{split}
\end{equation}
Hence, by the Phragm{\'e}n-{L}indel{\"o}f principle we deduce:
\begin{equation}\label{eq:2.5}
\big|\zeta_{\text{}_{2q}}(s)\breve{\upphi}(s)Y^{s}\big|\ll\,Y^{q}T^{-2}\log{Y}\qquad;\qquad-\delta\leq\Re(s)\leq q+\delta \text{ },\text{ }|\Im(s)|=T\,.
\end{equation}
Moving the line of integration to $\Re(s)=-\delta$, and using \eqref{eq:2.5}, we have by the theorem of residues
\begin{equation}\label{eq:2.6}
\begin{split}
\frac{1}{2\pi i}\bigintssss\limits_{q+\delta-iT}^{q+\delta+iT}\zeta_{\text{}_{2q}}(s)\breve{\upphi}(s)Y^{s}\textit{d}s&=\bigg\{\underset{s=q}{\text{Res}}+\underset{s=0}{\text{Res}}\bigg\}\zeta_{\text{}_{2q}}(s)\breve{\upphi}(s)Y^{s}+\frac{1}{2\pi i}\bigintssss\limits_{-\delta-iT}^{-\delta+iT}\zeta_{\text{}_{2q}}(s)\breve{\upphi}(s)Y^{s}\textit{d}s+O\bigg(Y^{q}T^{-2}\log{Y}\bigg)=\\
&=\frac{\pi^{q}}{\Gamma(q+2)}\,Y^{q}-1+\frac{1}{2\pi i}\bigintssss\limits_{-\delta-iT}^{-\delta+iT}\zeta_{\text{}_{2q}}(s)\breve{\upphi}(s)Y^{s}\textit{d}s+O\bigg(Y^{q}T^{-2}\log{Y}\bigg)\,.
\end{split}
\end{equation}
Inserting \eqref{eq:2.6} into the the RHS of \eqref{eq:2.2}, and applying the functional equation, we arrive at:
\begin{equation}\label{eq:2.7}
\mathscr{E}_{q}(Y)=\sum_{m=1}^{\infty}\frac{r_{2q}(m)}{(\pi m)^{q}}\text{\Large{J}}_{m}+O\Big(Y^{q}T^{-2}\log{Y}\Big)
\end{equation}
where
\begin{equation}\label{eq:2.8}
\text{\Large{J}}_{m}=\frac{1}{2\pi i}\bigintssss\limits_{-\delta-iT}^{-\delta+iT}\frac{\Gamma(q-s)(\pi^{2}mY)^{s}}{\Gamma(s+2)}\textit{d}s\,.
\end{equation}
We shall estimate \eqref{eq:2.8} separately for $m>M$ and $m\leq M$, with $M=\pi^{-2}T^{2}Y^{-1}\geq2$. Suppose $m>M$. By Stirling's asymptotic formula \eqref{eq:2.3} :
\begin{equation}\label{eq:2.9}
\big|\text{\Large{J}}_{m}\big|\ll\,\frac{1}{m^{\delta}}\bigg(\,\bigg|\bigintssss\limits_{1}^{T}g_{\text{}_{2\delta}}(t)e^{if_{m}(t)}\textit{d}t\bigg|+O\big(T^{q-2}\big)\bigg)
\end{equation}
where $\,\,g_{\text{}_{2\delta}}(t)=t^{q-2+2\delta\,\,}
$ and $\,\,f_{m}(t)=-2t\log{t}+2t +t\log{\big(\pi^{2}mY\big)}$. Trivial integration and integration by parts give
\begin{equation*}
\bigg|\bigintssss\limits_{1}^{T}g_{\text{}_{2\delta}}(t)e^{if_{m}(t)}\textit{d}t\bigg|\ll\,T^{q-2}\text{min}\bigg\{T,\frac{1}{\log{\frac{m}{M}}}\bigg\}
\end{equation*}
hence we obtain:
\begin{equation}\label{eq:2.10}
\big|\text{\Large{J}}_{m}\big|\ll\,\frac{T^{q-2}}{m^{\delta}}\bigg(1+\text{min}\bigg\{T,\frac{1}{\log{\frac{m}{M}}}\bigg\}\bigg).
\end{equation}
Suppose now that $m\leq M$. Set $\sigma=q-\delta$. Appealing to \eqref{eq:2.3}, we move the line of integration to $\Re(s)=\sigma$ and then extend the integral all the way from $-\infty$ to $\infty$, obtaining:
\begin{equation}\label{eq:2.11}
\begin{split}
\text{\Large{J}}_{m}&=\frac{1}{2\pi i}\bigintssss\limits_{ \sigma-iT}^{\sigma+iT}\frac{\Gamma(q-s)(\pi^{2}mY)^{s}}{\Gamma(s+2)}\textit{d}s+O\bigg(T^{q-2}\bigintssss\limits_{-\delta}^{\sigma}\bigg(\frac{m}{M}\bigg)^{\alpha}\textit{d}\alpha\bigg)=\\
&=\frac{1}{2\pi i}\bigintssss\limits_{ \sigma-i\infty}^{\sigma+i\infty}\frac{\Gamma(q-s)(\pi^{2}mY)^{s}}{\Gamma(s+2)}\textit{d}s+O\bigg(\big(\pi^{2}mY\big)^{\sigma}\bigg|\bigintssss\limits_{T}^{\infty}g_{\text{}_{-2\sigma}}(t)e^{if_{m}(t)}\textit{d}t\bigg|+T^{q-2}\bigg)
\end{split}
\end{equation}
with $g_{\text{}_{-2\sigma}}(t)=t^{q-2-2\sigma}$, and $f_{m}(t)$ is defined as before.
Trivial integration and integration by parts give
\begin{equation*}
\bigg|\bigintssss\limits_{T}^{\infty}g_{\text{}_{-2\sigma}}(t)e^{if_{m}(t)}\textit{d}t\bigg|\ll\,T^{q-2-2\sigma}\text{min}\bigg\{T,\frac{1}{\log{\frac{M}{m}}}\bigg\}
\end{equation*}
and since $\big(\pi^{2}mY\big)^{\sigma}T^{q-2-2\sigma}=\big(\frac{m}{M}\big)^{\sigma}T^{q-2}\leq T^{q-2}$, we obtain:
\begin{equation}\label{eq:2.12}
\begin{split}
\big(\pi^{2}mY\big)^{\sigma}\bigg|\bigintssss\limits_{T}^{\infty}g_{\text{}_{-2\sigma}}(t)e^{if_{m}(t)}\textit{d}t\bigg|+T^{q-2}&\ll\,T^{q-2}\bigg(1+\text{min}\bigg\{T,\frac{1}{\log{\frac{M}{m}}}\bigg\}\bigg)\ll M^{\delta}\frac{T^{q-2}}{m^{\delta}}\bigg(1+\text{min}\bigg\{T,\frac{1}{\log{\frac{M}{m}}}\bigg\}\bigg)\ll\\
&\ll\frac{T^{q-2}}{m^{\delta}}\bigg(1+\text{min}\bigg\{T,\frac{1}{\log{\frac{M}{m}}}\bigg\}\bigg)
\end{split}
\end{equation}
where that last $\ll$ follows since $M\ll Y$. Summing over all $m\in\mathbb{N}$, we get by \eqref{eq:2.7}, \eqref{eq:2.10} and \eqref{eq:2.12}:\\
\begin{equation}\label{eq:2.13}
\begin{split}
\mathscr{E}_{q}(Y)&=\sum_{1\,\leq\,m\,\leq\,M}^{}\frac{r_{2q}(m)}{(\pi m)^{q}}\frac{1}{2\pi i}\bigintssss\limits_{ \sigma-i\infty}^{\sigma+i\infty}\frac{\Gamma(q-s)(\pi^{2}mY)^{s}}{\Gamma(s+2)}\textit{d}s+O\bigg(T^{q-2}\sum_{m=1}^{\infty}\frac{r_{2q}(m)}{m^{q+\delta}}\bigg(1+\text{min}\bigg\{T,\frac{1}{|\log{\frac{M}{m}}|}\bigg\}\bigg)\bigg)+\\
&+O\Big(Y^{q}T^{-2}\log{Y}\Big)=\sum_{1\,\leq\,m\,\leq\,M}^{}\frac{r_{2q}(m)}{(\pi m)^{q}}\frac{1}{2\pi i}\bigintssss\limits_{ \sigma-i\infty}^{\sigma+i\infty}\frac{\Gamma(q-s)(\pi^{2}mY)^{s}}{\Gamma(s+2)}\textit{d}s+O\Big(Y^{q}T^{-2}\log{Y}\Big)=\\
&=\pi^{-1}Y^{\frac{q-1}{2}}\sum_{1\,\leq\,m\,\leq\,M}^{}\frac{r_{2q}(m)}{m^{\frac{q+1}{2}}}\mathcal{J}_{q+1}\big(2\pi\sqrt{mY}\,\big)+\Big(Y^{q}T^{-2}\log{Y}\Big)
\end{split}
\end{equation}
where for $\nu>0$ the Bessel function $\mathcal{J}_{\nu}$ of order $\nu$ is defined by
\begin{equation*}
\mathcal{J}_{\nu}(y)=\sum_{k=0}^{\infty}\frac{(-1)^{k}}{k!\Gamma(k+1+\nu)}\bigg(\frac{y}{2}\bigg)^{\nu+2k}\,.
\end{equation*}
To proceed further, we need an asymptotic estimate for the Bessel function (see \cite{iwaniec2002spectral}, B.4 (B.35)). For fixed $\nu>0$ :
\begin{equation}\label{eq:2.14}
\mathcal{J}_{\nu}(y)=\bigg(\frac{2}{\pi y}\bigg)^{1/2}\cos{\bigg(y-\frac{1}{2}\nu\pi-\frac{1}{4}\pi\bigg)}+O\bigg(\frac{1}{y^{3/2}}\bigg)\,,\text{ as } y\to\infty\,.
\end{equation}
For $m\in\mathbb{N}$ we have (see \cite{rankin_1977}, 7.4):
\begin{equation}\label{eq:2.15}
\begin{split}
&\text{ }\qquad\qquad\quad\,\,\, r_{2q}(m)=\rho_{2q}(m)+\tau_{2q}(m)\\\\
&q\equiv0\,(2)\,:\quad\rho_{2q}(m)=\varrho_{\text{}_{q}}m^{q-1}\bigg\{\underset{d\equiv1\,(2)}{\sum_{d|m}}d^{1-q}+(-1)^{\frac{q}{2}}\underset{d\equiv0\,(2)}{\sum_{d|m}}(-1)^{\frac{m}{d}}d^{1-q}\bigg\}\\
&q\equiv1\,(2)\,:\quad\rho_{2q}(m)=\varrho_{\text{}_{\chi,q}}m^{q-1}\bigg\{2^{q-1}\sum_{d|m}\chi(d)d^{1-q}+(-1)^{\frac{q-1}{2}}\sum_{d|m}\chi\bigg(\frac{m}{d}\bigg)d^{1-q}\bigg\}
\end{split}
\end{equation}
where $\tau_{2q}(m)=0$ for $q=3,4$. For $q\geq5$ we have $|\tau_{2q}(m)|\ll m^{\frac{q-1}{2}}\varpi(m)$ which is a consequence of Deligne's proof of the Riemann Hypothesis for varieties over finite fields \cite{deligne1974conjecture}. Using \eqref{eq:2.14} and \eqref{eq:2.15}, splitting the summation over $m$ into dyadic segments and unfolding $\rho_{2q}(m)$ we arrive at:
\begin{equation}\label{eq:2.16}
\begin{split}
\big|\mathscr{E}_{q}(Y)\big|\ll\,Y^{\frac{q-1}{2}-\frac{1}{4}}N^{\frac{q}{2}-\frac{7}{4}}\log{Y}\sum_{a\,(4)}\,\,\sum_{1\,\leq\,d\,\leq\,2N}\frac{1}{d^{q-1}}W_{a,d}+Y^{q}T^{-2}\log{Y}
\end{split}
\end{equation}
for some $1\leq N\leq M$, where 
\begin{equation}\label{eq:2.17}
\begin{split}
W_{a,d}&=\underset{\frac{N}{d}\,\leq\,y\,<\,y'\,\leq\,2\frac{N}{d}}{\sup}\,\,\bigg|\sum_{y\,<\,m\,\leq\,y'}\textit{\large{e}}\Big(f_{a,d}(m)\Big)\bigg|
\end{split}
\end{equation}
and $f_{a,d}(t)=\sqrt{dY}t^{\frac{1}{2}}+\frac{a}{4}t$. To estimate \eqref{eq:2.17} we appeal to (see \cite{graham1991van}, 2.1 (Thm 2.2)):
\begin{VDCL}
Let $\mathcal{I}\subset\mathbb{R}$ be a finite interval of length $\big|\mathcal{I}\big|\geq1$, g a real valued function with two continuous
derivatives on $\mathcal{I}$ satisfying:
\begin{equation*}
\lambda\leq|g^{(2)}(t)|\leq C\lambda\qquad;\qquad t\in\mathcal{I},\,\lambda>0 \text{ and } C\geq1\,.
\end{equation*}
Then:
\begin{equation}\label{eq:2.18}
\bigg|\sum_{m\in\mathcal{I}}\textit{\large{e}}\big(g(m)\big)\bigg|\ll_{\text{}_{C}}\,\big|\mathcal{I}\big|\lambda^{\frac{1}{2}}+\lambda^{-\frac{1}{2}}\,.
\end{equation}
\end{VDCL}
\text{ }\\
Now given $d\in\mathbb{N}$, $a\in\mathbb{Z}$ and $\frac{N}{d}\leq y<y'\leq2\frac{N}{d}$, we apply \eqref{eq:2.18} if $y'-y\geq1$ and estimate trivially otherwise, obtaining:
\begin{equation}\label{eq:2.19}
\big|W_{a,d}\big|\ll\,1+Y^{\frac{1}{4}}N^{\frac{1}{4}}+\frac{N^{\frac{3}{4}}}{Y^{\frac{1}{4}}d}\ll\,Y^{\frac{1}{4}}N^{\frac{1}{4}}\,.
\end{equation}
Inserting \eqref{eq:2.19} into the RHS of \eqref{eq:2.16}, and recalling that $N\ll M\ll T^{2}Y^{-1}$ we derive:
\begin{equation}
\begin{split}
&\big|\mathscr{E}_{q}(Y)\big|\ll\,\Big(YT^{q-3}+Y^{q}T^{-2}\Big)\log{Y}\,.
\end{split}
\end{equation}
To balance the two error terms we make the choice $T=Y$, so $\sqrt{2}\pi Y^{\frac{1}{2}}\leq T\leq Y$, which gives:
\begin{equation}
\big|\mathscr{E}_{q}(Y)\big|\ll Y^{q-2}\log{Y}
\end{equation}
as claimed.
\end{proof}
\text{ }\\
Using Lemma 2.1 we deduce:
\begin{cor}
For $Y>0$ and $k\in\mathbb{N}$ define:
\begin{equation*}
\mathscr{S}^{\ast}_{k,q}(Y)=\sum_{0\,\leq\,m\,\leq\, \sqrt{Y}}r_{2q}(m)\bigg(1-\frac{m^{2}}{Y}\bigg)^{k}
\end{equation*}
and set
\begin{equation*}
\mathscr{E}^{\ast}_{k,q}(Y)=\mathscr{S}^{\ast}_{k,q}(Y)-\frac{k!\pi^{q}}{(\frac{q}{2}+1)\cdots(\frac{q}{2}+k)\Gamma(q+1)}Y^{\frac{q}{2}}\,.
\end{equation*}
Then:
\begin{equation}\label{eq:2.22}
\big|\mathscr{E}^{\ast}_{k,q}(Y)\big|\ll\,Y^{\frac{q}{2}-1}\log{Y}
\end{equation}
where the implied constant does not depend on $k$.
\end{cor}
\begin{proof}
Let $Y$ be large. We have for $2\leq k\in\mathbb{N}$:
\begin{equation*}
\begin{split}
\mathscr{S}^{\ast}_{k,q}(Y)=\frac{k!}{Y^{k}}\bigintssss\limits_{y_{\text{}_{k-1}}=0}^{Y}\,\,\bigintssss\limits_{y_{\text{}_{k-2}}=0}^{y_{\text{}_{k-1}}}\ldots\bigintssss\limits_{y_{\text{}_{1}}=0}^{y_{\text{}_{2}}}y_{\text{}_{1}}\mathscr{S}^{\ast}_{1,q}(y_{\text{}_{1}})\textit{d}y_{\text{}_{1}}\ldots\textit{d}y_{\text{}_{k-1}}
\end{split}
\end{equation*}
and with the notation as in Lemma 2.1
\begin{equation*}
Y\mathscr{S}^{\ast}_{1,q}(Y)=2\bigg\{Y\mathscr{S}_{q}\big(\sqrt{Y}\,\big)-\bigintssss\limits_{ y_{\text{}_{0}}=0}^{\sqrt{Y}}y_{\text{}_{0}}\mathscr{S}_{q}(y_{\text{}_{0}})\textit{d}y_{\text{}_{0}}\bigg\}\,.
\end{equation*}
Thus, for $2\leq k\in\mathbb{N}$:
\begin{equation}\label{eq:2.23}
\begin{split}
\mathscr{E}^{\ast}_{k,q}(Y)=\frac{k!}{Y^{k}}\bigintssss\limits_{y_{\text{}_{k-1}}=0}^{Y}\,\,\bigintssss\limits_{y_{\text{}_{k-2}}=0}^{y_{\text{}_{k-1}}}\ldots\bigintssss\limits_{y_{\text{}_{1}}=0}^{y_{\text{}_{2}}}y_{\text{}_{1}}\mathscr{E}^{\ast}_{1,q}(y_{\text{}_{1}})\textit{d}y_{\text{}_{1}}\ldots\textit{d}y_{\text{}_{k-1}}
\end{split}
\end{equation}
and with the notation as in Lemma 2.1
\begin{equation}\label{eq:2.24}
Y\mathscr{E}^{\ast}_{1,q}(Y)=2\bigg\{Y\mathscr{E}_{\text{}_{q}}\big(\sqrt{Y}\,\big)-\bigintssss\limits_{ y_{\text{}_{0}}=0}^{\sqrt{Y}}y_{\text{}_{0}}\mathscr{E}_{q}(y_{\text{}_{0}})\textit{d}y_{\text{}_{0}}\bigg\}\,.
\end{equation}
Hence, by \eqref{eq:2.23}, \eqref{eq:2.24} and \eqref{eq:2.1}, we have for $k\in\mathbb{N}$ :
\begin{equation}
\big|\mathscr{E}^{\ast}_{\text{}_{k,q}}(Y)\big|\leq3\underset{0\,<\,y\,<\,\sqrt{Y}}{\sup}\,\big|\mathscr{E}_{\text{}_{q}}(y)\big|\ll\,Y^{\frac{q}{2}-1}\log{Y}
\end{equation}
which proves the claim.
\end{proof}
\text{ }\\
Combining Lemma 2.1 and Corollary 2.1, we have the following proposition regarding the error term for the number of weighted lattice points in Euclidean balls. We do not need to have a concrete expression for this error term, all we need to know is its order of magnitude.
\begin{prop}
For $x>0$ define: 
\begin{equation*}
\mathscr{S}^{\flat}_{q}(x)=\sum_{0\,\leq\,m\,\leq\, \frac{x^{2}}{\sqrt{2}}}r_{2q}(m)\Big(\sqrt{x^{4}-m^{2}}-m\Big)
\end{equation*}
and set
\begin{equation*}
\mathscr{E}^{\flat}_{q}(x)=\mathscr{S}^{\flat}_{q}(x)-\alpha_{q,\flat}x^{2q+2}
\end{equation*}
where
\begin{equation*}
\alpha_{q,\flat}=\frac{\pi^{q}}{2^{\frac{q+1}{2}}\Gamma(q+1)}\sum_{k=1}^{\infty}\frac{f^{(k)}(1)}{(\frac{q}{2}+1)\cdots(\frac{q}{2}+k)}+\frac{\pi^{q}}{2^{\frac{q+1}{2}}\Gamma(q+2)}\quad;\quad f(y)=\sqrt{y}\,.
\end{equation*}
Then:
\begin{equation}\label{eq:2.26}
\big|\mathscr{E}^{\flat}_{q}(x)\big|\ll\,x^{2q-2}\log{x}\,.
\end{equation}
\end{prop}
\begin{proof}
Let $x$ be large. For $0\leq m\leq \frac{x^{2}}{\sqrt{2}}$ an integer we have
\begin{equation*}
\sqrt{x^{4}-m^{2}}-m=\frac{x^{2}}{\sqrt{2}}\bigg(\sqrt{2-\frac{m^{2}}{x^{4}/2}}-1\bigg)
+\frac{x^{2}}{\sqrt{2}}\bigg(1-\frac{m}{x^{2}/\sqrt{2}}\bigg)
\end{equation*}
and by Taylor expansion with $f(y)=\sqrt{y}$
\begin{equation*}
\sqrt{2-\frac{m^{2}}{x^{4}/2}}-1=\sum_{k=1}^{\infty}\frac{f^{(k)}(1)}{k!}\bigg(1-\frac{m^{2}}{x^{4}/2}\bigg)^{k}
\end{equation*}
we obtain:
\begin{equation}\label{eq:2.27}
\sqrt{x^{4}-m^{2}}-m=\frac{x^{2}}{\sqrt{2}}\,\sum_{k=1}^{\infty}\frac{f^{(k)}(1)}{k!}\bigg(1-\frac{m^{2}}{x^{4}/2}\bigg)^{k}+\frac{x^{2}}{\sqrt{2}}\bigg(1-\frac{m}{x^{2}/\sqrt{2}}\bigg)\,.
\end{equation}
Multiplying \eqref{eq:2.27} by $r_{2q}(m)$ and Summing over all $0\leq m\leq x^{2}/\sqrt{2}$, we have with the notations as in Lemma 2.1 and Corollary 2.1
\begin{equation*}
\begin{split}
&\mathscr{S}^{\flat}_{q}(x)=\frac{x^{2}}{\sqrt{2}}\,\sum_{k=1}^{\infty}\frac{f^{(k)}(1)}{k!}\,\mathscr{S}^{\ast}_{k,q}\bigg(\frac{x^{4}}{2}\bigg)+\frac{x^{2}}{\sqrt{2}}\,\mathscr{S}_{q}\bigg(\frac{x^{2}}{\sqrt{2}}\bigg)=\alpha_{q,\flat}x^{2q+2}+\frac{x^{2}}{\sqrt{2}}\,\sum_{k=1}^{\infty}\frac{f^{(k)}(1)}{k!}\,\mathscr{E}^{\ast}_{k,q}\bigg(\frac{x^{4}}{2}\bigg)+\frac{x^{2}}{\sqrt{2}}\,\mathscr{E}_{q}\bigg(\frac{x^{2}}{\sqrt{2}}\bigg)
\end{split}
\end{equation*}
hence:
\begin{equation}
\mathscr{E}^{\flat}_{q}(x)=\frac{x^{2}}{\sqrt{2}}\,\sum_{k=1}^{\infty}\frac{f^{(k)}(1)}{k!}\,\mathscr{E}^{\ast}_{k,q}\bigg(\frac{x^{4}}{2}\bigg)+\frac{x^{2}}{\sqrt{2}}\,\mathscr{E}_{q}\bigg(\frac{x^{2}}{\sqrt{2}}\bigg)\,.
\end{equation}
By \eqref{eq:2.1} and \eqref{eq:2.22} we obtain:
\begin{equation}
\begin{split}
&\bigg|\frac{x^{2}}{\sqrt{2}}\,\sum_{k=1}^{\infty}\frac{f^{(k)}(1)}{k!}\,\mathscr{E}^{\ast}_{k,q}\bigg(\frac{x^{4}}{2}\bigg)+\frac{x^{2}}{\sqrt{2}}\,\mathscr{E}_{q}\bigg(\frac{x^{2}}{\sqrt{2}}\bigg)\bigg|\ll\,\bigg\{\sum_{k=1}^{\infty}\frac{|f^{(k)}(1)|}{k!}+1\bigg\}\,x^{2q-2}\log{x}\ll\,x^{2q-2}\log{x}
\end{split}
\end{equation}
as claimed.
\end{proof}
\subsection{Lattice points in shrinking annuli}
We now turn to the problem of counting the number of lattice points in shrinking annuli. Unlike the previous subsection, in which the shape of the error term was not of our concern (only its order of magnitude was relevant), here this is no longer the case. We shall need to have the error term in an explicit form.
\begin{mydef}
Let $\varphi:\mathbb{R}\longrightarrow\mathbb{C}$ be a 1-periodic function. For $f,g:[0,1]\longrightarrow\mathbb{R}$ and $Y>0$ define:
\begin{equation*}
\begin{split}
&\textit{(I)}\quad\quad\text{\large{S}}^{\,g}_{\varphi}(Y;f)=\sum_{0\,<\,n\,\leq\, \frac{Y}{\sqrt{2}}}g\bigg(\frac{n}{Y}\bigg)\varphi\bigg(Yf\bigg(\frac{n}{Y}\bigg)\bigg)\\
&\textit{(II)}\,\,\,\quad\text{\large{S}}^{\,g}_{\varphi,\chi}(Y;f)=\sum_{0\,<\,n\,\leq\, \frac{Y}{\sqrt{2}}}g\bigg(\frac{n}{Y}\bigg)\sum_{a\,(4)}\chi(a)\varphi\bigg(Yf\bigg(\frac{n}{Y}\bigg)+\frac{a}{4}\bigg)\\
&\textit{(III)}\quad\text{\large{S}}^{\,g,\chi}_{\varphi}(Y;f)=\sum_{0\,<\,n\,\leq\, \frac{Y}{\sqrt{2}}}g\bigg(\frac{n}{Y}\bigg)\chi(n)\varphi\bigg(Yf\bigg(\frac{n}{Y}\bigg)\bigg)\,.
\end{split}
\end{equation*}
\end{mydef}
\text{ }\\
\textbf{Remark.} When $\varphi(t)=\textit{\large{e}}(t)$, we shall drop the subscript $\varphi$ and write $\text{\large{S}}^{\,g}_{\varphi}(Y;f)=\text{\large{S}}^{\,g}(Y;f)$, $\text{\large{S}}^{\,g}_{\varphi,\chi}(Y;f)=\text{\large{S}}^{\,g}_{\chi}(Y;f)$ and $\text{\large{S}}^{\,g,\chi}_{\varphi}(Y;f)=\text{\large{S}}^{\,g,\chi}(Y;f)$.
\begin{prop}
For $x>0$ define:
\begin{equation*}
\mathscr{S}^{\sharp}_{q}(x)=\sum_{0\,<\,m\,\leq\, \frac{x^{2}}{\sqrt{2}}}r_{2q}(m)+\sum_{0\,\leq\,|n|\,\leq\, \frac{x^{2}}{\sqrt{2}}}\,\,\sum_{|n|\,<\,m\,\leq\,\sqrt{x^{4}-n^{2}}}r_{2q}(m)
\end{equation*}
and set
\begin{equation*}
\mathscr{E}^{\sharp}_{q}(x)=\mathscr{S}^{\sharp}_{q}(x)-\alpha_{q,\sharp}x^{2q+2}\qquad;\qquad\alpha_{q,\sharp}=\frac{2\pi^{q}}{\Gamma(q+1)}\bigintssss\limits_{0}^{\frac{1}{\sqrt{2}}}\Big(\big(1-t^{2}\big)^{\frac{q}{2}}-t^{q}\Big)\textit{d}t\,.
\end{equation*}
Then:\\\\
(1) For $q\equiv0\,(2)$
\begin{equation}\label{eq:2.30}
\mathscr{E}^{\sharp}_{q}(x)=-2\varrho_{q}x^{2q-2}\sum_{1\,\leq\,d\,\leq\,\frac{x^{2}}{\sqrt{2}}}\frac{\xi(d)}{d^{q-1}}\,\text{\large{S}}^{\,\mathfrak{g}}_{\psi}\Big(x^{2}\,;\,\frac{1}{d}\mathfrak{f}\Big)+O\Big(x^{2q-2}\log{x}\Big)\,.
\end{equation}
(2) For $q\equiv1\,(2)$
\begin{equation}\label{eq:2.31}
\begin{split}
\mathscr{E}^{\sharp}_{q}(x)=&-2\varrho_{\chi,q}x^{2q-2}\sum_{1\,\leq\,d\,\leq\,\frac{x^{2}}{\sqrt{2}}}\frac{1}{d^{q-1}}\bigg\{2^{q-1}\chi(d)\text{\large{S}}^{\,\mathfrak{g}}_{\psi}\Big(x^{2}\,;\,\frac{1}{d}\mathfrak{f}\Big)+(-1)^{\frac{q+1}{2}}\text{\large{S}}^{\,\mathfrak{g}}_{\psi,\chi}\Big(x^{2}\,;\,\frac{1}{4d}\mathfrak{f}\Big)\bigg\}+\\
&+O\Big(x^{2q-2}\log{x}\Big)\,.
\end{split}
\end{equation}
\end{prop}
\begin{proof}
Let $x$ be large.\\\\
\textit{(1)} Suppose $q\equiv0\,(2)$. Let $n\in\mathbb{Z}$ be such that $|n|\,\leq\, \frac{x^{2}}{\sqrt{2}}$. Using \eqref{eq:2.15} we have:
\begin{equation}\label{eq:2.32}
\sum_{|n|\,<\,m\,\leq\,\sqrt{x^{4}-n^{2}}}r_{2q}(m)=\varrho_{q}\sum_{1\,\leq\,d\,\leq\,\frac{x^{2}}{\sqrt{2}}}\xi(d)\sum_{\frac{|n|}{d}\,<\,m\,\leq\,\frac{\sqrt{x^{4}-n^{2}}}{d}}m^{q-1}+O\Big(x^{2}+\mathds{1}_{\text{}_{q\geq6}}x^{q+1}\log{x}\Big)\,.
\end{equation}
Fix $1\leq d\leq \frac{x^{2}}{\sqrt{2}}$. Applying the Euler–Maclaurin summation formula we obtain:
\begin{equation}\label{eq:2.33}
\sum_{\frac{|n|}{d}\,<\,m\,\leq\,\frac{\sqrt{x^{4}-n^{2}}}{d}}m^{q-1}=\frac{1}{qd^{q}}\mathscr{F}\big(|n|\big)+\frac{1}{d^{q-1}}\mathscr{L}_{d}\big(|n|\big)+O\bigg(\frac{x^{2q-4}}{d^{q-2}}\bigg)
\end{equation}
where 
\begin{equation*}
\mathscr{L}_{d}(y)=y^{q-1}\psi\bigg(\frac{y}{d}\bigg)-x^{2q-2}\,\mathfrak{g}\bigg(\frac{y}{x^{2}}\bigg)\psi\bigg(\frac{1}{d}x^{2}\mathfrak{f}\bigg(\frac{y}{x^{2}}\bigg)\bigg)
\end{equation*}
and $\mathscr{F}(y)=\bigg(x^{4}-y^{2}\bigg)^{\frac{q}{2}}-y^{q}$.\\\\
Inserting \eqref{eq:2.33} into the RHS of \eqref{eq:2.32}, and using the estimate
\begin{equation}\label{eq:2.34}
\begin{split}
q^{-1}\varrho_{q}\sum_{1\,\leq\,d\,\leq\,\frac{x^{2}}{\sqrt{2}}}\frac{\xi(d)}{d^{q}}=\frac{\pi^{q}}{\Gamma(q+1)}+O\Big(x^{-2q+2}\Big)
\end{split}
\end{equation}
we derive:
\begin{equation}\label{eq:2.35}
\begin{split}
\sum_{|n|\,<\,m\,\leq\,\sqrt{x^{4}-n^{2}}}r_{2q}(m)=\frac{\pi^{q}}{\Gamma(q+1)}\mathscr{F}\big(|n|\big)+\varrho_{q}\sum_{1\,\leq\,d\,\leq\,\frac{x^{2}}{\sqrt{2}}}\frac{\xi(d)}{d^{q-1}}\mathscr{L}_{d}\big(|n|\big)+O\Big(x^{2q-4}\log{x}\Big)\,.
\end{split}
\end{equation}
Summing \eqref{eq:2.35} over all $|n|\,\leq\, \frac{x^{2}}{\sqrt{2}}$, and using the estimate
\begin{equation}\label{eq:2.36}
\sum_{0\,\leq\,|n|\,\leq\, \frac{x^{2}}{\sqrt{2}}}\mathscr{F}\big(|n|\big)=2\bigintssss\limits_{0}^{\frac{1}{\sqrt{2}}}\Big(\big(1-t^{2}\big)^{\frac{q}{2}}-t^{q}\Big)\textit{d}t\,x^{2q+2}+O\Big(x^{2q-2}\Big)
\end{equation}
we arrive at:
\begin{equation}\label{eq:2.37}
\mathscr{S}^{\sharp}_{q}(x)=\alpha_{q,\sharp}x^{2q+2}-2\varrho_{q}x^{2q-2}\sum_{1\,\leq\,d\,\leq\,\frac{x^{2}}{\sqrt{2}}}\frac{\xi(d)}{d^{q-1}}\,\text{\large{S}}^{\,\mathfrak{g}}_{\psi}\Big(x^{2}\,;\,\frac{1}{d}\mathfrak{f}\Big)+E+O\Big(x^{2q-2}\log{x}\Big)
\end{equation}
where
\begin{equation*}
E=\sum_{0\,<\,m\,\leq\, \frac{x^{2}}{\sqrt{2}}}r_{2q}(m)
+\varrho_{q}\sum_{1\,\leq\,d\,\leq\,\frac{x^{2}}{\sqrt{2}}}\frac{\xi(d)}{d^{q-1}}\sum_{0\,\leq\,|n|\,\leq\, \frac{x^{2}}{\sqrt{2}}}|n|^{q-1}\psi\bigg(\frac{|n|}{d}\bigg)\,.
\end{equation*}
It remains to show that $|E|\ll\,x^{2q-2}\log{x}$. \\\\
Let $1\leq d\leq \frac{x^{2}}{\sqrt{2}}$ be an integer. Splitting into residue classes $(\text{mod }d)$, and using the identity $\sum\limits_{b\,(d)}\psi\big(\frac{b}{d}\big)=-\frac{1}{2}$, we have:
\begin{equation}\label{eq:2.38}
\sum_{0\,\leq\,|n|\,\leq\, \frac{x^{2}}{\sqrt{2}}}|n|^{q-1}\psi\bigg(\frac{|n|}{d}\bigg)=2\sum_{b\,(d)}\psi\bigg(\frac{b}{d}\bigg)\underset{n\equiv b\,(d)}{\sum_{0\,<\,n\,\leq\, \frac{x^{2}}{\sqrt{2}}}}n^{q-1}=-\frac{1}{qd}\bigg(\frac{x^{2}}{\sqrt{2}}\bigg)^{q}+O\Big(dx^{2q-2}\Big)\,.
\end{equation}
Multiplying \eqref{eq:2.38} by $\varrho_{q}d^{1-q}\xi(d)$, summing over all $1\leq d\leq \frac{x^{2}}{\sqrt{2}}$ and using \eqref{eq:2.34}, we obtain:
\begin{equation}\label{eq:2.39}
E=\sum_{0\,<\,m\,\leq\, \frac{x^{2}}{\sqrt{2}}}r_{2q}(m)-\frac{\pi^{q}}{\Gamma(q+1)}\bigg(\frac{x^{2}}{\sqrt{2}}\bigg)^{q}+O\Big(x^{2q-2}\log{x}\Big)\,.
\end{equation}
Finally, using the well known estimate (see \cite{iwaniec2004analytic}, 1.5 (1.76)):
\begin{equation}\label{eq:2.40}
\sum_{0\,<\,m\,\leq\, Y}r_{k}(m)-\frac{\pi^{\frac{k}{2}}}{\Gamma\big(\frac{k}{2}+1\big)}Y^{\frac{k}{2}}=O\Big(Y^{\frac{k}{2}-1}\log{Y}\Big)\qquad;\qquad k\geq4
\end{equation}
we deduce that $|E|\ll\,x^{2q-2}\log{x}$, which proves \eqref{eq:2.30}.\\\\
\textit{(2)} Suppoe $q\equiv1\,(2)$. Let $n\in\mathbb{Z}$ be such that $|n|\,\leq\, \frac{x^{2}}{\sqrt{2}}$. Using \eqref{eq:2.15} we have:
\begin{equation}\label{eq:2.41}
\begin{split}
\sum_{|n|\,<\,m\,\leq\,\sqrt{x^{4}-n^{2}}}r_{2q}(m)&=2^{q-1}\varrho_{\chi,q}\sum_{1\,\leq\,d\,\leq\,\frac{x^{2}}{\sqrt{2}}}\chi(d)\sum_{\frac{|n|}{d}\,<\,m\,\leq\,\frac{\sqrt{x^{4}-n^{2}}}{d}}m^{q-1}+\\
&+(-1)^{\frac{q-1}{2}}\varrho_{\chi,q}\sum_{1\,\leq\,d\,\leq\,\frac{x^{2}}{\sqrt{2}}}\,\,\sum_{\frac{|n|}{d}\,<\,m\,\leq\,\frac{\sqrt{x^{4}-n^{2}}}{d}}\chi(m) m^{q-1}+O\Big(x^{2}+\mathds{1}_{\text{}_{q\geq5}}x^{q+1}\log{x}\Big)\,.
\end{split}
\end{equation}
Fix $1\leq d\leq \frac{x^{2}}{\sqrt{2}}$. Applying the Euler–Maclaurin summation formula we obtain:
\begin{equation}\label{eq:2.42}
\sum_{\frac{|n|}{d}\,<\,m\,\leq\,\frac{\sqrt{x^{4}-n^{2}}}{d}}m^{q-1}=\frac{1}{qd^{q}}\mathscr{F}\big(|n|\big)+\frac{1}{d^{q-1}}\mathscr{L}_{d}\big(|n|\big)+O\bigg(\frac{x^{2q-4}}{d^{q-2}}\bigg)
\end{equation}
where $\mathscr{F}(y)$ and $\mathscr{L}_{d}(y)$ are defined as before, and:
\begin{equation}\label{eq:2.43}
\begin{split}
\sum_{\frac{|n|}{d}\,<\,m\,\leq\,\frac{\sqrt{x^{4}-n^{2}}}{d}}\chi(m) m^{q-1}=\frac{1}{d^{q-1}}\mathscr{L}_{4d,\chi}\big(|n|\big)+O\bigg(\frac{x^{2q-4}}{d^{q-2}}\bigg)
\end{split}
\end{equation}
where 
\begin{equation*}
\mathscr{L}_{4d,\chi}(y)=y^{q-1}\sum_{a\,(4)}\chi(a)\psi\bigg(\frac{y}{4d}-\frac{a}{4}\bigg)+x^{2q-2}\,\mathfrak{g}\bigg(\frac{y}{x^{2}}\bigg)\sum_{a\,(4)}\chi(a)\psi\bigg(\frac{1}{4d}x^{2}\mathfrak{f}\bigg(\frac{y}{x^{2}}\bigg)+\frac{a}{4}\bigg)\,.
\end{equation*}
Substituting \eqref{eq:2.42} and \eqref{eq:2.43} into the RHS of \eqref{eq:2.41}, and using the estimate 
\begin{equation}\label{eq:2.44}
2^{q-1}\varrho_{\chi,q}\sum_{1\,\leq\,d\,\leq\,\frac{x^{2}}{\sqrt{2}}}\frac{\chi(d)}{d^{q}}=\frac{\pi^{q}}{\Gamma(q+1)}+O\Big(x^{-2q+2}\Big)
\end{equation}
we derive:
\begin{equation}\label{eq:2.45}
\begin{split}
\sum_{|n|\,<\,m\,\leq\,\sqrt{x^{4}-n^{2}}}r_{2q}(m)&=\frac{\pi^{q}}{\Gamma(q+1)}\mathscr{F}(n)+2^{q-1}\varrho_{\chi,q}\sum_{1\,\leq\,d\,\leq\,\frac{x^{2}}{\sqrt{2}}}\frac{\chi(d)}{d^{q-1}}\mathscr{L}_{d}\big(|n|\big)+\\
&+(-1)^{\frac{q-1}{2}}\varrho_{\chi,q}\sum_{1\,\leq\,d\,\leq\,\frac{x^{2}}{\sqrt{2}}}\frac{1}{d^{q-1}}\mathscr{L}_{4d,\chi}\big(|n|\big)
+O\Big(x^{2q-4}\log{x}\Big)\,.
\end{split}{}
\end{equation}
Summing \eqref{eq:2.45} over all $|n|\,\leq\, \frac{x^{2}}{\sqrt{2}}$, we have by \eqref{eq:2.36}, \eqref{eq:2.38}, \eqref{eq:2.44} and \eqref{eq:2.40} :
\begin{equation}\label{eq:2.46}
\begin{split}
\mathscr{S}^{\sharp}_{q}(x)&=-2\varrho_{\chi,q}x^{2q-2}\sum_{1\,\leq\,d\,\leq\,\frac{x^{2}}{\sqrt{2}}}\frac{1}{d^{q-1}}\bigg\{2^{q-1}\chi(d)\text{\large{S}}^{\,\mathfrak{g}}_{\psi}\Big(x^{2}\,;\,\frac{1}{d}\mathfrak{f}\Big)+(-1)^{\frac{q+1}{2}}\text{\large{S}}^{\,\mathfrak{g}}_{\psi,\chi}\Big(x^{2}\,;\,\frac{1}{4d}\mathfrak{f}\Big)\bigg\}
+E_{\chi}+O\Big(x^{2q-2}\log{x}\Big)
\end{split}
\end{equation}
where
\begin{equation*}
E_{\chi}=(-1)^{\frac{q-1}{2}}\varrho_{\chi,q}\sum_{1\,\leq\,d\,\leq\,\frac{x^{2}}{\sqrt{2}}}\frac{1}{d^{q-1}}\sum_{0\,\leq\,|n|\,\leq\, \frac{x^{2}}{\sqrt{2}}}|n|^{q-1}\sum_{a\,(4)}\chi(a)\psi\bigg(\frac{|n|}{4d}-\frac{a}{4}\bigg).
\end{equation*}
It remains to show that $|E_{\chi}|\ll\,x^{2q-2}\log{x}$. \\\\
Let $1\leq d\leq \frac{x^{2}}{\sqrt{2}}$ be an integer. Splitting into residue classes $(\text{mod }4d)$, and using the identity $\sum\limits_{b\,(4d)}\psi\big(\frac{b}{4d}-\frac{a}{4}\big)=\sum\limits_{b\,(4d)}\psi\big(\frac{b}{4d}\big)=-\frac{1}{2}$ which is valid for any $a\in\mathbb{Z}$, we obtain:
\begin{equation}\label{eq:2.47}
\begin{split}
\sum_{0\,\leq\,|n|\,\leq\, \frac{x^{2}}{\sqrt{2}}}|n|^{q-1}\sum_{a\,(4)}\chi(a)\psi\bigg(\frac{|n|}{4d}-\frac{a}{4}\bigg)&=2\sum_{a\,(4)}\chi(a)\sum_{b\,(4d)}\psi\bigg(\frac{b}{4d}-\frac{a}{4}\bigg)\underset{n\equiv b\,(4d)}{\sum_{0\,<\,n\,\leq\, \frac{x^{2}}{\sqrt{2}}}}n^{q-1}=\\
&=-\frac{1}{4qd}\bigg(\frac{x^{2}}{\sqrt{2}}\bigg)^{q}\sum_{a\,(4)}\chi(a)+O\Big(dx^{2q-2}\Big)=O\bigg(dx^{2q-2}\bigg)\,.
\end{split}
\end{equation}
Multiplying \eqref{eq:2.47} by $d^{1-q}$, and summing over all $1\leq d\leq \frac{x^{2}}{\sqrt{2}}$, we deduce that $|E_{\chi}|\ll\,x^{2q-2}\log{x}$, which proves \eqref{eq:2.31}. 
\end{proof}
\subsection{The error term $\mathcal{E}_{q}(x)$ in its initial form}
Collecting the results from subsections $\S$2.1 and $\S$2.2, we can now extract the main term in the asymptotic estimate for the number of lattice points contained inside Heisenberg dilates of the Cygan-Kor{\'a}nyi norm ball, in which the error term is given in an adequate form.   
\begin{prop}
Let $x>0$ be large. Then:\\\\
(1) For $q\equiv0\,(2)$
\begin{equation}\label{eq:2.48}
\mathcal{E}_{q}(x)=-2\varrho_{q}x^{2q-2}\sum_{1\,\leq\,d\,\leq\,\frac{x^{2}}{\sqrt{2}}}\frac{\xi(d)}{d^{q-1}}\bigg\{\text{\large{S}}^{\,\mathfrak{g}}_{\psi}\Big(x^{2}\,;\,\frac{1}{d}\mathfrak{f}\Big)+\text{\large{S}}^{\,\hat{\mathfrak{g}}}_{\psi}\Big(\frac{1}{d}x^{2}\,;\,d\mathfrak{f}\Big)\bigg\}+O\Big(x^{2q-2}\log{x}\Big)
\end{equation}
(2) For $q\equiv1\,(2)$
\begin{equation}\label{eq:2.49}
\begin{split}
\mathcal{E}_{q}(x)=&-2\varrho_{\chi,q}x^{2q-2}\sum_{1\,\leq\,d\,\leq\,\frac{x^{2}}{\sqrt{2}}}\frac{2^{q-1}\chi(d)}{d^{q-1}}\bigg\{\text{\large{S}}^{\,\mathfrak{g}}_{\psi}\Big(x^{2}\,;\,\frac{1}{d}\mathfrak{f}\Big)+\text{\large{S}}^{\,\hat{\mathfrak{g}}}_{\psi}\Big(\frac{1}{d}x^{2}\,;\,d\mathfrak{f}\Big)\bigg\}+\\
&-2\varrho_{\chi,q}x^{2q-2}\sum_{1\,\leq\,d\,\leq\,\frac{x^{2}}{\sqrt{2}}}\frac{(-1)^{\frac{q+1}{2}}}{d^{q-1}}\bigg\{\text{\large{S}}^{\,\mathfrak{g}}_{\psi,\chi}\Big(x^{2}\,;\,\frac{1}{4d}\mathfrak{f}\Big)-\text{\large{S}}^{\,\hat{\mathfrak{g}},\chi}_{\psi}\Big(\frac{1}{d}x^{2}\,;\,d\mathfrak{f}\Big)\bigg\}+O\Big(x^{2q-2}\log{x}\Big)\,.
\end{split}
\end{equation}
\end{prop}
\begin{proof}
Let $x$ be large. By the definition of the Cygan-Kor{\'a}nyi norm, we have with the notations as in Proposition 2.1 and Proposition 2.2
\begin{equation}\label{eq:2.50}
\begin{split}
&\big|\mathbb{Z}^{2q+1}\cap\delta_{x}\mathcal{B}\big|=\sum_{0\,\leq\, m^{2}+n^{2}\,\leq\, x^{4}}r_{2q}(m)
=\sum_{0\,\leq\, m^{2}+n^{2}\,\leq\, x^{4}}r_{2q}(m)\bigg\{\mathds{1}_{m\,<\,|n|}+\mathds{1}_{m\,=\,|n|}+\mathds{1}_{m\,>\,|n|}\bigg\}=\\
&=\sum_{0\,\leq\,m\,\leq\, \frac{x^{2}}{\sqrt{2}}}r_{2q}(m)\bigg\{2\Big(\sqrt{x^{4}-m^{2}}-m\Big)-2\psi\Big(\sqrt{x^{4}-m^{2}}\,\Big)-1\bigg\}+1+2\sum_{0\,<\,m\,\leq\, \frac{x^{2}}{\sqrt{2}}}r_{2q}(m)+\\
&+\sum_{0\,\leq\,|n|\,\leq\, \frac{x^{2}}{\sqrt{2}}}\,\,\sum_{|n|\,<\,m\,\leq\,\sqrt{x^{4}-n^{2}}}r_{2q}(m)=
\alpha_{q}x^{2q+2}+\mathscr{E}^{\sharp}_{q}(x)-2\sum_{0\,\leq\,m\,\leq\, \frac{x^{2}}{\sqrt{2}}}r_{2q}(m)\psi\Big(\sqrt{x^{4}-m^{2}}\,\Big)+2\mathscr{E}^{\flat}_{q}(x)
\end{split}
\end{equation}
where $\alpha_{q}=2\alpha_{q,\flat}+\alpha_{q,\sharp}$. Using \eqref{eq:2.26}, we obtain:
\begin{equation}\label{eq:2.51}
\big|\mathbb{Z}^{2q+1}\cap\delta_{x}\mathcal{B}\big|-\alpha_{q}x^{2q+2}=\mathscr{E}^{\sharp}_{q}(x)-2\sum_{0\,\leq\,m\,\leq\, \frac{x^{2}}{\sqrt{2}}}r_{2q}(m)\psi\Big(\sqrt{x^{4}-m^{2}}\,\Big)+O\Big(x^{2q-2}\log{x}\Big)\,.
\end{equation}
By trivial estimation the RHS of \eqref{eq:2.51} is $\ll\,x^{2q}$, and since $\big|\mathbb{Z}^{2q+1}\cap\delta_{x}\mathcal{B}\big|\sim\textit{meas}\big(\mathcal{B}\big)x^{2q+2}$ as $x\to\infty$, it follows that $\alpha_{q}=\textit{meas}\big(\mathcal{B}\big)$. Appealing to \eqref{eq:2.15} one last time, we have:
\begin{equation*}
\sum_{0\,\leq\,m\,\leq\, \frac{x^{2}}{\sqrt{2}}}r_{2q}(m)\psi\Big(\sqrt{x^{4}-m^{2}}\,\Big)=\sum_{0\,<\,m\,\leq\, \frac{x^{2}}{\sqrt{2}}}\rho_{2q}(m)\psi\Big(\sqrt{x^{4}-m^{2}}\,\Big)+O\Big(1+\mathds{1}_{\text{}_{q\geq5}}x^{q+1}\log{x}\Big)\,.
\end{equation*}
Inserting this into the RHS of \eqref{eq:2.51}, we arrive at:
\begin{equation}
\begin{split}
\mathcal{E}_{q}(x)&=\mathscr{E}^{\sharp}_{q}(x)-2\sum_{0\,<\,n\,\leq\, \frac{x^{2}}{\sqrt{2}}}\rho_{2q}(n)\psi\Big(\sqrt{x^{4}-n^{2}}\,\Big)+O\Big(x^{2q-2}\log{x}+\mathds{1}_{\text{}_{q\geq5}}x^{q+1}\log{x}\Big)=\\
&=\mathscr{E}^{\sharp}_{q}(x)-2\sum_{0\,<\,n\,\leq\, \frac{x^{2}}{\sqrt{2}}}\rho_{2q}(n)\psi\Big(\sqrt{x^{4}-n^{2}}\,\Big)+O\Big(x^{2q-2}\log{x}\Big)\,.
\end{split}
\end{equation}
Finally, we have for $q\equiv0\,(2)$:
\begin{equation*}
\begin{split}
\sum_{0\,<\,m\,\leq\, \frac{x^{2}}{\sqrt{2}}}\rho_{2q}(m)\psi\Big(\sqrt{x^{4}-m^{2}}\,\Big)=\varrho_{q}x^{2q-2}\sum_{1\,\leq\,d\,\leq\,\frac{x^{2}}{\sqrt{2}}}\frac{\xi(d)}{d^{q-1}}\,\text{\large{S}}^{\,\hat{\mathfrak{g}}}_{\psi}\Big(\frac{1}{d}x^{2}\,;\,d\mathfrak{f}\Big)
\end{split}
\end{equation*}
and for $q\equiv1\,(2)$:
\begin{equation*}
\begin{split}
&\sum_{0\,<\,m\,\leq\, \frac{x^{2}}{\sqrt{2}}}\rho_{2q}(m)\psi\Big(\sqrt{x^{4}-m^{2}}\,\Big)=\\
&=\varrho_{\chi,q}x^{2q-2}\sum_{1\,\leq\,d\,\leq\,\frac{x^{2}}{\sqrt{2}}}\frac{1}{d^{q-1}}\bigg\{2^{q-1}\chi(d)\text{\large{S}}^{\,\hat{\mathfrak{g}}}_{\psi}\Big(\frac{1}{d}x^{2}\,;\,d\mathfrak{f}\Big)+(-1)^{\frac{q-1}{2}}\text{\large{S}}^{\,\hat{\mathfrak{g}},\chi}_{\psi}\Big(\frac{1}{d}x^{2}\,;\,d\mathfrak{f}\Big)\bigg\}
\end{split}
\end{equation*}
which combined with \eqref{eq:2.30} and \eqref{eq:2.31} concludes the proof.
\end{proof}
\section{Transformation of $\mathcal{E}_{q}(x)$ and proof of Theorem 1}
As it stands, the initial expression for $\mathcal{E}_{q}(x)$ obtained in Proposition 2.3 is not yet ready for applications, and needs to be subjected to a transformation process. At the end of this process, a new 
expression for $\mathcal{E}_{q}(x)$ will emerge which is well suited and flexible enough to meet our needs. This will be the subject of the current section, and the main results are stated in Proposition 3.1 \& 3.2. In subsection 3.5 we shall give a proof of Theorem 1.
\subsection{Transitioning from $\psi$-sums to exponential sums}
We now turn our attention to the sums involving the 1-periodic function $\psi$ appearing on the RHS of \eqref{eq:2.48} and \eqref{eq:2.49}, and we begin the  transformation process of $\mathcal{E}_{q}(x)$ by applying a suitable approximation to $\psi$ for which a proof can be found in \cite{vaaler1985some}.
\begin{VL}
Let $\mathcal{H}\geq1$, and define the trigonometrical polynomials:
\begin{equation*}
\begin{split}
&\textit{(1)}\qquad\psi_{\text{}_{\mathcal{H}}}(\omega)=\sum_{1\,\leq\,h\,\leq\,\mathcal{H}}\tau\bigg(\frac{h}{[\mathcal{H}]+1}\bigg)\frac{1}{h}\sin{(-2\pi h\omega)}\\
&\textit{(2)}\qquad\psi^{\ast}_{\text{}_{\mathcal{H}}}(\omega)=\sum_{1\,\leq\,h\,\leq\,\mathcal{H}}\tau^{\ast}\bigg(\frac{h}{[\mathcal{H}]+1}\bigg)\frac{1}{h}\cos{(-2\pi h\omega)}\,.
\end{split}
\end{equation*}
where:
\begin{equation*}
\begin{split}
&\tau(t)=t(1-t)\cot{(\pi t)}+\pi^{-1}t\qquad\,\,\,\,\,\,\,\,;\qquad0<t<1\\
&\tau^{\ast}(t)=t(1-t)\qquad\quad\qquad\qquad\qquad;\qquad0<t<1\,.
\end{split}
\end{equation*}
Then there holds the inequality:
\begin{equation*}
\big|\psi(\omega)-\psi_{\text{}_{\mathcal{H}}}(\omega)\big|\leq\psi^{\ast}_{\text{}_{\mathcal{H}}}(\omega)+\frac{1}{2[\mathcal{H}]+2}\,.
\end{equation*}
\end{VL}
\text{ }\\
Having Vaaler's Lemma, we can now transition from fractional part sums to exponential sums.\\
\begin{lem} 
Let $x>0$ large, $d\in\mathbb{N}$ satisfying $1\leq d\leq \frac{x^{2}}{\sqrt{2}}$. Then for any $\mathcal{H}\geq1$ :\\\\
\textit{(1)}
\begin{equation}\label{eq:3.1}
\text{\large{S}}^{\,\mathfrak{g}}_{\psi}\Big(x^{2}\,;\,\frac{1}{d}\mathfrak{f}\Big)=\sum_{1\,\leq\,h\,\leq\,\mathcal{H}}\tau\bigg(\frac{h}{[\mathcal{H}]+1}\bigg)\frac{1}{h}\Im\Bigg(\text{\large{S}}^{\,\mathfrak{g}}\Big(x^{2}\,;\,-\frac{h}{d}\mathfrak{f}\Big)\Bigg)+\mathscr{E}_{1}\big(\mathcal{H},d\big)
\end{equation}
where $\mathscr{E}_{1}\big(\mathcal{H},d\big)$ satisfies the bound:
\begin{equation}\label{eq:3.2}
\big|\mathscr{E}_{1}\big(\mathcal{H},d\big)\big|\leq\sum_{1\,\leq\,h\,\leq\,\mathcal{H}}\tau^{\ast}\bigg(\frac{h}{[\mathcal{H}]+1}\bigg)\frac{1}{h}\Re\Bigg(\text{\large{S}}^{\,\mathfrak{g}}\Big(x^{2}\,;\,-\frac{h}{d}\mathfrak{f}\Big)\Bigg)
+O\bigg(\frac{x^{2}}{\mathcal{H}}\bigg)\,.
\end{equation}
\textit{(2)}
\begin{equation}\label{eq:3.3}
\text{\large{S}}^{\,\mathfrak{g}}_{\psi,\chi}\Big(x^{2}\,;\,\frac{1}{4d}\mathfrak{f}\Big)=2\sum_{1\,\leq\,h\,\leq\,\mathcal{H}}\chi(-h)\tau\bigg(\frac{h}{[\mathcal{H}]+1}\bigg)\frac{1}{h}\Re\Bigg(\text{\large{S}}^{\,\mathfrak{g}}\Big(x^{2}\,;\,-\frac{h}{4d}\mathfrak{f}\Big)\Bigg)+\mathscr{E}_{2}\big(\mathcal{H},d\big)
\end{equation}
where $\mathscr{E}_{2}\big(\mathcal{H},d\big)$ satisfies the bound:
\begin{equation}\label{eq:3.4}
\big|\mathscr{E}_{2}\big(\mathcal{H},d\big)\big|\leq2\sum_{1\,\leq\,h\,\leq\,\mathcal{H}}\Big\{\mathds{1}_{h\equiv0(4)}-\mathds{1}_{h\equiv2(4)}\Big\}\tau^{\ast}\bigg(\frac{h}{[\mathcal{H}]+1}\bigg)\frac{1}{h}\Re\Bigg(\text{\large{S}}^{\,\mathfrak{g}}\Big(x^{2}\,;\,-\frac{h}{4d}\mathfrak{f}\Big)\Bigg)
+O\bigg(\frac{x^{2}}{\mathcal{H}}\bigg)\,.
\end{equation}
\textit{(3)}
\begin{equation}\label{eq:3.5}
\text{\large{S}}^{\,\hat{\mathfrak{g}}}_{\psi}\Big(\frac{1}{d}x^{2}\,;\,d\mathfrak{f}\Big)=\sum_{1\,\leq\,h\,\leq\,\mathcal{H}}\tau\bigg(\frac{h}{[\mathcal{H}]+1}\bigg)\frac{1}{h}\Im\Bigg(\text{\large{S}}^{\,\hat{\mathfrak{g}}}\Big(\frac{1}{d}x^{2}\,;\,-dh\mathfrak{f}\Big)\Bigg)+\mathscr{E}_{3}\big(\mathcal{H},d\big)
\end{equation}
where $\mathscr{E}_{3}\big(\mathcal{H},d\big)$ satisfies the bound:
\begin{equation}\label{eq:3.6}
\big|\mathscr{E}_{3}\big(\mathcal{H},d\big)\big|\leq\sum_{1\,\leq\,h\,\leq\,\mathcal{H}}\tau^{\ast}\bigg(\frac{h}{[\mathcal{H}]+1}\bigg)\frac{1}{h}\Re\Bigg(\text{\large{S}}^{\,\hat{\mathfrak{g}}}\Big(\frac{1}{d}x^{2}\,;\,-dh\mathfrak{f}\Big)\Bigg)
+O\bigg(\frac{x^{2}}{d\mathcal{H}}\bigg)\,.
\end{equation}
\textit{(4)}
\begin{equation}\label{eq:3.7}
-\text{\large{S}}^{\,\hat{\mathfrak{g}},\chi}_{\psi}\Big(\frac{1}{d}x^{2}\,;\,d\mathfrak{f}\Big)=\frac{1}{2}\sum_{1\,\leq\,h\,\leq\,\mathcal{H}}\tau\bigg(\frac{h}{[\mathcal{H}]+1}\bigg)\frac{1}{h}\Re\Bigg(\sum_{a\,(4)}\chi(-a)\text{\large{S}}^{\,\hat{\mathfrak{g}}}\Big(\frac{1}{d}x^{2}\,;\,-dh\mathfrak{f}\,-\,\frac{a}{4}\mathfrak{h}\Big)\Bigg)+\mathscr{E}_{4}\big(\mathcal{H},d\big)
\end{equation}
where $\mathscr{E}_{4}\big(\mathcal{H},d\big)$ satisfies the bound:
\begin{equation}\label{eq:3.8}
\bigg|\mathscr{E}_{4}\big(\mathcal{H},d\big)\bigg|\leq\sum_{1\,\leq\,h\,\leq\,\mathcal{H}}\tau^{\ast}\bigg(\frac{h}{[\mathcal{H}]+1}\bigg)\frac{1}{h}\Re\Bigg(\text{\large{S}}^{\,\hat{\mathfrak{g}}}\Big(\frac{1}{d}x^{2}\,;\,-dh\mathfrak{f}\Big)\Bigg)
+O\bigg(\frac{x^{2}}{d\mathcal{H}}\bigg)
\end{equation}
and $\mathfrak{h}(t)=t$.
\end{lem}
\begin{proof}
The above assertions all follow from a direct application of Valler's Lemma, except for \eqref{eq:3.3}, \eqref{eq:3.4} and \eqref{eq:3.7} which necessitate some additional wrok in order to deal with the presence of the Dirichlet character $\chi$. The derivation of \eqref{eq:3.7} is simple and follows by expanding $\chi$ using additive characters:
\begin{equation*}
\chi(m)=\frac{1}{2i}\sum_{a\,(4)}\chi(a)\textit{\large{e}}\bigg(\frac{ma}{4}\bigg)\,.
\end{equation*}
As for \eqref{eq:3.3} and \eqref{eq:3.4}, we make use of the additional identity
\begin{equation*}
\sum_{a\,(4)}|\chi(a)|\textit{\large{e}}\bigg(\frac{ma}{4}\bigg)=2\Big\{\mathds{1}_{m\equiv0(4)}-\mathds{1}_{m\equiv2(4)}\Big\}
\end{equation*}
from which we derive for $\vary\in\mathbb{R}$:
\begin{equation*}
\sum_{a\,(4)}\chi(a)\sin{\Big(2\pi m\Big(\vary+\frac{a}{4}\Big)\Big)}=\Im\bigg(\sum_{a\,(4)}\chi(a)\textit{\large{e}}\bigg(\frac{ma}{4}\bigg)\textit{\large{e}}\big(m\vary\big)\bigg)=2\chi(m)\cos{(2\pi m\vary)}
\end{equation*}
and:
\begin{equation*}
\qquad\quad\,\,\,\,\sum_{a\,(4)}|\chi(a)|\cos{\Big(2\pi m\Big(\vary+\frac{a}{4}\Big)\Big)}=\Re\bigg(\sum_{a\,(4)}|\chi(a)|\textit{\large{e}}\bigg(\frac{ma}{4}\bigg)\textit{\large{e}}\big(m\vary\big)\bigg)=2\Big\{\mathds{1}_{m\equiv0(4)}-\mathds{1}_{m\equiv2(4)}\Big\}\cos{(2\pi m\vary)}\,.
\end{equation*}
Taking $m=-h$ and $\vary=\frac{x^{2}}{4d}\mathfrak{f}\Big(\frac{n}{x^{2}}\Big)$ with $1\leq h\leq \mathcal{H}$ and $0<n\leq\frac{x^{2}}{\sqrt{2}}$, one obtains \eqref{eq:3.3} and \eqref{eq:3.4}.  
\end{proof}
\subsection{Estimating exponential sums of certain type}
We now arrive at the core of the transformation process, which is the estimation of the exponential sums produced in Lemma 3.1. These sums will be estimated by applying the \textit{B}-process of Van der Corput. The standard estimates for the error terms produced by the \textit{B}-process will suffice for proving Theorem 1 and Theorem 3, but will fall just short of what we need when we arrive at the proof of Theorem 2. Thus, in order to make the transformation process applicable to all of our present purposes, we appeal to a result of Karatsuba and Korolev \cite{karatsuba2007theorem}. Here, we shall state it in a more convenient form.  
\begin{lem}[A sharp form of the \textit{B}-process]
Let $r,s\in\mathbb{R}$ with $s>r\geq0$, $f$ and $g$ real functions defined on $[r,s]$. Suppose the following conditions hold: \\\\
(C.1) $f\in C^{4}\big([r,s]\big)$ and $g\in C^{2}\big([r,s]\big)$.\\\\
(C.2) There exists a constant $c_{f}<0$ such that $f^{(2)}(t)\leq c_{f}$ for $t\in[r,s]$.\\\\
Write $\digamma$ for the inverse function of $f^{(1)}$, $\rho=f^{(1)}(r)$ and $\sigma=f^{(1)}(s)$. Let $u,v\in\mathbb{R}$ with $u>0$, and set $f_{u,v}(t)=uf(t)+vt$. Then for $Y\geq1$ we have:
\begin{equation}\label{eq:3.9}
\begin{split}
&\sum_{rY\,<\,n\,\leq\,sY}g\bigg(\frac{n}{Y}\bigg)\textit{\large{e}}\bigg(-Yf_{u,v}\bigg(\frac{n}{Y}\bigg)\bigg)=\\
&=\sqrt{\frac{Y}{u}}\underset{-u\rho-v\,\leq\, n\,\leq\,-u\sigma-v}{\sideset{}{''}\sum_{n\in\mathbb{Z}}}\,\,\frac{g\circ\digamma\Big(-\frac{n+v}{u}\Big)}{\sqrt{\Big|f^{(2)}\circ\digamma\Big(-\frac{n+v}{u}\Big)\Big|}}\textit{\large{e}}\bigg(-Yf_{u,n+v}\circ\digamma\Big(-\frac{n+v}{u}\Big)+\frac{1}{8}\bigg)+\mathscr{E}
\end{split}
\end{equation}
where the double-dash $''$ indicates that if one of the limit points in the above summation is an integer, then the corresponding summand is multiplied by $1/2$. The error term $\mathscr{E}$ satisfies the bound:
\begin{equation*}
\begin{split}
&|\mathscr{E}|\ll\log{\big((\rho-\sigma)u+2\big)}+1+\frac{1}{u}+\frac{u}{Y}+\mathcal{R}_{\rho}(u,v)+\mathcal{R}_{\sigma}(u,v)\\
&\textit{ where for } \delta=\rho,\sigma\quad \mathcal{R}_{\delta}(u,v)=\left\{
        \begin{array}{ll}
            0& ;\,u\delta+v\in\mathbb{Z}\\
            \textit{min}\bigg\{\sqrt{\frac{Y}{u}},\parallel -u\delta-v\parallel^{-1}\bigg\}&;\,\textit{otherwise}
        \end{array}
    \right.
\end{split}
\end{equation*}
and the implied constant in the $\ll$ relation depends on r,s and bounds for $f^{(k)}$ and $g^{(j)}$ with $k=2,3,4$ and $j=0,1,2$.
\end{lem}
\text{ }\\
With the aid of Lemma 3.2 we can now prove:
\begin{lem}
Let $x>0$ large, $d\in\mathbb{N}$ satisfying $1\leq d\leq \frac{x^{2}}{\sqrt{2}}$. Then for any integer $h\geq1$ we have:
\begin{equation}\label{eq:3.10}
\text{\large{S}}^{\,\mathfrak{g}}\Big(x^{2}\,;\,-\frac{h}{d}\mathfrak{f}\Big)=x\sqrt{d}h\underset{n\,\equiv\,0\,(d)}{\sideset{}{''}\sum_{0\,\leq\,n\,\leq\,h}}\frac{\mathfrak{g}\Big(\frac{n}{\sqrt{n^{2}+h^{2}}}\Big)}{\big(n^{2}+h^{2}\big)^{3/4}}\textit{\large{e}}\bigg(-\frac{\sqrt{n^{2}+h^{2}}}{d}x^{2}+\frac{1}{8}\bigg)
+O\Big(\log{2h}+d+\frac{h}{dx^{2}}\Big)\,.
\end{equation}
\begin{equation}\label{eq:3.11}
\text{\large{S}}^{\,\hat{\mathfrak{g}}}\Big(\frac{1}{d}x^{2}\,;\,-dh\mathfrak{f}\Big)=x\sqrt{d}h\sideset{}{''}\sum_{0\,\leq\,n\,\leq\,dh}\frac{\hat{\mathfrak{g}}\Big(\frac{n}{\sqrt{n^{2}+(dh)^{2}}}\Big)}{\big(n^{2}+(dh)^{2}\big)^{3/4}}\textit{\large{e}}\bigg(-\frac{\sqrt{n^{2}+(dh)^{2}}}{d}x^{2}+\frac{1}{8}\bigg)
+O\Big(\log{2hd}+\frac{hd^{2}}{x^{2}}\Big)\,.
\end{equation}
\begin{equation}\label{eq:3.12}
\begin{split}
\sum_{a\,(4)}\chi(-a)\text{\large{S}}^{\,\hat{\mathfrak{g}}}\Big(\frac{1}{d}x^{2}\,;\,-dh\mathfrak{f}\,-\,\frac{a}{4}\mathfrak{h}\Big)=&
4x\sqrt{4d}h\sideset{}{''}\sum_{0\,\leq\,n\,\leq\,4dh}\chi(-n)\frac{\hat{\mathfrak{g}}\Big(\frac{n}{\sqrt{n^{2}+(4dh)^{2}}}\Big)}{\big(n^{2}+(4dh)^{2}\big)^{3/4}}\textit{\large{e}}\bigg(-\frac{\sqrt{n^{2}+(4dh)^{2}}}{4d}x^{2}+\frac{1}{8}\bigg)+\\
&+O\Big(\log{2hd}+\frac{hd^{2}}{x^{2}}\Big)\,.
\end{split}
\end{equation}
\end{lem}
\begin{proof}
Set $r=0$, $s=\frac{1}{\sqrt{2}}$, $\rho=0$, $\sigma=-1$, $f=\mathfrak{f}$ and $\digamma(t)=-\frac{t}{\sqrt{1+t^{2}}}$.\\\\
\textit{(1)} The exponential sum $\text{\large{S}}^{\,\mathfrak{g}}\Big(x^{2}\,;\,-\frac{h}{d}\mathfrak{f}\Big)$ is of the type appearing on the RHS of \eqref{eq:3.9}, where $Y=x^{2}$, $g=\mathfrak{g}$, $u=\frac{h}{d}$ and $v=0$. Conditions C.1 and C.2 are clearly satisfied. We may thus appeal to Lemma 3.2 to obtain:\text{ }\\
\begin{equation}\label{eq:3.13}
\begin{split}
\text{\large{S}}^{\,\mathfrak{g}}\Big(x^{2}\,;\,-\frac{h}{d}\mathfrak{f}\Big)
&=x\sqrt{\frac{d}{h}}\,\,\sideset{}{''}\sum_{0\,\leq\,n\,\leq\,\frac{h}{d}}\,\,\frac{\mathfrak{g}\circ\digamma\Big(-\frac{dn}{h}\Big)}{\sqrt{\Big|\mathfrak{f}^{(2)}\circ\digamma\Big(-\frac{dn}{h}\Big)\Big|}}\textit{\large{e}}\bigg(-x^{2}\,\mathfrak{f}_{\frac{h}{d},n}\circ\digamma\Big(-\frac{dn}{h}\Big)+\frac{1}{8}\bigg)+\\
&+O\bigg(\log{\Big(\frac{h}{d}+2\Big)}+1+\frac{d}{h}+\frac{h}{dx^{2}}+\mathcal{R}_{0}\Big(\frac{h}{d},0\Big)+\mathcal{R}_{-1}\Big(\frac{h}{d},0\Big)\bigg)=\\
&=x\sqrt{d}h\underset{n\,\equiv\,0\,(d)}{\sideset{}{''}\sum_{0\,\leq\,n\,\leq\,h}}\frac{\mathfrak{g}\Big(\frac{n}{\sqrt{n^{2}+h^{2}}}\Big)}{\big(n^{2}+h^{2}\big)^{3/4}}\textit{\large{e}}\bigg(-\frac{\sqrt{n^{2}+h^{2}}}{d}x^{2}+\frac{1}{8}\bigg)
+O\Big(\log{2h}+d+\frac{h}{dx^{2}}\Big)\,.
\end{split}
\end{equation}
\textit{(2)} The exponential sum $\text{\large{S}}^{\,\hat{\mathfrak{g}}}\Big(\frac{1}{d}x^{2}\,;\,-dh\mathfrak{f}\Big)$ is of the type appearing on the RHS of \eqref{eq:3.9}, where $Y=\frac{x^{2}}{d}$, $g=\hat{\mathfrak{g}}$, $u=dh$ and $v=0$. Conditions C.1 and C.2 are clearly satisfied. We may thus appeal to Lemma 3.2 to obtain:\text{ }\\
\begin{equation}\label{eq:3.14}
\begin{split}
\text{\large{S}}^{\,\hat{\mathfrak{g}}}\Big(\frac{1}{d}x^{2}\,;\,-dh\mathfrak{f}\Big)&=\frac{x}{d\sqrt{h}}\,\,\,\sideset{}{''}\sum_{0\,\leq\,n\,\leq\,dh}\,\,\frac{\hat{\mathfrak{g}}\circ\digamma\Big(-\frac{n}{dh}\Big)}{\sqrt{\Big|\mathfrak{f}^{(2)}\circ\digamma\Big(-\frac{n}{dh}\Big)\Big|}}\textit{\large{e}}\bigg(-\frac{x^{2}}{d}\,\mathfrak{f}_{dh,n}\circ\digamma\Big(-\frac{n}{dh}\Big)+\frac{1}{8}\bigg)+\\
&+O\bigg(\log{\big(dh+2\big)}+1+\frac{1}{dh}+\frac{hd^{2}}{x^{2}}+\mathcal{R}_{0}\Big(dh,0\Big)+\mathcal{R}_{-1}\Big(dh,0\Big)\bigg)=\\
&=x\sqrt{d}h\sideset{}{''}\sum_{0\,\leq\,n\,\leq\,dh}\frac{\hat{\mathfrak{g}}\Big(\frac{n}{\sqrt{n^{2}+(dh)^{2}}}\Big)}{\big(n^{2}+(dh)^{2}\big)^{3/4}}\textit{\large{e}}\bigg(-\frac{\sqrt{n^{2}+(dh)^{2}}}{d}x^{2}+\frac{1}{8}\bigg)
+O\Big(\log{2hd}+\frac{hd^{2}}{x^{2}}\Big)\,.
\end{split}
\end{equation}
\textit{(3)} Fix an integer $0\leq a\leq 4$. Then the exponential sum $\text{\large{S}}^{\,\hat{\mathfrak{g}}}\Big(\frac{1}{d}x^{2}\,;\,-dh\mathfrak{f}\,-\,\frac{a}{4}\mathfrak{h}\Big)$ has the same parameters as in \eqref{eq:3.11}, except that now $v=\frac{a}{4}$. We thus have:
\begin{equation}\label{eq:3.15}
\begin{split}
\text{\large{S}}^{\,\hat{\mathfrak{g}}}\Big(\frac{1}{d}x^{2}\,;\,-dh\mathfrak{f}\,-\,\frac{a}{4}\mathfrak{h}\Big)&=\frac{x}{d\sqrt{h}}\,\,\sideset{}{''}\sum_{-\frac{a}{4}\,\leq\,n\,\leq\,dh-\frac{a}{4}}\,\,\frac{\hat{\mathfrak{g}}\circ\digamma\Big(-\frac{4n+a}{4dh}\Big)}{\sqrt{\Big|\mathfrak{f}^{(2)}\circ\digamma\Big(-\frac{4n+a}{4dh}\Big)\Big|}}\textit{\large{e}}\bigg(-\frac{x^{2}}{d}\,\mathfrak{f}_{dh,n+\frac{a}{4}}\circ\digamma\Big(-\frac{4n+a}{4dh}\Big)+\frac{1}{8}\bigg)+\\
&+O\bigg(\log{\big(dh+2\big)}+1+\frac{1}{dh}+\frac{hd^{2}}{x^{2}}+\mathcal{R}_{0}\Big(dh,\frac{a}{4}\Big)+\mathcal{R}_{-1}\Big(dh,\frac{a}{4}\Big)\bigg)=\\
&=4x\sqrt{4d}h\underset{n\,\equiv\,a\,(4)}{\sideset{}{''}\sum_{0\,\leq\,n\,\leq\,4dh}}\frac{\hat{\mathfrak{g}}\Big(\frac{n}{\sqrt{n^{2}+(4dh)^{2}}}\Big)}{\big(n^{2}+(4dh)^{2}\big)^{3/4}}\textit{\large{e}}\bigg(-\frac{\sqrt{n^{2}+(4dh)^{2}}}{4d}x^{2}+\frac{1}{8}\bigg)+O\Big(\log{2hd}+\frac{hd^{2}}{x^{2}}\Big)\,.
\end{split}
\end{equation}
Multiplying \eqref{eq:3.15} by $\chi(-a)$, and summing over $0\leq a\leq4$, we obtain \eqref{eq:3.12}.
\end{proof}
\subsection{Completing the transformation process}
With Lemma 3.1 \& 3.3 at our disposal, we can now successfully transform the sums involving $\psi$. In order to state the results, we shall need the following definitions.
\begin{mydef}
Let $H\geq1$. For $1\leq d\leq H$ an integer and $m\in\mathbb{N}$ define:
\begin{equation*}
\mathfrak{a}_{H}\big(m,d\big)=\frac{1}{m^{3/4}}\Bigg\{\,\,\,\underset{n\,\equiv\,0\,(d)}{\underset{0\,\leq\,n\,\leq\,h}{\underset{1\,\leq\, h\,\leq\, H}{\sideset{}{''}\sum_{n^{2}+h^{2}=m}}}}\tau\bigg(\frac{h}{[H]+1}\bigg)\mathfrak{g}\bigg(\frac{n}{\sqrt{m}}\bigg)+\underset{h\,\equiv\,0\,(d)}{\underset{0\,\leq\,n\,\leq\,h}{\underset{1\,\leq\, h\,\leq\, H}{\sideset{}{''}\sum_{n^{2}+h^{2}=m}}}}\tau\bigg(\frac{h}{d[H/d]+d}\bigg)\hat{\mathfrak{g}}\bigg(\frac{n}{\sqrt{m}}\bigg)\Bigg\}
\end{equation*}
\begin{equation*}
\begin{split}
&\mathfrak{a}^{\ast}_{H}\big(m,d\big)=\frac{1}{m^{3/4}}\Bigg\{\,\,\,\underset{n\,\equiv\,0\,(d)}{\underset{0\,\leq\,n\,\leq\,h}{\underset{1\,\leq\, h\,\leq\, H}{\sideset{}{''}\sum_{n^{2}+h^{2}=m}}}}\tau^{\ast}\bigg(\frac{h}{[H]+1}\bigg)\mathfrak{g}\bigg(\frac{n}{\sqrt{m}}\bigg)+\underset{h\,\equiv\,0\,(d)}{\underset{0\,\leq\,n\,\leq\,h}{\underset{1\,\leq\, h\,\leq\, H}{\sideset{}{''}\sum_{n^{2}+h^{2}=m}}}}\tau^{\ast}\bigg(\frac{h}{d[H/d]+d}\bigg)\hat{\mathfrak{g}}\bigg(\frac{n}{\sqrt{m}}\bigg)\Bigg\}\\
&\mathfrak{a}_{H,\chi}\big(m,d\big)=\frac{2}{m^{3/4}}\Bigg\{\,\,\,\underset{n\,\equiv\,0\,(d)}{\underset{0\,\leq\,n\,\leq\,h}{\underset{1\,\leq\, h\,\leq\, H}{\sideset{}{''}\sum_{n^{2}+h^{2}=m}}}}\chi(-h)\tau\bigg(\frac{h}{[H]+1}\bigg)\mathfrak{g}\bigg(\frac{n}{\sqrt{m}}\bigg)+\underset{h\,\equiv\,0\,(d)}{\underset{0\,\leq\,n\,\leq\,h}{\underset{1\,\leq\, h\,\leq\, H}{\sideset{}{''}\sum_{n^{2}+h^{2}=m}}}}\chi(-n)\tau\bigg(\frac{h}{d[H/d]+d}\bigg)\hat{\mathfrak{g}}\bigg(\frac{n}{\sqrt{m}}\bigg)\Bigg\}\\
&\mathfrak{b}^{\ast}_{H}\big(m,d\big)=\frac{2}{m^{3/4}}\Bigg\{\,\,\,\underset{n\,\equiv\,0\,(d)}{\underset{0\,\leq\,n\,\leq\,h}{\underset{1\,\leq\, h\,\leq\, H}{\sideset{}{''}\sum_{n^{2}+h^{2}=m}}}}\lambda(h)\tau^{\ast}\bigg(\frac{h}{[H]+1}\bigg)\mathfrak{g}\bigg(\frac{n}{\sqrt{m}}\bigg)+2\underset{n\,\equiv\,0\,(4)\,,\,h\,\equiv\,0\,(d)}{\underset{0\,\leq\,n\,\leq\,h}{\underset{1\,\leq\, h\,\leq\, H}{\sideset{}{''}\sum_{n^{2}+h^{2}=m}}}}\tau^{\ast}\bigg(\frac{h}{d[H/d]+d}\bigg)\hat{\mathfrak{g}}\bigg(\frac{n}{\sqrt{m}}\bigg)\Bigg\}
\end{split}
\end{equation*}
Here, the double-dash $''$ indicates that the terms $(n,h)=(0,h),(h,h)$ are multiplied by $1/2$, and $\lambda(h)=\mathds{1}_{h\equiv0(4)}-\mathds{1}_{h\equiv2(4)}$.
\end{mydef}
\text{ }\\
\textbf{Remark.} Note that if $m>2H^{2}$, then $\mathfrak{a}_{H}\big(m,\cdot\big),\,\mathfrak{a}^{\ast}_{H}\big(m,\cdot\big),\,\mathfrak{a}_{H,\chi}\big(m,\cdot\big),\,\mathfrak{b}^{\ast}_{H}\big(m,\cdot\big)\equiv0$. For $m$ of moderate size relative to $H$, $\mathfrak{a}_{H}\big(m,\cdot\big)$ and $\mathfrak{a}_{H,\chi}\big(m,\cdot\big)$ are very well approximated by a simple closed form expression, see $\S4.3$ Lemma $4.5$. For the remaining two, we shall only need an upper bound, which is given in $\S4.2$ Lemma $4.4$.\\\\ 
We are now ready to state the main results of this subsection. Note that in \eqref{eq:3.16} and \eqref{eq:3.18} stated below we obtain two different expressions for the same $\psi$-sum in the range $\sqrt{H}<d\leq H$. These tow expressions are of a different form, but are equal up to an admissible error. This will be crucial when we arrive at the mean square estimate for $q=3$, where \eqref{eq:3.18} will be used. The same remark applies to \eqref{eq:3.20} and \eqref{eq:3.22} for $\sqrt{H}4<d\leq H/4$. 
\begin{lem}
Let $x>0$ be large, $1\leq H<\frac{x^{2}}{\sqrt{2}}$ and $d\in\mathbb{N}$ satisfying $1\leq d\leq \frac{x^{2}}{\sqrt{2}}$.\\\\
\textit{(1)} If $d\leq H$, then:
\begin{equation}\label{eq:3.16}
\begin{split}
\text{\large{S}}^{\,\mathfrak{g}}_{\psi}\Big(x^{2}\,;\,\frac{1}{d}\mathfrak{f}\Big)+\text{\large{S}}^{\,\hat{\mathfrak{g}}}_{\psi}\Big(\frac{1}{d}x^{2}\,;\,d\mathfrak{f}\Big)=x\sqrt{d}\sum_{m}\mathfrak{a}_{H}\big(m,d\big)\sin{\Big(-2\pi \frac{\sqrt{m}}{d}x^{2}+\frac{\pi}{4}\Big)}+\mathcal{E}_{1}(d,H)
+O\Big(\log^{2}{x}+d\log{x}\Big)
\end{split}
\end{equation}
where $\mathcal{E}_{1}(d,H)$ satisfies the bound:
\begin{equation}\label{eq:3.17}
\begin{split}
\big|\mathcal{E}_{1}(d,H)\big|\leq x\sqrt{d}\sum_{m}\mathfrak{a}^{\ast}_{H}\big(m,d\big)\cos{\Big(-2\pi \frac{\sqrt{m}}{d}x^{2}+\frac{\pi}{4}\Big)}+O\Big(\log^{2}{x}+d\log{x}+\frac{x^{2}}{H}\Big)\,. 
\end{split}
\end{equation}
\textit{(2)} If $\sqrt{H}<d\leq H$, then:
\begin{equation}\label{eq:3.18}
\begin{split}
\text{\large{S}}^{\,\mathfrak{g}}_{\psi}\Big(x^{2}\,;\,\frac{1}{d}\mathfrak{f}\Big)+\text{\large{S}}^{\,\hat{\mathfrak{g}}}_{\psi}\Big(\frac{1}{d}x^{2}\,;\,d\mathfrak{f}\Big)&=\frac{x\sqrt{d}}{2}\sum_{1\,\leq\,h\,\leq\frac{H}{d}}\frac{1}{h^{3/2}}\tau\bigg(\frac{h}{[H/d]+1}\bigg)\sin{\Big(-2\pi \frac{h}{d}x^{2}+\frac{\pi}{4}\Big)}+\mathcal{E}_{2}(d,H)+\\
&+O\Big(\log^{2}{x}+d\log{x}+\frac{x^{2}}{d}\Big)
\end{split}
\end{equation}
where $\mathcal{E}_{2}(d,H)$ satisfies the bound:
\begin{equation}\label{eq:3.19}
\begin{split}
\big|\mathcal{E}_{2}(d,H)\big|\leq\frac{x\sqrt{d}}{2}\sum_{1\,\leq\,h\,\leq\frac{H}{d}}\frac{1}{h^{3/2}}\tau^{\ast}\bigg(\frac{h}{[H/d]+1}\bigg)\cos{\Big(-2\pi \frac{h}{d}x^{2}+\frac{\pi}{4}\Big)}+O\Big(\log^{2}{x}+d\log{x}+\frac{dx^{2}}{H}\Big)\,.
\end{split}
\end{equation}
\textit{(3)} If $d\leq\ H/4$, then:
\begin{equation}\label{eq:3.20}
\begin{split}
\text{\large{S}}^{\,\mathfrak{g}}_{\psi,\chi}\Big(x^{2}\,;\,\frac{1}{4d}\mathfrak{f}\Big)-\text{\large{S}}^{\,\hat{\mathfrak{g}},\chi}_{\psi}\Big(\frac{1}{d}x^{2}\,;\,d\mathfrak{f}\Big)&=x\sqrt{4d}\sum_{m}\mathfrak{a}_{H,\chi}\big(m,4d\big)\cos{\Big(-2\pi \frac{\sqrt{m}}{4d}x^{2}+\frac{\pi}{4}\Big)}+\mathcal{E}_{3}(d,H)+\\
&+O\Big(\log^{2}{x}+d\log{x}\Big)
\end{split}
\end{equation}
where $\mathcal{E}_{3}(d,H)$ satisfies the bound:
\begin{equation}\label{eq:3.21}
\begin{split}
\big|\mathcal{E}_{3}(d,H)\big|\leq x\sqrt{4d}\sum_{m}\mathfrak{b}^{\ast}_{H}\big(m,4d\big)\cos{\Big(-2\pi \frac{\sqrt{m}}{4d}x^{2}+\frac{\pi}{4}\Big)}+O\Big(\log^{2}{x}+d\log{x}+\frac{x^{2}}{H}\Big)\,. 
\end{split}
\end{equation}
\textit{(4)} If $\sqrt{H}/4<d\leq H/4$, then:
\begin{equation}\label{eq:3.22}
\begin{split}
\text{\large{S}}^{\,\mathfrak{g}}_{\psi,\chi}\Big(x^{2}\,;\,\frac{1}{4d}\mathfrak{f}\Big)-\text{\large{S}}^{\,\hat{\mathfrak{g}},\chi}_{\psi}\Big(\frac{1}{d}x^{2}\,;\,d\mathfrak{f}\Big)=&x\sqrt{4d}\sum_{1\,\leq\,h\,\leq\frac{H}{4d}}\frac{\chi(-h)}{h^{3/2}}\tau\bigg(\frac{h}{[H/4d]+1}\bigg)\cos{\Big(-2\pi \frac{h}{4d}x^{2}+\frac{\pi}{4}\Big)}+\\
&+\mathcal{E}_{4}(d,H)+O\Big(\log^{2}{x}+d\log{x}+\frac{x^{2}}{d}\Big)
\end{split}
\end{equation}
where $\mathcal{E}_{4}(d,H)$ satisfies the bound:
\begin{equation}\label{eq:3.23}
\begin{split}
\big|\mathcal{E}_{4}(d,H)\big|\leq x\sqrt{4d}\sum_{1\,\leq\,h\,\leq\frac{H}{4d}}\frac{\lambda(h)}{h^{3/2}}\tau^{\ast}\bigg(\frac{h}{[H/4d]+1}\bigg)\cos{\Big(-2\pi\frac{h}{4d}x^{2}+\frac{\pi}{4}\Big)}+O\Big(\log^{2}{x}+d\log{x}+\frac{dx^{2}}{H}\Big)\,.
\end{split}
\end{equation}
\end{lem}
\begin{proof}
\text{ }\\\\
\textit{(1)} Suppose $d\leq H$. We begin by applying Lemma 3.1 to $\text{\large{S}}^{\,\mathfrak{g}}_{\psi}\Big(x^{2}\,;\,\frac{1}{d}\mathfrak{f}\Big)$ with $\mathcal{H}=H$, and to $\text{\large{S}}^{\,\hat{\mathfrak{g}}}_{\psi}\Big(\frac{1}{d}x^{2}\,;\,d\mathfrak{f}\Big)$ with $\mathcal{H}=H/d$, obtaining:
\begin{equation}\label{eq:3.24}
\begin{split}
\text{\large{S}}^{\,\mathfrak{g}}_{\psi}\Big(x^{2}\,;\,\frac{1}{d}\mathfrak{f}\Big)+\text{\large{S}}^{\,\hat{\mathfrak{g}}}_{\psi}\Big(\frac{1}{d}x^{2}\,;\,d\mathfrak{f}\Big)=&\sum_{1\,\leq\,h\,\leq\,H}\tau\bigg(\frac{h}{[H]+1}\bigg)\frac{1}{h}\Im\Bigg(\text{\large{S}}^{\,\mathfrak{g}}\Big(x^{2}\,;\,-\frac{h}{d}\mathfrak{f}\Big)\Bigg)+\\
&+\sum_{1\,\leq\,h\,\leq\,\frac{H}{d}}\tau\bigg(\frac{h}{[H/d]+1}\bigg)\frac{1}{h}\Im\Bigg(\text{\large{S}}^{\,\hat{\mathfrak{g}}}\Big(\frac{1}{d}x^{2}\,;\,-dh\mathfrak{f}\Big)\Bigg)+\mathcal{E}_{1}(d,H)
\end{split}
\end{equation}
where $\mathcal{E}_{1}(d,H)$ satisfies the bound:
\begin{equation}\label{eq:3.25}
\begin{split}
\big|\mathcal{E}_{1}(d,H)\big|\leq&\sum_{1\,\leq\,h\,\leq\,H}\tau^{\ast}\bigg(\frac{h}{[H]+1}\bigg)\frac{1}{h}\Re\Bigg(\text{\large{S}}^{\,\mathfrak{g}}\Big(x^{2}\,;\,-\frac{h}{d}\mathfrak{f}\Big)\Bigg)+\sum_{1\,\leq\,h\,\leq\,\frac{H}{d}}\tau^{\ast}\bigg(\frac{h}{[H/d]+1}\bigg)\frac{1}{h}\Re\Bigg(\text{\large{S}}^{\,\hat{\mathfrak{g}}}\Big(\frac{1}{d}x^{2}\,;\,-dh\mathfrak{f}\Big)\Bigg)+\\
&+O\bigg(\frac{x^{2}}{H}\bigg)\,.
\end{split}
\end{equation}
Next, for each integer $h\geq1$, the exponential sums $\text{\large{S}}^{\,\mathfrak{g}}\Big(x^{2}\,;\,-\frac{h}{d}\mathfrak{f}\Big)$ and $\text{\large{S}}^{\,\hat{\mathfrak{g}}}\Big(\frac{1}{d}x^{2}\,;\,-dh\mathfrak{f}\Big)$ are estimated by \eqref{eq:3.10} (resp.) \eqref{eq:3.11} in Lemma 3.3. Inserting these estimates to the RHS of \eqref{eq:3.24}, summing over all $1\leq h\leq H$ and $1\leq h\leq\frac{H}{d}$, and then grouping the terms together, we obtain:
\begin{equation}\label{eq:3.26}
\begin{split}
&\sum_{1\,\leq\,h\,\leq\,H}\tau\bigg(\frac{h}{[H]+1}\bigg)\frac{1}{h}\Im\Bigg(\text{\large{S}}^{\,\mathfrak{g}}\Big(x^{2}\,;\,-\frac{h}{d}\mathfrak{f}\Big)\Bigg)
+\sum_{1\,\leq\,h\,\leq\,\frac{H}{d}}\tau\bigg(\frac{h}{[H/d]+1}\bigg)\frac{1}{h}\Im\Bigg(\text{\large{S}}^{\,\hat{\mathfrak{g}}}\Big(\frac{1}{d}x^{2}\,;\,-dh\mathfrak{f}\Big)\Bigg)=\\
&=x\sqrt{d}\sum_{1\,\leq\,h\,\leq\,H}\tau\bigg(\frac{h}{[H]+1}\bigg)\underset{n\,\equiv\,0\,(d)}{\sideset{}{''}\sum_{0\,\leq\,n\,\leq\,h}}\frac{\mathfrak{g}\Big(\frac{n}{\sqrt{n^{2}+h^{2}}}\Big)}{\big(n^{2}+h^{2}\big)^{3/4}}\sin{\Big(-2\pi \frac{\sqrt{n^{2}+h^{2}}}{d}x^{2}+\frac{\pi}{4}\Big)}+\\
&+x\sqrt{d}\sum_{1\,\leq\,h\,\leq\,\frac{H}{d}}\tau\bigg(\frac{h}{[H/d]+1}\bigg)\sideset{}{''}\sum_{0\,\leq\,n\,\leq\,dh}\frac{\hat{\mathfrak{g}}\Big(\frac{n}{\sqrt{n^{2}+(dh)^{2}}}\Big)}{\big(n^{2}+(dh)^{2}\big)^{3/4}}\sin{\Big(-2\pi \frac{\sqrt{n^{2}+h^{2}}}{d}x^{2}+\frac{\pi}{4}\Big)}+\\
&+O\Big(\log^{2}{x}+d\log{x}+\frac{H}{dx^{2}}+\frac{dH}{x^{2}}\Big)=x\sqrt{d}\sum_{m}\mathfrak{a}_{H}\big(m,d\big)\sin{\Big(-2\pi \frac{\sqrt{m}}{d}x^{2}+\frac{\pi}{4}\Big)}+O\Big(\log^{2}{x}+d\log{x}\Big)\,.
\end{split}
\end{equation}
The same calculation gives:
\begin{equation}\label{eq:3.27}
\begin{split}
\big|\mathcal{E}_{1}(d,H)\big|\leq x\sqrt{d}\sum_{m}\mathfrak{a}^{\ast}_{H}\big(m,d\big)\cos{\bigg(-2\pi \frac{\sqrt{m}}{d}x^{2}+\frac{\pi}{4}\bigg)}+O\bigg(\log^{2}{x}+d\log{x}+\frac{x^{2}}{H}\bigg)\,. 
\end{split}
\end{equation}
\textit{(2)} Suppose now that $\sqrt{H}<d\leq H$. We apply Lemma 3.1 to $\text{\large{S}}^{\,\mathfrak{g}}_{\psi}\Big(x^{2}\,;\,\frac{1}{d}\mathfrak{f}\Big)$ with $\mathcal{H}=H/d$, and bound $\text{\large{S}}^{\,\hat{\mathfrak{g}}}_{\psi}\Big(\frac{1}{d}x^{2}\,;\,d\mathfrak{f}\Big)$ trivially, obtaining:
\begin{equation}\label{eq:3.28}
\begin{split}
\text{\large{S}}^{\,\mathfrak{g}}_{\psi}\Big(x^{2}\,;\,\frac{1}{d}\mathfrak{f}\Big)+\text{\large{S}}^{\,\hat{\mathfrak{g}}}_{\psi}\Big(\frac{1}{d}x^{2}\,;\,d\mathfrak{f}\Big)&=\sum_{1\,\leq\,h\,\leq\,\frac{H}{d}}\tau\bigg(\frac{h}{[H/d]+1}\bigg)\frac{1}{h}\Im\Bigg(\text{\large{S}}^{\,\mathfrak{g}}\Big(x^{2}\,;\,-\frac{h}{d}\mathfrak{f}\Big)\Bigg)+\mathcal{E}_{2}(d,H)+O\bigg(\frac{x^{2}}{d}\bigg)
\end{split}
\end{equation}
where $\mathcal{E}_{2}(d,H)$ satisfies the bound:
\begin{equation}\label{eq:3.29}
\big|\mathcal{E}_{2}(d,H)\big|\leq\sum_{1\,\leq\,h\,\leq\,\frac{H}{d}}\tau^{\ast}\bigg(\frac{h}{[H/d]+1}\bigg)\frac{1}{h}\Re\Bigg(\text{\large{S}}^{\,\mathfrak{g}}\Big(x^{2}\,;\,-\frac{h}{d}\mathfrak{f}\Big)\Bigg)
+O\bigg(\frac{dx^{2}}{H}\bigg)\,.
\end{equation}
Fix $1\leq h\leq\frac{H}{d}$. By \eqref{eq:3.10} in Lemma 3.3
\begin{equation}\label{eq:3.30}
\begin{split}
\text{\large{S}}^{\,\mathfrak{g}}\Big(x^{2}\,;\,-\frac{h}{d}\mathfrak{f}\Big)&=x\sqrt{d}h\underset{n\,\equiv\,0\,(d)}{\sideset{}{''}\sum_{0\,\leq\,n\,\leq\,h}}\frac{\mathfrak{g}\Big(\frac{n}{\sqrt{n^{2}+h^{2}}}\Big)}{\big(n^{2}+h^{2}\big)^{3/4}}\textit{\large{e}}\bigg(-\frac{\sqrt{n^{2}+h^{2}}}{d}x^{2}+\frac{1}{8}\bigg)+O\Big(\log{2h}+d+\frac{h}{dx^{2}}\Big)
\end{split}
\end{equation}
and since $d>\sqrt{H}$, we have:
\begin{equation}\label{eq:3.31}
\begin{split}
&\underset{n\,\equiv\,0\,(d)}{\sideset{}{''}\sum_{0\,\leq\,n\,\leq\,h}}\frac{\mathfrak{g}\Big(\frac{n}{\sqrt{n^{2}+h^{2}}}\Big)}{\big(n^{2}+h^{2}\big)^{3/4}}\textit{\large{e}}\bigg(-\frac{\sqrt{n^{2}+h^{2}}}{d}x^{2}+\frac{1}{8}\bigg)=\frac{\mathfrak{g}(0)}{2h^{3/2}}\textit{\large{e}}\bigg(-\frac{h}{d}x^{2}+\frac{1}{8}\bigg)=\frac{1}{2h^{3/2}}\textit{\large{e}}\bigg(-\frac{h}{d}x^{2}+\frac{1}{8}\bigg)\,.
\end{split}
\end{equation}
Inserting these estimates to the RHS of \eqref{eq:3.28} and \eqref{eq:3.29}, summing over all $1\leq h\leq\frac{H}{d}$, we obtain \eqref{eq:3.18} and \eqref{eq:3.19}.
\text{ }\\\\
\textit{(3)} Suppose $d\leq H/4$. Applying Lemma 3.1 to $\text{\large{S}}^{\,\mathfrak{g}}_{\psi,\chi}\Big(x^{2}\,;\,\frac{1}{4d}\mathfrak{f}\Big)$ with $\mathcal{H}=H$, and to $-\text{\large{S}}^{\,\hat{\mathfrak{g}},\chi}_{\psi}\Big(\frac{1}{d}x^{2}\,;\,d\mathfrak{f}\Big)$ with $\mathcal{H}=H/4d$, we obtain:
\begin{equation}\label{eq:3.32}
\begin{split}
\text{\large{S}}^{\,\mathfrak{g}}_{\psi,\chi}\Big(x^{2}\,;\,\frac{1}{4d}\mathfrak{f}\Big)-\text{\large{S}}^{\,\hat{\mathfrak{g}},\chi}_{\psi}\Big(\frac{1}{d}x^{2}\,;\,d\mathfrak{f}\Big)&=2\sum_{1\,\leq\,h\,\leq\,H}\chi(-h)\tau\bigg(\frac{h}{[H]+1}\bigg)\frac{1}{h}\Re\Bigg(\text{\large{S}}^{\,\mathfrak{g}}\Big(x^{2}\,;\,-\frac{h}{4d}\mathfrak{f}\Big)\Bigg)+\\
&+\frac{1}{2}\sum_{1\,\leq\,h\,\leq\,\frac{H}{4d}}\tau\bigg(\frac{h}{[H/4d]+1}\bigg)\frac{1}{h}\Re\Bigg(\sum_{a\,(4)}\chi(-a)\text{\large{S}}^{\,\hat{\mathfrak{g}}}\Big(\frac{1}{d}x^{2}\,;\,-dh\mathfrak{f}\,-\,\frac{a}{4}\mathfrak{h}\Big)\Bigg)+\mathcal{E}_{3}(d,H)
\end{split}
\end{equation}
where $\mathcal{E}_{3}(d,H)$ satisfies the bound:
\begin{equation}\label{eq:3.33}
\begin{split}
\big|\mathcal{E}_{3}(d,H)\big|\leq&2\sum_{1\,\leq\,h\,\leq\,\mathcal{H}}\Big\{\mathds{1}_{h\equiv0(4)}-\mathds{1}_{h\equiv2(4)}\Big\}\tau^{\ast}\bigg(\frac{h}{[\mathcal{H}]+1}\bigg)\frac{1}{h}\Re\Bigg(\text{\large{S}}^{\,\mathfrak{g}}\Big(x^{2}\,;\,-\frac{h}{4d}\mathfrak{f}\Big)\Bigg)+\\
&+\sum_{1\,\leq\,h\,\leq\,\frac{H}{4d}}\tau^{\ast}\bigg(\frac{h}{[H/4d]+1}\bigg)\frac{1}{h}\Re\Bigg(\text{\large{S}}^{\,\hat{\mathfrak{g}}}\Big(\frac{1}{d}x^{2}\,;\,-dh\mathfrak{f}\Big)\Bigg)+O\bigg(\frac{x^{2}}{H}\bigg)\,.
\end{split}
\end{equation}
Next, for each integer $h\geq1$, the exponential sums $\text{\large{S}}^{\,\mathfrak{g}}\Big(x^{2}\,;\,-\frac{h}{4d}\mathfrak{f}\Big)$, $\text{\large{S}}^{\,\hat{\mathfrak{g}}}\Big(\frac{1}{d}x^{2}\,;\,-dh\mathfrak{f}\Big)$ and\newline$\sum\limits_{a\,(4)}\chi(-a)\text{\large{S}}^{\,\hat{\mathfrak{g}}}\Big(\frac{1}{d}x^{2}\,;\,-dh\mathfrak{f}\,-\,\frac{a}{4}\mathfrak{h}\Big)$ are estimates by \eqref{eq:3.10} \eqref{eq:3.11} and \eqref{eq:3.12} respectively in Lemma 3.3. Inserting these estimates to the RHS of \eqref{eq:3.32}, summing over all $1\leq h\leq H$ and $1\leq h\leq\frac{H}{4d}$, and then grouping the terms together, we obtain:\\
\begin{equation}\label{eq:3.34}
\begin{split}
&2\sum_{1\,\leq\,h\,\leq\,H}\chi(-h)\tau\bigg(\frac{h}{[H]+1}\bigg)\frac{1}{h}\Re\Bigg(\text{\large{S}}^{\,\mathfrak{g}}\Big(x^{2}\,;\,-\frac{h}{4d}\mathfrak{f}\Big)\Bigg)
+\frac{1}{2}\sum_{1\,\leq\,h\,\leq\,\frac{H}{4d}}\tau\bigg(\frac{h}{[H/4d]+1}\bigg)\frac{1}{h}\Re\Bigg(\sum_{a\,(4)}\chi(-a)\text{\large{S}}^{\,\hat{\mathfrak{g}}}\Big(\frac{1}{d}x^{2}\,;\,-dh\mathfrak{f}\,-\,\frac{a}{4}\mathfrak{h}\Big)\Bigg)=\\
&=2x\sqrt{4d}\sum_{1\,\leq\,h\,\leq\,H}\chi(-h)\tau\bigg(\frac{h}{[H]+1}\bigg)\underset{n\,\equiv\,0\,(4d)}{\sideset{}{''}\sum_{0\,\leq\,n\,\leq\,h}}\frac{\mathfrak{g}\Big(\frac{n}{\sqrt{n^{2}+h^{2}}}\Big)}{\big(n^{2}+h^{2}\big)^{3/4}}\cos\Big(-\frac{\sqrt{n^{2}+h^{2}}}{4d}x^{2}+\frac{1}{8}\Big)+\\
&+2x\sqrt{4d}\sum_{1\,\leq\,h\,\leq\,\frac{H}{4d}}\tau\bigg(\frac{h}{[H/4d]+1}\bigg)\sideset{}{''}\sum_{0\,\leq\,n\,\leq\,4dh}\chi(-n)\frac{\hat{\mathfrak{g}}\Big(\frac{n}{\sqrt{n^{2}+(4dh)^{2}}}\Big)}{\big(n^{2}+(4dh)^{2}\big)^{3/4}}\cos\Big(-\frac{\sqrt{n^{2}+(4dh)^{2}}}{4d}x^{2}+\frac{1}{8}\Big)+\\
&+O\Big(\log^{2}{x}+d\log{x}+\frac{H}{dx^{2}}+\frac{dH}{x^{2}}\Big)=x\sqrt{4d}\sum_{m}\mathfrak{a}_{H,\chi}\big(m,4d\big)\cos{\Big(-2\pi \frac{\sqrt{m}}{4d}x^{2}+\frac{\pi}{4}\Big)}+O\Big(\log^{2}{x}+d\log{x}\Big)\,.
\end{split}
\end{equation}
The same calculation gives:
\begin{equation}\label{eq:3.35}
\begin{split}
\big|\mathcal{E}_{3}(d,H)\big|\leq x\sqrt{4d}\sum_{m}\mathfrak{b}^{\ast}_{H}\big(m,4d\big)\cos{\Big(-2\pi \frac{\sqrt{m}}{4d}x^{2}+\frac{\pi}{4}\Big)}+O\Big(\log^{2}{x}+d\log{x}+\frac{x^{2}}{H}\Big)\,. 
\end{split}
\end{equation}
\textit{(4)} The proof is the same as in \textit{(2)}.
\end{proof}
\subsection{Approximate expressions for $\mathcal{E}_{q}(x)$}
We now gather the results from the previous subsections in order to obtain various expressions for $\mathcal{E}_{q}(x)$. We begin with an approximate expression for $\mathcal{E}_{q}(x)$ which will be used for the proof of Theorem 2.
\begin{prop}
Let $X>0$ be large, and set $H=X^{2}/2$. Then for any $X\leq x\leq2X$ we have:\\\\
\textit{(1)} For $q\equiv0\,(2)$
\begin{equation}\label{eq:3.36}
\begin{split}
&\mathcal{E}_{q}(x)=-2\varrho_{q}x^{2q-1}\sum_{1\,\leq\,d\,\leq\,\sqrt{H}}\frac{\xi(d)}{d^{q-3/2}}\sum_{m}\mathfrak{a}_{H}\big(m,d\big)\sin{\Big(-2\pi \frac{\sqrt{m}}{d}x^{2}+\frac{\pi}{4}\Big)}+
\mathcal{E}^{H}_{q}(x)+O\Big(x^{2q-2}\log^{2}{x}\Big)
\end{split}
\end{equation}
\text{ }\\
where $\mathcal{E}^{H}_{q}(x)$ satisfies the bound:
\begin{equation}\label{eq:3.37}
\begin{split}
&\big|\mathcal{E}^{H}_{q}(x)\big|\leq2\varrho_{q}x^{2q-1}\sum_{1\,\leq\,d\,\leq\,\sqrt{H}}\frac{|\xi(d)|}{d^{q-3/2}}\sum_{m}\mathfrak{a}^{\ast}_{H}\big(m,d\big)\cos{\Big(-2\pi \frac{\sqrt{m}}{d}x^{2}+\frac{\pi}{4}\Big)}+O\Big(x^{2q-2}\log^{2}{x}\Big)\,.
\end{split}
\end{equation}
\textit{(2)} For $q\equiv1\,(2)$
\begin{equation}\label{eq:3.38}
\begin{split}
\mathcal{E}_{q}(x)=&-2\varrho_{\chi,q}x^{2q-1}\sum_{1\,\leq\,d\,\leq\,\sqrt{H}}\frac{2^{q-1}\chi(d)}{d^{q-3/2}}\sum_{m}\mathfrak{a}_{H}\big(m,d\big)\sin{\Big(-2\pi \frac{\sqrt{m}}{d}x^{2}+\frac{\pi}{4}\Big)}+\\
&-2\varrho_{\chi,q}x^{2q-1}\underset{d\,\equiv\,0\,(4)}{\sum_{1\,\leq\,d\,\leq\,\sqrt{H}}}\frac{(-1)^{\frac{q+1}{2}}4^{q-1}}{d^{q-3/2}}\sum_{m}\mathfrak{a}_{H,\chi}\big(m,d\big)\cos{\Big(-2\pi \frac{\sqrt{m}}{d}x^{2}+\frac{\pi}{4}\Big)}+\\
&+\mathds{1}_{q=3}\bigg\{\Theta^{H}_{q,\chi}(x)+\Theta^{H,\chi}_{q}(x)\bigg\}+\mathcal{E}^{H}_{q}(x)+O\Big(x^{2q-2}\log^{2}{x}\Big)
\end{split}
\end{equation}
where $\mathcal{E}^{H}_{q}(x)$ satisfies the bound:
\begin{equation}\label{eq:3.39}
\begin{split}
&\big|\mathcal{E}^{H}_{q}(x)\big|\leq2\varrho_{\chi,q}x^{2q-1}\sum_{1\,\leq\,d\,\leq\,\sqrt{H}}\frac{1}{d^{q-3/2}}\sum_{m}\mathfrak{d}^{\ast}_{H}\big(m,d\big)\cos{\Big(-2\pi \frac{\sqrt{m}}{d}x^{2}+\frac{\pi}{4}\Big)}
+\mathds{1}_{q=3}\Theta^{H}_{q}(x)+O\Big(x^{2q-2}\log^{2}{x}\Big)
\end{split}
\end{equation}
with $\mathfrak{d}^{\ast}_{H}\big(m,d\big)=2^{q-1}\mathfrak{a}^{\ast}_{H}\big(m,d\big)+\mathds{1}_{d\equiv0(4)}4^{q-1}\mathfrak{b}^{\ast}_{H}\big(m,d\big)$. The terms $\Theta^{H}_{q,\chi}(x)$, $\Theta^{H,\chi}_{q}(x)$ and $\Theta^{\ast}_{q}(H)$ are given by:
\begin{equation*}
\begin{split}
&\Theta^{H}_{q,\chi}(x)=-2\varrho_{\chi,q}x^{2q-1}\sum_{\sqrt{H}\,<\,d\,\leq\,H}\frac{2^{q-2}\chi(d)}{d^{q-3/2}}\sum_{1\,\leq\,h\,\leq\frac{H}{d}}\frac{1}{h^{3/2}}\tau\bigg(\frac{h}{[H/d]+1}\bigg)\sin{\Big(-2\pi \frac{h}{d}x^{2}+\frac{\pi}{4}\Big)}\\
&\Theta^{H,\chi}_{q}(x)=-2\varrho_{\chi,q}x^{2q-1}\underset{d\,\equiv\,0\,(4)}{\sum_{\sqrt{H}\,<\,d\,\leq\,H}}\frac{(-1)^{\frac{q+1}{2}}4^{q-1}}{d^{q-3/2}}\sum_{1\,\leq\,h\,\leq\frac{H}{d}}\frac{\chi(-h)}{h^{3/2}}\tau\bigg(\frac{h}{[H/d]+1}\bigg)\cos{\Big(-2\pi \frac{h}{d}x^{2}+\frac{\pi}{4}\Big)}\\
&\Theta^{H}_{q}(x)=2\varrho_{\chi,q}x^{2q-1}\sum_{\sqrt{H}\,<\,d\,\leq\,H}\frac{1}{d^{q-3/2}}\sum_{1\,\leq\,h\,\leq\frac{H}{d}}\lambda^{\ast}(h,d)\frac{1}{h^{3/2}}\tau^{\ast}\bigg(\frac{h}{[H/d]+1}\bigg)\cos{\Big(-2\pi \frac{h}{d}x^{2}+\frac{\pi}{4}\Big)}
\end{split}
\end{equation*}
\text{ }\\
with $\lambda^{\ast}(h,d)=2^{q-2}+\mathds{1}_{d\equiv0(4)}4^{q-1}\lambda(h)$.
\end{prop}
\begin{proof}
Let $X\leq x\leq2X$.\\\\
\textit{(1)} Suppose $q\equiv0\,(2)$. By \eqref{eq:2.48} in Proposition 2.3, and \eqref{eq:3.16} in Lemma 3.4, we have:
\begin{equation}\label{eq:3.40}
\begin{split}
&\mathcal{E}_{q}(x)=-2\varrho_{q}x^{2q-2}\sum_{1\,\leq\,d\,\leq\,\sqrt{H}}\frac{\xi(d)}{d^{q-1}}\bigg\{\text{\large{S}}^{\,\mathfrak{g}}_{\psi}\Big(x^{2}\,;\,\frac{1}{d}\mathfrak{f}\Big)+\text{\large{S}}^{\,\hat{\mathfrak{g}}}_{\psi}\Big(\frac{1}{d}x^{2}\,;\,d\mathfrak{f}\Big)\bigg\}+O\Big(H^{1-q/2}x^{2q}+x^{2q-2}\log{x}\Big)=\\
&=-2\varrho_{q}x^{2q-2}\sum_{1\,\leq\,d\,\leq\,\sqrt{H}}\frac{\xi(d)}{d^{q-1}}\bigg\{x\sqrt{d}\sum_{m}\mathfrak{a}_{H}\big(m,d\big)\sin{\Big(-2\pi \frac{\sqrt{m}}{d}x^{2}+\frac{\pi}{4}\Big)}+\mathcal{E}_{1}(d,H)
+O\Big(\log^{2}{x}+d\log{x}\Big)\bigg\}+\\
&+O\Big(x^{2q-2}\log{x}\Big)=-2\varrho_{q}x^{2q-1}\sum_{1\,\leq\,d\,\leq\,\sqrt{H}}\frac{\xi(d)}{d^{q-3/2}}\sum_{m}\mathfrak{a}_{H}\big(m,d\big)\sin{\bigg(-2\pi \frac{\sqrt{m}}{d}x^{2}+\frac{\pi}{4}\bigg)}+\mathcal{E}^{H}_{q}(x)+O\Big(x^{2q-2}\log^{2}{x}\Big)
\end{split}
\end{equation}
where $\mathcal{E}^{H}_{q}(x)$ is given by:
\begin{equation}\label{eq:3.41}
\begin{split}
\mathcal{E}^{H}_{q}(x)=
-2\varrho_{q}x^{2q-2}\sum_{1\,\leq\,d\,\leq\,\sqrt{H}}\frac{\xi(d)}{d^{q-1}}\mathcal{E}_{1}(d,H)\,.
\end{split}
\end{equation}
Using \eqref{eq:3.17} in Lemma 3.4, we obtain the desired bound for $\mathcal{E}^{H}_{q}(x)$.
\text{ }\\\\
\textit{(2)} Suppose $q\equiv1\,(2)$. By \eqref{eq:2.49} in Proposition 2.3, and \eqref{eq:3.16}, \eqref{eq:3.18}, \eqref{eq:3.20}, \eqref{eq:3.22} in Lemma 3.4, we have:
\begin{equation}\label{eq:3.42}
\begin{split}
\mathcal{E}_{q}(x)=&-2\varrho_{\chi,q}x^{2q-2}\Bigg\{\sum_{1\,\leq\,d\,\leq\,\sqrt{H}}\frac{2^{q-1}\chi(d)}{d^{q-1}}\bigg\{\ldots\bigg\}+\sum_{\sqrt{H}\,<\,d\,\leq\,H}\frac{2^{q-1}\chi(d)}{d^{q-1}}\bigg\{\ldots\bigg\}\Bigg\}+\\
&-2\varrho_{\chi,q}x^{2q-2}\Bigg\{\sum_{1\,\leq\,d\,\leq\,\sqrt{H}/4}\frac{(-1)^{\frac{q+1}{2}}}{d^{q-1}}\bigg\{\ldots\bigg\}+\sum_{\sqrt{H}/4\,<\,d\,\leq\,H/4}\frac{(-1)^{\frac{q+1}{2}}}{d^{q-1}}\bigg\{\ldots\bigg\}\Bigg\}+O\bigg(\frac{x^{2q}}{H^{q-2}}+x^{2q-2}\log{x}\bigg)=\\
&=-2\varrho_{\chi,q}x^{2q-1}\sum_{1\,\leq\,d\,\leq\,\sqrt{H}}\frac{2^{q-1}\chi(d)}{d^{q-3/2}}\sum_{m}\mathfrak{a}_{H}\big(m,d\big)\sin{\Big(-2\pi \frac{\sqrt{m}}{d}x^{2}+\frac{\pi}{4}\Big)}+\\
&-2\varrho_{\chi,q}x^{2q-1}\underset{d\,\equiv\,0\,(4)}{\sum_{1\,\leq\,d\,\leq\,\sqrt{H}}}\frac{(-1)^{\frac{q+1}{2}}4^{q-1}}{d^{q-3/2}}\sum_{m}\mathfrak{a}_{H,\chi}\big(m,d\big)\cos{\Big(-2\pi \frac{\sqrt{m}}{d}x^{2}+\frac{\pi}{4}\Big)}+\Theta^{H}_{q,\chi}(x)+\Theta^{H,\chi}_{q}(x)+\\
&+\mathcal{E}^{H}_{q}(x)+O\Big(x^{2q-2}\log^{2}{x}\Big)
\end{split}
\end{equation}
where $\mathcal{E}^{H}_{q}(x)$ is given by:
\begin{equation}\label{eq:3.43}
\begin{split}
\mathcal{E}^{H}_{q}(x)=&
-2\varrho_{\chi,q}x^{2q-2}\Bigg\{\sum_{1\,\leq\,d\,\leq\,\sqrt{H}}\frac{2^{q-1}\chi(d)}{d^{q-1}}\mathcal{E}_{1}(d,H)+\sum_{1\,\leq\,d\,\leq\,\sqrt{H}/4}\frac{(-1)^{\frac{q+1}{2}}}{d^{q-1}}\mathcal{E}_{3}(d,H)\Bigg\}+\\
&-2\varrho_{\chi,q}x^{2q-2}\Bigg\{\sum_{\sqrt{H}\,<\,d\,\leq\,H}\frac{2^{q-1}\chi(d)}{d^{q-1}}\mathcal{E}_{2}(d,H)+\sum_{\sqrt{H}/4\,<\,d\,\leq\,H/4}\frac{(-1)^{\frac{q+1}{2}}}{d^{q-1}}\mathcal{E}_{4}(d,H)\Bigg\}\,.
\end{split}
\end{equation}
Using \eqref{eq:3.17}, \eqref{eq:3.19}, \eqref{eq:3.21} and \eqref{eq:3.23} in Lemma 3.4, we obtain:
\begin{equation}\label{eq:3.44}
\begin{split}
&\big|\mathcal{E}^{H}_{q}(x)\big|\leq2\varrho_{\chi,q}x^{2q-1}\sum_{1\,\leq\,d\,\leq\,\sqrt{H}}\frac{1}{d^{q-3/2}}\sum_{m}\mathfrak{d}^{\ast}_{H}\big(m,d\big)\cos{\Big(-2\pi \frac{\sqrt{m}}{d}x^{2}+\frac{\pi}{4}\Big)}
+\Theta^{H}_{q}(x)+O\Big(x^{2q-2}\log^{2}{x}\Big)\,.
\end{split}
\end{equation}
Now, since for $q>3$ odd we have that $\big|\Theta^{H}_{q,\chi}(x)\big|,\,\big|\Theta^{H,\chi}_{q}(x)\big|,\,\big|\Theta^{H}_{q}(x)\big|\ll x^{2q-2}$, we may write
\begin{equation*}
\begin{split}
\Theta^{H}_{q,\chi}(x)+\Theta^{H,\chi}_{q}(x)&=\mathds{1}_{q=3}\bigg\{\Theta^{H}_{q,\chi}(x)+\Theta^{H,\chi}_{q}(x)\bigg\}+\mathds{1}_{q>3}\bigg\{\Theta^{H}_{q,\chi}(x)+\Theta^{H,\chi}_{q}(x)\bigg\}=\\
&=\mathds{1}_{q=3}\bigg\{\Theta^{H}_{q,\chi}(x)+\Theta^{H,\chi}_{q}(x)\bigg\}+O\Big(x^{2q-2}\Big)
\end{split}
\end{equation*}
and
\begin{equation*}
\begin{split}
\Theta^{H}_{q}(x)=\mathds{1}_{q=3}\Theta^{H}_{q}(x)+\mathds{1}_{q>3}\Theta^{H}_{q}(x)=
\mathds{1}_{q=3}\Theta^{H}_{q}(x)+O\Big(x^{2q-2}\Big)\,.
\end{split}
\end{equation*}
Inserting these estimates to the RHS of \eqref{eq:3.38} and \eqref{eq:3.39}, we obtain \eqref{eq:2.41} and \eqref{eq:2.42}. 
\end{proof}
\text{ }\\
We end this section with the following proposition which will be used for the proof of Theorem 3.
\begin{prop}
Let $X>0$ large, and set $H=X/2$. For $x>0$ define:
\begin{equation*}
\begin{split}
&\Delta_{q}(x)=-\frac{\mathcal{E}_{q}\big(\sqrt{x}\,\big)}{2\varrho_{q}x^{q-1/2}}\quad\,\,\,;\quad q\equiv0\,(2)\\\\
&\Delta_{q}(x)=-\frac{\mathcal{E}_{q}\big(\sqrt{x}\,\big)}{2\varrho_{\chi,q}x^{q-1/2}}\quad;\quad q\equiv1\,(2)\,.
\end{split}
\end{equation*}
Then for any $X-1\leq x\leq X+1$ we have:\\\\
\textit{(1)} For $q\equiv0\,(2)$
\begin{equation}\label{eq:3.45}
\Delta_{q}(x)=\sum_{1\,\leq\,d\,\leq\,\sqrt{H}}\frac{\xi(d)}{d^{q-3/2}}\sum_{m\neq\square}\mathfrak{a}_{H}\big(m,d\big)\sin{\Big(-2\pi \frac{\sqrt{m}}{d}x+\frac{\pi}{4}\Big)}+
\Delta^{H}_{q}(x)+O\big(1\big)
\end{equation}
where $\Delta^{H}_{q}(x)$ satisfies the bound:
\begin{equation}\label{eq:3.46}
\begin{split}
&\big|\Delta^{H}_{q}(x)\big|\leq\sum_{1\,\leq\,d\,\leq\,\sqrt{H}}\frac{|\xi(d)|}{d^{q-3/2}}\sum_{m\neq\square}\mathfrak{a}^{\ast}_{H}\big(m,d\big)\cos{\Big(-2\pi \frac{\sqrt{m}}{d}x+\frac{\pi}{4}\Big)}+O\big(1\big)\,.
\end{split}
\end{equation}
\textit{(2)} For $q\equiv1\,(2)$
\begin{equation}\label{eq:3.47}
\begin{split}
\Delta_{q}(x)=&\sum_{1\,\leq\,d\,\leq\,\sqrt{H}}\frac{2^{q-1}\chi(d)}{d^{q-3/2}}\sum_{m\neq\square}\mathfrak{a}_{H}\big(m,d\big)\sin{\Big(-2\pi \frac{\sqrt{m}}{d}x+\frac{\pi}{4}\Big)}+\\
&+\underset{d\,\equiv\,0\,(4)}{\sum_{1\,\leq\,d\,\leq\,\sqrt{H}}}\frac{(-1)^{\frac{q+1}{2}}4^{q-1}}{d^{q-3/2}}\sum_{m\neq\square}\mathfrak{a}_{H,\chi}\big(m,d\big)\cos{\Big(-2\pi \frac{\sqrt{m}}{d}x+\frac{\pi}{4}\Big)}+\Delta^{H}_{q}(x)+O\big(1\big)
\end{split}
\end{equation}
where $\Delta^{H}_{q}(x)$ satisfies the bound:
\begin{equation}\label{eq:3.48}
\begin{split}
&\big|\Delta^{H}_{q}(x)\big|\leq\sum_{1\,\leq\,d\,\leq\,\sqrt{H}}\frac{1}{d^{q-\frac{3}{2}}}\sum_{m\neq\square}\mathfrak{d}^{\ast}_{H}\big(m,d\big)\cos{\Big(-2\pi \frac{\sqrt{m}}{d}x+\frac{\pi}{4}\Big)}
+O\big(1\big)\,.
\end{split}
\end{equation}
\end{prop}
\begin{proof}
Let $X-1\leq x\leq X+1$. Applying Proposition 2.3 with $x$ replaced by $\sqrt{x}$, and \eqref{eq:3.16}, \eqref{eq:3.17}, \eqref{eq:3.20}, \eqref{eq:3.21} in Lemma 3.4 with $x^{2}$ replaced by $x,$ we have by the same arguments as in the proof of Proposition 3.1 :\\\\
\textit{(1)} For $q\equiv0\,(2)$
\begin{equation}\label{eq:3.49}
\Delta_{q}(x)=\sum_{1\,\leq\,d\,\leq\,\sqrt{H}}\frac{\xi(d)}{d^{q-3/2}}\sum_{m}\mathfrak{a}_{H}\big(m,d\big)\sin{\Big(-2\pi \frac{\sqrt{m}}{d}x+\frac{\pi}{4}\Big)}+
\Delta^{H}_{q}(x)+O\big(1\big)
\end{equation}
with
\begin{equation}\label{eq:3.50}
\begin{split}
&\big|\Delta^{H}_{q}(x)\big|\leq\sum_{1\,\leq\,d\,\leq\,\sqrt{H}}\frac{|\xi(d)|}{d^{q-3/2}}\sum_{m}\mathfrak{a}^{\ast}_{H}\big(m,d\big)\cos{\Big(-2\pi \frac{\sqrt{m}}{d}x+\frac{\pi}{4}\Big)}+O\big(1\big)\,.
\end{split}
\end{equation}
\textit{(2)} For $q\equiv1\,(2)$
\begin{equation}\label{eq:3.51}
\begin{split}
\Delta_{q}(x)=&\sum_{1\,\leq\,d\,\leq\,\sqrt{H}}\frac{2^{q-1}\chi(d)}{d^{q-3/2}}\sum_{m}\mathfrak{a}_{H}\big(m,d\big)\sin{\Big(-2\pi \frac{\sqrt{m}}{d}x+\frac{\pi}{4}\Big)}+\\
&+\underset{d\,\equiv\,0\,(4)}{\sum_{1\,\leq\,d\,\leq\,\sqrt{H}}}\frac{(-1)^{\frac{q+1}{2}}4^{q-1}}{d^{q-3/2}}\sum_{m}\mathfrak{a}_{H,\chi}\big(m,d\big)\cos{\Big(-2\pi \frac{\sqrt{m}}{d}x+\frac{\pi}{4}\Big)}+\Delta^{H}_{q}(x)+O\big(1\big)
\end{split}
\end{equation}
with
\begin{equation}\label{eq:3.52}
\begin{split}
&\big|\Delta^{H}_{q}(x)\big|\leq\sum_{1\,\leq\,d\,\leq\,\sqrt{H}}\frac{1}{d^{q-3/2}}\sum_{m\neq\square}\mathfrak{d}^{\ast}_{H}\big(m,d\big)\cos{\Big(-2\pi \frac{\sqrt{m}}{d}x+\frac{\pi}{4}\Big)}
+O\big(1\big)\,.
\end{split}
\end{equation}
Since $|\mathfrak{a}_{H}\big(m^{2},d\big)|,\,|\mathfrak{a}^{\ast}_{H}\big(m^{2},d\big)|,\,|\mathfrak{a}_{H,\chi}\big(m^{2},d\big)|,\,|\mathfrak{d}^{\ast}_{H}\big(m^{2},d\big)|\ll m^{-3/2}r_{2}(m^{2})$, their total contribution is $O\big(1\big)$. Thus we may remove these terms.
\end{proof}
\subsection{Proof of Theorem 1}
We have everything in place for the proof of Theorem 1. In fact, we can prove the following much stronger result.
\begin{theorem*}
Let $q\geq3$ be fixed integer, $x>0$ large. Then for any large $1\leq H< \frac{x^{2}}{\sqrt{2}}$ we have:\\\\
\textit{(1)} For $q\equiv0\,(2)$
\begin{equation}\label{eq:3.53}
\begin{split}
\big|\mathcal{E}_{q}(x)\big|\ll&\,\, x^{2q-1}\sum_{1\,\leq\,d\,\leq\,H}\frac{1}{d^{q-3/2}}\Bigg\{\,\bigg|\sum_{m}\mathfrak{a}_{H}\big(m,d\big)\textit{\large{e}}\Big(\frac{\sqrt{m}}{d}x^{2}\Big)\bigg|+\bigg|\sum_{m}\mathfrak{a}^{\ast}_{H}\big(m,d\big)\textit{\large{e}}\Big(\frac{\sqrt{m}}{d}x^{2}\Big)\bigg|\Bigg\}+\\
&+\frac{x^{2q}}{H}+x^{2q-2}\log^{2}{x}\,.
\end{split}
\end{equation}
\textit{(2)} For $q\equiv1\,(2)$
\begin{equation}\label{eq:3.54}
\begin{split}
\big|\mathcal{E}_{q}(x)\big|\ll&\,\, x^{2q-1}\sum_{1\,\leq\,d\,\leq\,H}\frac{1}{d^{q-3/2}}\Bigg\{\,\bigg|\sum_{m}\mathfrak{a}_{H}\big(m,d\big)\textit{\large{e}}\Big(\frac{\sqrt{m}}{d}x^{2}\Big)\bigg|+\bigg|\sum_{m}\mathfrak{a}^{\ast}_{H}\big(m,d\big)\textit{\large{e}}\Big(\frac{\sqrt{m}}{d}x^{2}\Big)\bigg|\Bigg\}+\\
&+x^{2q-1}\sum_{1\,\leq\,d\,\leq\,H}\frac{1}{d^{q-3/2}}\Bigg\{\,\bigg|\sum_{m}\mathfrak{a}_{H,\chi}\big(m,d\big)\textit{\large{e}}\Big(\frac{\sqrt{m}}{d}x^{2}\Big)\bigg|+\bigg|\sum_{m}\mathfrak{b}^{\ast}_{H}\big(m,d\big)\textit{\large{e}}\Big(\frac{\sqrt{m}}{d}x^{2}\Big)\bigg|\Bigg\}+\\
&+\frac{x^{2q}}{H}+x^{2q-2}\log^{2}{x}\,.
\end{split}
\end{equation}
\end{theorem*}
\text{ }\\
Assuming the truth of Theorem 1.A for the moment, we present the proof of Theorem 1.
\begin{proof}(Theorem 1). Let $x,H>0$ be large, with $1\leq H< \frac{x^{2}}{\sqrt{2}}$. Before we appeal to Theorem 1.A, let us note that for any integers $d,m\in\mathbb{N}$ we have the crude estimate:
\begin{equation*}
\begin{split}
|\mathfrak{a}_{H}\big(m,d\big)|,|\mathfrak{a}^{\ast}_{H}\big(m,d\big)|,|\mathfrak{a}_{H,\chi}\big(m,d\big)|,|\mathfrak{b}^{\ast}_{H}\big(m,d\big)|\ll\frac{r_{2}(m)}{m^{3/4}}\mathds{1}_{m\leq2H^{2}}\,. 
\end{split}
\end{equation*}
Thus, estimating trivially the exponential sums in \eqref{eq:3.53} and \eqref{eq:3.54}, and then summing over all $1\leq d\leq H$, we obtain:
\begin{equation}\label{eq:3.55}
\begin{split}
\big|\mathcal{E}_{q}(x)\big|&\ll x^{2q-1}\sum_{1\,\leq\,m\,\leq\,2H^{2}}\frac{r_{2}(m)}{m^{3/4}}+\frac{x^{2q}}{H}+
x^{2q-2}\log^{2}{x}\ll x^{2q-1}H^{1/2}+\frac{x^{2q}}{H}+x^{2q-2}\log^{2}{x}\ll x^{2q-2/3}
\end{split}
\end{equation}
upon making the optimal choice $H=x^{2/3}$.
\end{proof}
\text{ }\\
We now proceed to the proof of Theorem 1.A.
\begin{proof}(Theorem 1.A).
Let $x,H>0$ be large with $1\leq H<\frac{x^{2}}{\sqrt{2}}$.\\\\
\textit{(1)} By \eqref{eq:2.48} in Proposition 2.3, and \eqref{eq:3.16} in Lemma 3.4, we have:
\begin{equation}
\begin{split}
\mathcal{E}_{q}(x)&=-2\varrho_{q}x^{2q-2}\sum_{1\,\leq\,d\,\leq\,\frac{x^{2}}{\sqrt{2}}}\frac{\xi(d)}{d^{q-1}}\bigg\{\text{\large{S}}^{\,\mathfrak{g}}_{\psi}\Big(x^{2}\,;\,\frac{1}{d}\mathfrak{f}\Big)+\text{\large{S}}^{\,\hat{\mathfrak{g}}}_{\psi}\Big(\frac{1}{d}x^{2}\,;\,d\mathfrak{f}\Big)\bigg\}+O\Big(x^{2q-2}\log{x}\Big)=\\
&=-2\varrho_{q}x^{2q-2}\sum_{1\,\leq\,d\,\leq\,H}\frac{\xi(d)}{d^{q-1}}\bigg\{\text{\large{S}}^{\,\mathfrak{g}}_{\psi}\Big(x^{2}\,;\,\frac{1}{d}\mathfrak{f}\Big)+\text{\large{S}}^{\,\hat{\mathfrak{g}}}_{\psi}\Big(\frac{1}{d}x^{2}\,;\,d\mathfrak{f}\Big)\bigg\}+O\Big(x^{2q-2}\log{x}+H^{2-q}x^{2q}\Big)=\\
&=-2\varrho_{q}x^{2q-2}\sum_{1\,\leq\,d\,\leq\,H}\frac{\xi_{\text{}_{q}}(d)}{d^{q-1}}\bigg\{x\sqrt{d}\sum_{m}\mathfrak{a}_{H}\big(m,d\big)\sin{\Big(-2\pi \frac{\sqrt{m}}{d}x^{2}+\frac{\pi}{4}\Big)}+\mathcal{E}_{1}(d,H)
+O\Big(\log^{2}{x}+d\log{x}\Big)\bigg\}+\\
&+O\Big(x^{2q-2}\log{x}+H^{-1}x^{2q}\Big)=-2\varrho_{q}x^{2q-1}\sum_{1\,\leq\,d\,\leq\,H}\frac{\xi_{\text{}_{q}}(d)}{d^{q-\frac{3}{2}}}\sum_{m}\mathfrak{a}_{H}\big(m,d\big)\sin{\Big(-2\pi \frac{\sqrt{m}}{d}x^{2}+\frac{\pi}{4}\Big)}+\\
&-2\varrho_{q}x^{2q-2}\sum_{1\,\leq\,d\,\leq\,H}\frac{\xi_{\text{}_{q}}(d)}{d^{q-1}}\mathcal{E}_{1}(d,H)+O\Big(x^{2q-2}\log^{2}{x}+H^{-1}x^{2q}\Big)\,.
\end{split}
\end{equation}
Taking absolute value, and using \eqref{eq:3.17} in Lemma 3.4, we obtain \eqref{eq:3.53}.
\text{ }\\\\
\textit{(2)} The proof is the same as in \textit{(1)}, where now we appeal to \eqref{eq:2.49} in Proposition 2.3, and \eqref{eq:3.16}, \eqref{eq:3.17}, \eqref{eq:3.20}, \eqref{eq:3.21} in Lemma 3.4.
\end{proof}
\section{Mean square estimates and proof of Theorem 2}
This section is devoted to the proof of Theorem 2, and we shall break down the proof into several steps. Consider the approximate expression for $\mathcal{E}_{q}(x)$ obtained in Proposition 3.1. To begin with, we first need to treat the remainder term $\mathcal{E}^{H}_{q}(x)$, and we shall do so separately for $q=3$ and $q>3$. In the latter case, the arguments are straight forward. When $q=3$ some care is needed, and in addition we shall also need to deal with $\Theta^{H}_{q,\chi}(x)$, $\Theta^{H,\chi}_{q}(x)$ and $\Theta^{\ast}_{q}(H)$. After some preparation work, we shall obtain the required bounds in either cases, and the relevant results are stated in Proposition 4.1 \& 4.2. Our next step will be then to estimate in mean-square the leading terms in the approximate expression, and  we shall do so in subsection 4.3. In the last subsection we shall give the proof of Theorem 2.
\subsection{General estimates}
In this subsection we shall state and prove some mean-square estimates for a certain class of arithmetical functions. The following lemma will be our primary tool in proving all of these results, where a proof may be found in \cite{montgomery1974hilbert}.    
\begin{HI}
Let $(a_{\lambda})_{\lambda\in\Lambda}$ and $(b_{\lambda})_{\lambda\in\Lambda}$ be two sequences of complex numbers indexed by a finite set $\Lambda$ of real numbers. Then:
\begin{equation*}
\bigg|\,\underset{\lambda\neq\nu}{\sum_{\lambda,\nu\in\Lambda}}\frac{a_{\lambda}\bar{b}_{\nu}}{\lambda-\nu}\bigg|\ll\bigg(\sum_{\lambda\in\Lambda}|a_{\lambda}|^{2}\delta_{\lambda}^{-1}\bigg)^{1/2}\bigg(\sum_{\lambda\in\Lambda}|b_{\lambda}|^{2}\delta_{\lambda}^{-1}\bigg)^{1/2}
\end{equation*}
where $\delta_{\lambda}=\underset{\nu\neq \lambda}{\underset{\nu\in\Lambda}{\textit{min}}}\,|\lambda-\nu|$, and the implied constant is absolute.
\end{HI}
\begin{lem}
Let $X>0$ large, and set $H=X^{2}/2$. Suppose the functions $\nu,\eta:\mathbb{N}\longrightarrow\mathbb{R}$ and  $\alpha,\beta:\mathbb{N}^{2}\longrightarrow\mathbb{R}$ satisfy the following two conditions:
\begin{equation*}
\begin{split}
&\textit{(C.1)} \quad |\nu(d)|,\,|\eta(d)|\ll \frac{1}{d^{q-3/2}}\,\mathds{1}_{d\leq\sqrt{H}}\\
&\textit{(C.2)} \quad |\alpha\big(m,d\big)|,\,|\beta\big(m,d\big)|\ll\frac{1}{m^{3/4}}\bigg\{\underset{a\neq0\,,\,b\equiv0(d)}{\sum_{a^{2}+b^{2}=m}}1\bigg\}\mathds{1}_{m\leq2H^{2}}
\end{split}
\end{equation*}
For $x>0$ define:
\begin{equation*}
\begin{split}
&\Psi^{\sin{}}_{H}\big(x;\nu,\alpha\big)=x^{2q-1}\sum_{d,m}\nu(d)\alpha\big(m,d\big)\sin{\Big(-2\pi \frac{\sqrt{m}}{d}x^{2}+\frac{\pi}{4}\Big)}\\
&\Psi^{\cos{}}_{H}\big(x;\eta,\beta\big)=x^{2q-1}\sum_{d,m}\eta(d)\beta\big(m,d\big)\cos{\Big(-2\pi \frac{\sqrt{m}}{d}x^{2}+\frac{\pi}{4}\Big)}\,.
\end{split}
\end{equation*}
and set
\begin{equation*}
\begin{split}
&\Xi_{H}\big(\nu,\alpha\big)=\underset{(d,m)=1}{\sum_{\ell,d,m}}|\mu(\ell)|\bigg(\sum_{r}\nu(rd)\alpha\big((rm)^{2}\ell,rd\big)\bigg)^{2}\quad;\quad
\Xi_{H}\big(\eta,\beta\big)=\underset{(d,m)=1}{\sum_{\ell,d,m}}|\mu(\ell)|\bigg(\sum_{r}\eta(rd)\beta\big((rm)^{2}\ell,rd\big)\bigg)^{2}
\end{split}
\end{equation*}
where $\mu$ is the m{\"o}bius function. Then:
\begin{equation}\label{eq:4.1}
\begin{split}
&\frac{1}{X}\bigintssss\limits_{ X}^{2X}\bigg\{\Psi^{\sin{}}_{H}\big(x;\nu,\alpha\big)+\Psi^{\cos{}}_{H}\big(x;\eta,\beta\big)\bigg\}^{2}\textit{d}x=\frac{c_{q}}{2}\bigg\{\Xi_{H}\big(\nu,\alpha\big)+\Xi_{H}\big(\eta,\beta\big)\bigg\}X^{2(2q-1)}+X^{2(2q-1)-2}\mathscr{E}^{H}\big(\nu,\alpha;\eta,\beta\big)
\end{split}
\end{equation}
where $\mathscr{E}^{H}\big(\nu,\alpha;\eta,\beta\big)$ satisfies the bound:
\begin{equation}\label{eq:4.2}
\begin{split}
\big|\mathscr{E}^{H}\big(\nu,\alpha;\eta,\beta\big)\big|&\ll \bigg(\sum_{1\,\leq\,d\,\leq\,\sqrt{H}}\frac{1}{d^{q-5/2}}\bigg)^{2}\log{X}+\bigg(\sum_{1\,\leq\,d\,\leq\,\sqrt{H}}\frac{\varpi(d)}{d^{q-2}}\bigg)^{2}\log^{2}{X}
+\bigg(\sum_{1\,\leq\,d\,\leq\,\sqrt{H}}\frac{1}{d^{q-5/2}}\bigg)\bigg(\sum_{1\,\leq\,d\,\leq\,\sqrt{H}}\frac{\varpi(d)}{d^{q-2}}\bigg)\log^{3/2}{X}
\end{split}
\end{equation}
and $c_{q}=\frac{2^{4q-1}-1}{4q-1}$.
\end{lem}
\text{ }\\
\textbf{Remark.} The estimation of the error term is not optimal with respect to the second and third summands appearing on the RHS of \eqref{eq:4.2}, and with more effort one could obtain sharper bounds. However, for the proof of Theorem 2, the above estimate will more than suffice.
\begin{proof}
We have:
\begin{equation}\label{eq:4.3}
\begin{split}
&\frac{1}{X}\bigintssss\limits_{ X}^{2X}\bigg\{\Psi^{\sin{}}_{H}\big(x;\nu,\alpha\big)+\Psi^{\cos{}}_{H}\big(x;\eta,\beta\big)\bigg\}^{2}\textit{d}x=\sum_{d_{1},d_{2},m_{1},m_{2}}\nu(d_{1})\nu(d_{2})\alpha\big(m_{1},d_{1}\big)\alpha\big(m_{2},d_{2}\big)\mathcal{I}_{1}\bigg(\frac{\sqrt{m_{1}}}{d_{1}},\frac{\sqrt{m_{2}}}{d_{2}}\bigg)+\\
&+\sum_{d_{1},d_{2},m_{1},m_{2}}\eta(d_{1})\eta(d_{2})\beta\big(m_{1},d_{1}\big)\beta\big(m_{2},d_{2}\big)\mathcal{I}_{2}\bigg(\frac{\sqrt{m_{1}}}{d_{1}},\frac{\sqrt{m_{2}}}{d_{2}}\bigg)+\\
&+2\sum_{d_{1},d_{2},m_{1},m_{2}}\nu(d_{1})\eta(d_{2})\alpha\big(m_{1},d_{1}\big)\beta\big(m_{2},d_{2}\big)\mathcal{I}_{3}\bigg(\frac{\sqrt{m_{1}}}{d_{1}},\frac{\sqrt{m_{2}}}{d_{2}}\bigg)=\text{\large{S}}_{1}+\text{\large{S}}_{2}+2\text{\large{S}}_{3}
\end{split}
\end{equation}
\text{ }\\
where for real numbers $r,s>0$  : 
\begin{equation}\label{eq:4.4}
\begin{split}
\mathcal{I}_{1}\big(r,s\big)&=\frac{1}{X}\bigintssss\limits_{ X}^{2X}x^{2(2q-1)}\sin{\Big(-2\pi rx^{2}+\frac{\pi}{4}\Big)}\sin{\Big(-2\pi sx^{2}+\frac{\pi}{4}\Big)}\textit{d}x=\frac{c_{q}}{2}\mathds{1}_{r=s}X^{2(2q-1)}+\\
&+\mathds{1}_{r\neq s}\frac{1}{r-s}\frac{1}{16\pi iX}\bigintssss\limits_{ X}^{2X}x^{2(2q-1)-1}\frac{\textit{d}}{\textit{d}x}\bigg\{\textit{\large{e}}\Big((r-s)x^{2}\Big)-\textit{\large{e}}\Big((s-r)x^{2}\Big)\bigg\}\textit{d}x
+O\bigg(\frac{X^{2(2q-1)-2}}{\sqrt{rs}}\bigg)
\end{split}
\end{equation}
\begin{equation}\label{eq:4.5}
\begin{split}
\mathcal{I}_{2}\big(r,s\big)&=\frac{1}{X}\bigintssss\limits_{ X}^{2X}x^{2(2q-1)}\cos{\Big(-2\pi rx^{2}+\frac{\pi}{4}\Big)}\cos{\Big(-2\pi sx^{2}+\frac{\pi}{4}\Big)}\textit{d}x=
\frac{c_{q}}{2}\mathds{1}_{r=s}X^{2(2q-1)}+\\
&+\mathds{1}_{r\neq s}\frac{1}{r-s}\frac{1}{16\pi iX}\bigintssss\limits_{ X}^{2X}x^{2(2q-1)-1}\frac{\textit{d}}{\textit{d}x}\bigg\{\textit{\large{e}}\Big((r-s)x^{2}\Big)-\textit{\large{e}}\Big((s-r)x^{2}\Big)\bigg\}\textit{d}x
+O\bigg(\frac{X^{2(2q-1)-2}}{\sqrt{rs}}\bigg)
\end{split}
\end{equation}
\begin{equation}\label{eq:4.6}
\begin{split}
\mathcal{I}_{3}\big(r,s\big)&=\frac{1}{X}\bigintssss\limits_{ X}^{2X}x^{2(2q-1)}\sin{\bigg(-2\pi rx^{2}+\frac{\pi}{4}\bigg)}\cos{\bigg(-2\pi sx^{2}+\frac{\pi}{4}\bigg)}\textit{d}x=\\
&+\mathds{1}_{r\neq s}\frac{1}{r-s}\frac{1}{16\pi X}\bigintssss\limits_{ X}^{2X}x^{2(2q-1)-1}\frac{\textit{d}}{\textit{d}x}\bigg\{\textit{\large{e}}\Big((r-s)x^{2}\Big)+\textit{\large{e}}\Big((s-r)x^{2}\Big)\bigg\}\textit{d}x
+O\bigg(\frac{X^{2(2q-1)-2}}{\sqrt{rs}}\bigg)\,.
\end{split}
\end{equation}
Next, we insert \eqref{eq:4.4}, \eqref{eq:4.5} and \eqref{eq:4.6} into the RHS of \eqref{eq:4.3}. For the off-diagonal terms, we fix $d_{1},d_{2}$, insert the summation over $\sqrt{m_{1}}/d_{1}\neq\sqrt{m_{2}}/d_{2}$ inside the integral, and then apply integration by parts once. We obtain:
\begin{equation}\label{eq:4.7}
\begin{split}
&\text{\large{S}}_{1}=\frac{c_{q}}{2}\sum_{d_{2}\sqrt{m_{1}}=d_{1}\sqrt{m_{2}}}\nu(d_{1})\nu(d_{2})\alpha\big(m_{1},d_{1}\big)\alpha\big(m_{2},d_{2}\big)+X^{2(2q-1)-2}E_{H}\big(\alpha,\alpha\big)+O\Big(X^{2(2q-1)-2}\log^{2}{X}\Big)\\\\
&\text{\large{S}}_{2}=\frac{c_{q}}{2}\sum_{d_{2}\sqrt{m_{1}}=d_{1}\sqrt{m_{2}}}\eta(d_{1})\eta(d_{2})\beta\big(m_{1},d_{1}\big)\beta\big(m_{2},d_{2}\big)+X^{2(2q-1)-2}E_{H}\big(\beta,\beta\big)+O\Big(X^{2(2q-1)-2}\log^{2}{X}\Big)\\\\
&\text{\large{S}}_{3}= X^{2(2q-1)-2}E_{H}\big(\alpha,\beta\big)+O\Big(X^{2(2q-1)-2}\log^{2}{X}\Big)
\end{split}
\end{equation}
where the terms $E_{H}\big(\alpha,\alpha\big),\,E_{H}\big(\beta,\beta\big)$ and $E_{H}\big(\alpha,\beta\big)$ satisfy the bounds:
\begin{equation}\label{eq:4.8}
\begin{split}
&\big|E_{H}\big(\alpha,\alpha\big)\big|\ll \sum_{1\,\leq\,d_{1},d_{2}\,\leq\,\sqrt{H}}\frac{1}{d_{1}^{q-5/2}}\frac{1}{d_{2}^{q-5/2}}\,\underset{X^{2}\,\leq\,\mathfrak{u}\,\leq\,4X^{2}}{\sup}\bigg|\sum_{m_{1}\neq m_{2}}\frac{\alpha_{\mathfrak{u}}\big(m_{1};d_{2},d_{1}\big)\bar{\alpha}_{\mathfrak{u}}\big(m_{2};d_{1},d_{2}\big)}{\sqrt{m_{1}}-\sqrt{m_{2}}}\bigg|\\\\
&\big|E_{H}\big(\beta,\beta\big)\big|\ll \sum_{1\,\leq\,d_{1},d_{2}\,\leq\,\sqrt{H}}\frac{1}{d_{1}^{q-5/2}}\frac{1}{d_{2}^{q-5/2}}\,\underset{X^{2}\,\leq\,\mathfrak{u}\,\leq\,4X^{2}}{\sup}\bigg|\sum_{m_{1}\neq m_{2}}\frac{\beta_{\mathfrak{u}}\big(m_{1};d_{2},d_{1}\big)\bar{\beta}_{\mathfrak{u}}\big(m_{2};d_{1},d_{2}\big)}{\sqrt{m_{1}}-\sqrt{m_{2}}}\bigg|\\\\
&\big|E_{H}\big(\alpha,\beta\big)\big|\ll \sum_{1\,\leq\,d_{1},d_{2}\,\leq\,\sqrt{H}}\frac{1}{d_{1}^{q-5/2}}\frac{1}{d_{2}^{q-5/2}}\,\underset{X^{2}\,\leq\,\mathfrak{u}\,\leq\,4X^{2}}{\sup}\bigg|\sum_{m_{1}\neq m_{2}}\frac{\alpha_{\mathfrak{u}}\big(m_{1};d_{2},d_{1}\big)\bar{\beta}_{\mathfrak{u}}\big(m_{2};d_{1},d_{2}\big)}{\sqrt{m_{1}}-\sqrt{m_{2}}}\bigg|
\end{split}
\end{equation}
and for $a,b\in\mathbb{N}$, $\mathfrak{u}\in\mathbb{R}$:
\begin{equation*}
\begin{split}
&\alpha_{\mathfrak{u}}\big(m;a,b\big)=\alpha\big(ma^{-2},b\big)\textit{\large{e}}\bigg(\frac{\sqrt{m}}{ab}\mathfrak{u}\bigg)\mathds{1}_{m\,\equiv\,0\,(a^{2})}\\\\
&\beta_{\mathfrak{u}}\big(m;a,b\big)=\beta\big(ma^{-2},b\big)\textit{\large{e}}\bigg(\frac{\sqrt{m}}{ab}\mathfrak{u}\bigg)\mathds{1}_{m\,\equiv\,0\,(a^{2})}\,.
\end{split}
\end{equation*}
We now proceed to deal with the terms appearing in \eqref{eq:4.8}. We shall only deal with $E_{H}\big(\alpha,\alpha\big)$, as the treatment of the other terms is identical.\\\\
Fix $1\leq d_{1},d_{2}\leq\sqrt{H}$ and $X^{2}\leq\mathfrak{u}\leq4X^{2}$. We first consider the sum over $m_{1}\neq m_{2}$ for which $\textit{min}\{\sqrt{m_{1}},\sqrt{m_{2}}\,\}\,\leq\,d_{1}d_{2}/2$. We have:
\begin{equation}\label{eq:4.9}
\begin{split}
&\bigg|\underset{\textit{min}\{\sqrt{m_{1}},\sqrt{m_{2}}\,\}\,\leq\,d_{1}d_{2}/2}{\sum_{m_{1}\neq m_{2}}}\frac{\alpha_{\mathfrak{u}}\big(m_{1};d_{2},d_{1}\big)\bar{\alpha}_{\mathfrak{u}}\big(m_{2};d_{1},d_{2}\big)}{\sqrt{m_{1}}-\sqrt{m_{2}}}\bigg|\ll\bigg|\underset{\frac{1}{2}\sqrt{m_{1}}\,<\,\sqrt{m_{2}}\,<\,2\sqrt{m_{1}}}{\underset{\textit{min}\{\sqrt{m_{1}},\sqrt{m_{2}}\,\}\,\leq\,d_{1}d_{2}/2}{\sum_{m_{1}\neq m_{2}}}}\frac{\alpha_{\mathfrak{u}}\big(m_{1};d_{2},d_{1}\big)\bar{\alpha}_{\mathfrak{u}}\big(m_{2};d_{1},d_{2}\big)}{\sqrt{m_{1}}-\sqrt{m_{2}}}\bigg|+\\\\
&+\frac{1}{d_{1}^{1/2}d_{2}^{1/2}}\log^{2}X\leq\underset{\sqrt{m_{1}}\,<\,d_{1}\,,\,\sqrt{m_{2}}\,<\,d_{2}}{\sum_{d_{2}^{2}m_{1}\neq d_{1}^{2}m_{2}}}\frac{\big|\alpha\big(m_{1},d_{1}\big)\big|\big|\alpha\big(m_{2},d_{2}\big)\big|}{|d_{2}\sqrt{m_{1}}-d_{1}\sqrt{m_{2}}|}+\frac{1}{d_{1}^{1/2}d_{2}^{1/2}}\log^{2}X\,.
\end{split}
\end{equation}
Note that for $\sqrt{m}<d$ we have by condition \textit{(C.2)}
\begin{equation*}
\begin{split}
\big|\alpha\big(m,d\big)\big|&\ll\frac{1}{m^{3/4}}\bigg\{\underset{a\neq0\,,\,b\equiv0(d)}{\sum_{a^{2}+b^{2}=m}}1\bigg\}\mathds{1}_{m\leq2H^{2}}=\frac{2}{m^{3/4}}\mathds{1}_{m=\square}\,\mathds{1}_{m\leq2H^{2}}\,.
\end{split}
\end{equation*}
Hence:
\begin{equation}\label{eq:4.10}
\begin{split}
&\underset{\sqrt{m_{1}}\,<\,d_{1}\,,\,\sqrt{m_{2}}\,<\,d_{2}}{\sum_{d_{2}^{2}m_{1}\neq d_{1}^{2}m_{2}}}\frac{\big|\alpha\big(m_{1},d_{1}\big)\big|\big|\alpha\big(m_{2},d_{2}\big)\big|}{|d_{2}\sqrt{m_{1}}-d_{1}\sqrt{m_{2}}|}\ll\underset{m_{1},m_{2}\geq1}{\sum_{d_{2}m_{1}\neq d_{1}m_{2}}}\frac{1}{m_{1}^{3/2}m_{2}^{3/2}|d_{2}m_{1}-d_{1}m_{2}|}\ll1\,.
\end{split}
\end{equation}
Inserting \eqref{eq:4.10} into \eqref{eq:4.9} we obtain:
\begin{equation}\label{eq:4.11}
\begin{split}
&\bigg|\underset{\textit{min}\{\sqrt{m_{1}},\sqrt{m_{2}}\,\}\,\leq\,d_{1}d_{2}/2}{\sum_{m_{1}\neq m_{2}}}\frac{\alpha_{\mathfrak{u}}\big(m_{1};d_{2},d_{1}\big)\bar{\alpha}_{\mathfrak{u}}\big(m_{2};d_{1},d_{2}\big)}{\sqrt{m_{1}}-\sqrt{m_{2}}}\bigg|\ll1+\frac{1}{d_{1}^{1/2}d_{2}^{1/2}}\log^{2}X\,.
\end{split}
\end{equation}
To estimate the complementary sum, we appeal to Hilbert's inequality, which gives:
\begin{equation}\label{eq:4.12}
\begin{split}
&\bigg|\underset{\sqrt{m_{1}},\sqrt{m_{2}}\,>\,d_{1}d_{2}/2}{\sum_{m_{1}\neq m_{2}}}\frac{\alpha_{\mathfrak{u}}\big(m_{1};d_{2},d_{1}\big)\bar{\alpha}_{\mathfrak{u}}\big(m_{2};d_{1},d_{2}\big)}{\sqrt{m_{1}}-\sqrt{m_{2}}}\bigg|\ll
\Bigg(\sum_{\sqrt{m}\,>\,d_{1}d_{2}/2}\big|\alpha_{\mathfrak{u}}\big(m;d_{2},d_{1}\big)\big|^{2}\sqrt{m}\Bigg)^{1/2}\Bigg(\sum_{\sqrt{m}\,>\,d_{1}d_{2}/2}\big|\bar{\alpha}_{\mathfrak{u}}\big(m;d_{1},d_{2}\big)\big|^{2}\sqrt{m}\Bigg)^{1/2}\ll\\
&\ll(d_{1}d_{2})^{1/2}\Bigg(\,\,\underset{m\,>\,d_{1}^{2}/4}{\sum_{1\,\leq\,m\,\leq\,2H^{2}}}\frac{1}{m}\bigg\{\underset{a\neq0\,,\,b\equiv0(d_{1})}{\sum_{a^{2}+b^{2}=m}}1\bigg\}^{2}\Bigg)^{1/2}\Bigg(\,\,\underset{m\,>\,d_{2}^{2}/4}{\sum_{1\,\leq\,m\,\leq\,2H^{2}}}\frac{1}{m}\bigg\{\underset{a\neq0\,,\,b\equiv0(d_{2})}{\sum_{a^{2}+b^{2}=m}}1\bigg\}^{2}\Bigg)^{1/2}\,.
\end{split}
\end{equation}
For $1\leq d\leq\sqrt{H}$ we have:
\begin{equation}\label{eq:4.13}
\underset{m\,>\,d^{2}/4}{\sum_{1\,\leq\,m\,\leq\,2H^{2}}}\frac{1}{m}\bigg\{\underset{a\neq0\,,\,b\equiv0(d)}{\sum_{a^{2}+b^{2}=m}}1\bigg\}^{2}\ll\frac{1}{d}\log{X}+\bigg(\frac{\varpi(d)}{d}\bigg)^{2}\log^{2}{X}\,.
\end{equation}
Inserting \eqref{eq:4.13} into \eqref{eq:4.12} we obtain:\\
\begin{equation}\label{eq:4.14}
\begin{split}
&\bigg|\underset{\sqrt{m_{1}},\sqrt{m_{2}}\,>\,d_{1}d_{2}/2}{\sum_{m_{1}\neq m_{2}}}\frac{\alpha_{\mathfrak{u}}\big(m_{1};d_{2},d_{1}\big)\bar{\alpha}_{\mathfrak{u}}\big(m_{2};d_{1},d_{2}\big)}{\sqrt{m_{1}}-\sqrt{m_{2}}}\bigg|\ll\log{X}+\frac{\varpi(d_{1})}{d_{1}^{1/2}}\log^{3/2}{X}+\frac{\varpi(d_{2})}{d_{2}^{1/2}}\log^{3/2}{X}+\frac{\varpi(d_{1})\varpi(d_{2})}{d_{1}^{1/2}d_{2}^{1/2}}\log^{2}{X}\,.
\end{split}
\end{equation}
Combining \eqref{eq:4.11} and \eqref{eq:4.14}, and summing over all $1\leq d_{1},d_{2}\leq\sqrt{H}$, we derive:
\begin{equation}\label{eq:4.15}
\begin{split}
\big|E_{H}\big(\alpha,\alpha\big)\big|&\ll \bigg(\sum_{1\,\leq\,d\,\leq\,\sqrt{H}}\frac{1}{d^{q-5/2}}\bigg)^{2}\log{X}+\bigg(\sum_{1\,\leq\,d\,\leq\,\sqrt{H}}\frac{\varpi(d)}{d^{q-2}}\bigg)^{2}\log^{2}{X}
+\bigg(\sum_{1\,\leq\,d\,\leq\,\sqrt{H}}\frac{1}{d^{q-5/2}}\bigg)\bigg(\sum_{1\,\leq\,d\,\leq\,\sqrt{H}}\frac{\varpi(d)}{d^{q-2}}\bigg)\log^{3/2}{X}
\end{split}
\end{equation}
The same bound holds for $E_{H}\big(\beta,\beta\big)$ and $E_{H}\big(\alpha,\beta\big)$. This proves \eqref{eq:4.2}.\\\\
It remains to deal with the diagonal terms appearing in \eqref{eq:4.7}, and we shall do so only for $\text{\large{S}}_{1}$, as the other case is identical. Recalling the simple fact that every integer $n\in\mathbb{N}$ can be written uniquely as $n=n_{1}^{2}n_{2}$ with $n_{2}$ square-free, we obtain:
\begin{equation}\label{eq:4.16}
\begin{split}
\sum_{d_{2}\sqrt{m_{1}}=d_{1}\sqrt{m_{2}}}\nu(d_{1})\nu(d_{2})\alpha\big(m_{1},d_{1}\big)\alpha\big(m_{2},d_{2}\big)&=\sum_{(d_{2}m_{1})^{2}\ell_{1}=(d_{1}m_{2})^{2}\ell_{2}}|\mu(\ell_{1})||\mu(\ell_{2})|\nu(d_{1})\nu(d_{2})\alpha\big(m_{1}^{2}\ell_{1},d_{1}\big)\alpha\big(m_{2}^{2}\ell_{2},d_{2}\big)=\\\\
&=\sum_{\ell}|\mu(\ell)|\sum_{d_{2}m_{1}=d_{1}m_{2}}\nu(d_{1})\nu(d_{2})\alpha\big(m_{1}^{2}\ell,d_{1}\big)\alpha\big(m_{2}^{2}\ell,d_{2}\big)=\\\\
&=\sum_{\ell,r_{1},r_{2}}|\mu(\ell)|\underset{(d_{1},m_{1})=(d_{2},m_{2})=1}{\sum_{d_{2}m_{1}=d_{1}m_{2}}}\nu(r_{1}d_{1})\nu(r_{2}d_{2})\alpha\big((r_{1}m_{1})^{2}\ell,r_{1}d_{1}\big)\alpha\big((r_{2}m_{2})^{2}\ell,r_{2}d_{2}\big)=\\\\
&=\underset{(d,m)=1}{\sum_{\ell,d,m}}|\mu(\ell)|\bigg(\sum_{r}\nu(rd)\alpha\big((rm)^{2}\ell,rd\big)\bigg)^{2}\,.
\end{split}
\end{equation}
This concludes the proof.
\end{proof}
\begin{lem}
Let $X>0$ large, and set $H=X^{2}/2$. Suppose $\eta:\mathbb{N}\longrightarrow\mathbb{R}$, $\beta:\mathbb{N}^{2}\longrightarrow\mathbb{R}$ satisfy conditions (C.1) and (C.2) from Lemma 4.1. Then:
\begin{equation}\label{eq:4.17}
\begin{split}
&\frac{1}{X}\bigintssss\limits_{ X}^{2X}\bigg\{\Psi^{\cos{}}_{H}\big(x;\eta,\beta\big)\bigg\}^{2}\textit{d}x\ll\bigg\{\Upsilon_{1}\big(\eta,\beta\big)+\Upsilon_{2}\big(\eta,\beta\big)\bigg\}X^{2(2q-1)}\log{X}+X^{2(2q-1)-2}\log^{2}{X}
\end{split}
\end{equation}
where
\begin{equation*}
\begin{split}
\Upsilon_{1}\big(\eta,\beta\big)=\sum_{d,m}|\eta(d)|\beta^{2}\big(m,d\big)\quad,\quad
\Upsilon_{2}\big(\eta,\beta\big)=\sum_{(d,m)=1}\gamma^{2}\big(m,d\big)\quad;\quad\gamma\big(m,d\big)=\sum_{r}\eta(rd)\beta\big((rm)^{2},rd\big)\,.
\end{split}
\end{equation*}
\end{lem}
\begin{proof}
Writing
\begin{equation*}
\begin{split}
&\Psi^{\neq\square}_{H}\big(x;\eta,\beta\big)=x^{2q-1}\underset{m\neq\square}{\sum_{d,m}}\eta(d)\beta\big(m,d\big)\textit{\large{e}}\bigg(\frac{\sqrt{m}}{d}x^{2}\bigg)\\
&\Psi^{=\square}_{H}\big(x;\eta,\beta\big)=x^{2q-1}\underset{m=\square}{\sum_{d,m}}\eta(d)\beta\big(m,d\big)\textit{\large{e}}\bigg(\frac{\sqrt{m}}{d}x^{2}\bigg)=x^{2q-1}\sum_{d,m}\eta(d)\beta\big(m^{2},d\big)\textit{\large{e}}\bigg(\frac{m}{d}x^{2}\bigg)
\end{split}
\end{equation*}
We have:
\begin{equation}\label{eq:4.18}
\begin{split}
&\frac{1}{X}\bigintssss\limits_{ X}^{2X}\bigg\{\Psi^{\cos{}}_{H}\big(x;\eta,\beta\big)\bigg\}^{2}\textit{d}x\ll\frac{1}{X}\bigintssss\limits_{ X}^{2X}\Big|\Psi^{\neq\square}_{H}\big(x;\eta,\beta\big)\Big|^{2}\textit{d}x+\frac{1}{X}\bigintssss\limits_{ X}^{2X}\Big|\Psi^{=\square}_{H}\big(x;\eta,\beta\big)\Big|^{2}\textit{d}x\,.
\end{split}
\end{equation}
We first consider $\Psi^{\neq\square}_{H}\big(x;\eta,\beta\big)$. Applying Cauchy–Schwarz inequality in order to move the summation over $d$ to the outside, we have:
\begin{equation}\label{eq:4.19}
\begin{split}
&\frac{1}{X}\bigintssss\limits_{ X}^{2X}\Big|\Psi^{\neq\square}_{H}\big(x;\eta,\beta\big)\Big|^{2}\textit{d}x\ll\sum_{d}|\eta(d)|\,\frac{1}{X}\bigintssss\limits_{ X}^{2X}x^{2(2q-1)}\bigg|\sum_{m\neq\square}\beta\big(m,d\big)\textit{\large{e}}\bigg(\frac{\sqrt{m}}{d}x^{2}\bigg)\bigg|^{2}\textit{d}x
\end{split}
\end{equation}
Fix $d\in\mathbb{N}$. Then, by the same arguments as in Lemma 3.1, we have:
\begin{equation}\label{eq:4.20}
\begin{split}
&\frac{1}{X}\bigintssss\limits_{ X}^{2X}x^{2(2q-1)}\bigg|\sum_{m\neq\square}\beta\big(m,d\big)\textit{\large{e}}\bigg(\frac{\sqrt{m}}{d}x^{2}\bigg)\bigg|^{2}\textit{d}x=c_{q}\sum_{m\neq\square}\beta^{2}\big(m,d\big)+dX^{2(2q-1)-2}E_{H}\big(\beta;d\big)
\end{split}
\end{equation}
where $E_{H}\big(\beta;d\big)$ satisfies the bound:
\begin{equation}\label{eq:4.21}
\begin{split}
&\big|E_{H}\big(\beta;d\big)\big|\ll \underset{X^{2}\,\leq\,\mathfrak{u}\,\leq\,4X^{2}}{\sup}\bigg|\underset{m_{1},m_{2}\neq\square}{\sum_{m_{1}\neq m_{2}}}\frac{\beta_{\mathfrak{u}}\big(m_{1},d\big)\bar{\beta}_{\mathfrak{u}}\big(m_{2},d\big)}{\sqrt{m_{1}}-\sqrt{m_{2}}}\bigg|
\end{split}
\end{equation}
and for $\mathfrak{u}\in\mathbb{R}$:
\begin{equation*}
\beta_{\mathfrak{u}}\big(m,d\big)=\beta\big(m,d\big)\textit{\large{e}}\bigg(\frac{\sqrt{m}}{d}\mathfrak{u}\bigg)\,.
\end{equation*}
\text{ }\\
Applying Hilbert's inequality, we get:
\begin{equation}\label{eq:4.22}
\begin{split}
&\underset{X^{2}\,\leq\,\mathfrak{u}\,\leq\,4X^{2}}{\sup}\bigg|\underset{m_{1},m_{2}\neq\square}{\sum_{m_{1}\neq m_{2}}}\frac{\beta_{\mathfrak{u}}\big(m_{1},d\big)\bar{\beta}_{\xi}\big(m_{2};d\big)}{\sqrt{m_{1}}-\sqrt{m_{2}}}\bigg|\ll\underset{m\neq\square}{\sum_{1\,\leq\,m\,\leq\,2H^{2}}}\frac{1}{m}\bigg\{\underset{a\neq0\,,\,b\equiv0(d)}{\sum_{a^{2}+b^{2}=m}}1\bigg\}^{2}=\\
&=\underset{m\neq\square}{\sum_{1\,\leq\,m\,\leq\,2H^{2}}}\frac{1}{m}\bigg\{\underset{a,b\neq0\,,\,b\equiv0(d)}{\sum_{a^{2}+b^{2}=m}}1\bigg\}^{2}\ll\frac{1}{d}\log{X}+\bigg(\frac{\varpi(d)}{d}\bigg)^{2}\log^{2}{X}\,.
\end{split}
\end{equation}
Summing over all $d$'s, we obtain by \eqref{eq:4.19} and \eqref{eq:4.22}:
\begin{equation}\label{eq:4.23}
\begin{split}
\frac{1}{X}\bigintssss\limits_{ X}^{2X}\Big|\Psi^{\neq\square}_{H}\big(x;\eta,\beta\big)\Big|^{2}\textit{d}x\ll\Upsilon_{1}\big(\eta,\beta\big)X^{2(2q-1)}+X^{2(2q-1)-2}\log^{2}{X}\,.
\end{split}
\end{equation}
Now we treat $\Psi^{=\square}_{H}\big(x;\eta,\beta\big)$. We begin by extracting the greatest common divisor of $d$ and $m$.\\
\begin{equation*}
\Psi^{=\square}_{H}\big(x;\eta,\beta\big)=x^{2q-1}\sum_{(d,m)=1}\bigg\{\sum_{r}\eta(rd)\beta\big((rm)^{2},rd\big)\bigg\}\textit{\large{e}}\bigg(\frac{m}{d}x^{2}\bigg)=x^{2q-1}\sum_{(d,m)=1}\gamma\big(m,d\big)\textit{\large{e}}\bigg(\frac{m}{d}x^{2}\bigg)\,.
\end{equation*}
Note that for $d,m\in\mathbb{N}$, $\gamma\big(m,d\big)$ satisfies the bound: 
\begin{equation}\label{eq:4.24}
\begin{split}
&|\gamma\big(m,d\big)|\ll\frac{1}{d^{q-3/2}m^{3/2}}\underset{r\,\leq\,\textit{min}\,\big\{\frac{\sqrt{H}}{d},\frac{\sqrt{2}H}{m}\big\}}{\sum_{r\in\mathbb{N}}}\frac{1}{r^{q}}\bigg\{\underset{a\neq0\,,\,b\equiv0(rd)}{\sum_{a^{2}+b^{2}=(rm)^{2}}}1\bigg\}=\\
&=\frac{1}{d^{q-3/2}m^{3/2}}\bigg\{\underset{a\neq0\,,\,b\equiv0(d)}{\sum_{a^{2}+b^{2}=m^{2}}}1\bigg\}\underset{r\,\leq\,\textit{min}\,\big\{\frac{\sqrt{H}}{d},\frac{\sqrt{2}H}{m}\big\}}{\sum_{r\in\mathbb{N}}}\frac{1}{r^{q}}\ll\frac{1}{d^{q-3/2}m^{3/2}}\bigg\{\underset{a\neq0\,,\,b\equiv0(d)}{\sum_{a^{2}+b^{2}=m^{2}}}1\bigg\}\mathds{1}_{d\leq\sqrt{H}}\,\mathds{1}_{m\leq\sqrt{2}H}\,.
\end{split}
\end{equation}
Set $\mathcal{D}=\Big\{2^{k}:k\in\mathbb{N}_{0}\,,\,2^{k}\leq\sqrt{H}\Big\}$, and for $D\in\mathcal{D}$ we shall write $d\sim D$ to mean $D\leq d<2D$. Then we have:
\begin{equation}\label{eq:4.25}
\Psi^{=\square}_{H}\big(x;\eta,\beta\big)=\sum_{D\in\mathcal{D}}\,\,x^{2q-1}\underset{d\,\sim\,D}{\sum_{(d,m)=1}}\gamma\big(m,d\big)\textit{\large{e}}\bigg(\frac{m}{d}x^{2}\bigg)=\sum_{D\in\mathcal{D}}\,\Psi^{=\square}_{H}\big(x;\eta,\beta;D\big)
\end{equation}
where
\begin{equation*}
\Psi^{=\square}_{H}\big(x;\eta,\beta;D\big)=x^{2q-1}\underset{d\,\sim\,D}{\sum_{(d,m)=1}}\gamma\big(m,d\big)\textit{\large{e}}\bigg(\frac{m}{d}x^{2}\bigg).
\end{equation*}
Applying Cauchy–Schwarz inequality in order to localize the variable $d$ to a dyadic segment $\sim D$, we have:  
\begin{equation}\label{eq:4.26}
\begin{split}
&\frac{1}{X}\bigintssss\limits_{ X}^{2X}\Big|\Psi^{=\square}_{H}\big(x;\eta,\beta\big)\Big|^{2}\textit{d}x\ll\log{X}\sum_{D\in\mathcal{D}}\,\,\frac{1}{X}\bigintssss\limits_{ X}^{2X}\Big|\Psi^{=\square}_{H}\big(x;\eta,\beta;D\big)\Big|^{2}\textit{d}x\,.
\end{split}
\end{equation}
Fix $D\in\mathcal{D}$. Then
\begin{equation}\label{eq:4.27}
\begin{split}
&\frac{1}{X}\bigintssss\limits_{ X}^{2X}\Big|\Psi^{=\square}_{H}\big(x;\eta,\beta;D\big)\Big|^{2}\textit{d}x=c_{q}X^{2(2q-1)}\underset{d\,\sim\, D}{\sum_{(d,m)=1}}\gamma^{2}\big(m,d\big)+X^{2(2q-1)-2}E_{H}\big(D\big)
\end{split}
\end{equation}
where $E_{H}\big(D\big)$ satisfies the bound:
\begin{equation}\label{eq:4.28}
\begin{split}
&\big|E_{H}\big(D\big)\big|\ll\underset{X^{2}\,\leq\,\mathfrak{u}\,\leq\,4X^{2}}{\sup}\bigg|\underset{d_{1},d_{2}\,\sim\,D}{\underset{(m_{1},d_{1})=(m_{2},d_{2})=1}{\sum_{\frac{m_{1}}{d_{1}}\neq\frac{m_{2}}{d_{2}}}}}\frac{\gamma_{\mathfrak{u}}\big(m_{1},d_{1}\big)\bar{\gamma}_{\mathfrak{u}}\big(m_{2},d_{2}\big)}{\frac{m_{1}}{d_{1}}-\frac{m_{2}}{d_{2}}}\bigg|
\end{split}
\end{equation}
and for $d,m\in\mathbb{N}$, $\mathfrak{u}\in\mathbb{R}$:
\begin{equation*}
\gamma_{\mathfrak{u}}\big(m,d\big)=\gamma\big(m,d\big)\textit{\large{e}}\bigg(\frac{m}{d}\mathfrak{u}\bigg)\,.
\end{equation*}
Applying Hilbert's inequality, we get by \eqref{eq:4.24}:
\begin{equation}\label{eq:4.29}
\begin{split}
&\underset{X^{2}\,\leq\,\mathfrak{u}\,\leq\,4X^{2}}{\sup}\bigg|\underset{d_{1},d_{2}\,\sim\,D}{\underset{(m_{1},d_{1})=(m_{2},d_{2})=1}{\sum_{\frac{m_{1}}{d_{1}}\neq\frac{m_{2}}{d_{2}}}}}\frac{\gamma_{\mathfrak{u}}\big(m_{1},d_{1}\big)\bar{\gamma}_{\mathfrak{u}}\big(m_{2},d_{2}\big)}{\frac{m_{1}}{d_{1}}-\frac{m_{2}}{d_{2}}}\bigg|\ll\underset{d\,\sim\,D}{\sum_{(d,m)=1}}d^{2}\gamma^{2}\big(m,d\big)\ll\sum_{d\,\sim\,D}\frac{1}{d^{2q-5}}
\end{split}
\end{equation}
Summing over all $D$'s, we obtain by \eqref{eq:4.26} and \eqref{eq:4.29}:
\begin{equation}\label{eq:4.30}
\begin{split}
&\frac{1}{X}\bigintssss\limits_{ X}^{2X}\Big|\Psi^{=\square}_{H}\Big(x;\eta,\beta\Big)\Big|^{2}\textit{d}x\ll\Upsilon_{2}\big(\eta,\beta\big) X^{2(2q-1)}\log{X}+X^{2(2q-1)-2}\log^{2}{X}.
\end{split}
\end{equation}
By \eqref{eq:4.18} this concludes the proof.
\end{proof}
\begin{lem}
Let $X>0$ large, and set $H=X^{2}/2$. Suppose the function $\theta:\mathbb{N}^{2}\longrightarrow\mathbb{R}$ satisfies the following codition:
\begin{equation*}
\textit{(C.3)} \quad |\theta\big(h,d\big)|\ll \frac{1}{d^{q-\frac{3}{2}}h^{3/2}}\mathds{1}_{\text{}_{\sqrt{H}<d\leq H}}\mathds{1}_{\text{}_{dh\leq H}}\,.
\end{equation*}
For $x>0$ define:
\begin{equation*}
\begin{split}
&\Phi^{\sin{}}_{H}\big(x;\theta\big)=x^{2q-1}\sum_{d,h}\theta\big(h,d\big)\sin{\Big(-2\pi \frac{h}{d}x^{2}+\frac{\pi}{4}\Big)}\\
&\Phi^{\cos{}}_{H}\big(x;\theta\big)=x^{2q-1}\sum_{d,h}\theta\big(h,d\big)\cos{\Big(-2\pi \frac{h}{d}x^{2}+\frac{\pi}{4}\Big)}\,.
\end{split}
\end{equation*}
Then:
\begin{equation}\label{eq:4.31}
\begin{split}
\frac{1}{X}\bigintssss\limits_{ X}^{2X}\bigg\{\Phi^{\sin{}}_{H}\Big(x;\theta\Big)\bigg\}^{2}\textit{d}x\quad,\quad\frac{1}{X}\bigintssss\limits_{ X}^{2X}\bigg\{\Phi^{\cos{}}_{H}\Big(x;\theta\Big)\bigg\}^{2}\textit{d}x\ll\,X^{2(2q-1)-2}\log^{2}{X}\,.
\end{split}
\end{equation}
\end{lem}
\begin{proof}
We have:
\begin{equation}\label{eq:4.32}
\Big|\Phi^{\sin{}}_{H}\big(x;\theta\big)\big|\quad,\quad\Big|\Phi^{\cos{}}_{H}\big(x;\theta\big)\big|\leq\sum_{D\in\mathcal{D}}\Big|\Phi_{H}\big(x;\theta;D\big)\Big|
\end{equation}
where $\mathcal{D}=\Big\{2^{k}:k\in\mathbb{N}_{0}\,,\,2^{k}\leq H\Big\}$, and
\begin{equation}\label{eq:4.33}
\Phi_{H}\big(x;\theta;D\big)=x^{2q-1}\underset{d\,\sim\,D}{\sum_{(d,h)=1}}\wp\big(h,d\big)\textit{\large{e}}\bigg(\frac{h}{d}x^{2}\bigg)\quad;\quad\wp\big(h,d\big)=\sum_{r}\theta\big(rh,rd\big)
\end{equation}
with the notation $d\sim D\Longleftrightarrow\,D\leq d<2D$. Note that
\begin{equation}\label{eq:4.34}
|\wp\big(h,d\big)|\ll\frac{1}{d^{q-\frac{3}{2}}h^{3/2}}\underset{r^{2}dh\,\leq\, H}{\underset{\sqrt{H}\,<\,rd\,\leq\, H}{\sum_{r\in\mathbb{N}}}}\frac{1}{r^{q}}
\end{equation}
and so in particular:
\begin{equation}\label{eq:4.35}
|\wp\big(h,d\big)|\ll\frac{1}{d^{q-\frac{3}{2}}h^{3/2}}\mathds{1}_{dh\leq H}\,.
\end{equation}
Applying Cauchy–Schwarz inequality in order to localize the variable $d$ to a dyadic segment $\sim D$, we have:
\begin{equation}\label{eq:4.36}
\begin{split}
&\frac{1}{X}\bigintssss\limits_{ X}^{2X}\bigg\{\sum_{D\in\mathcal{D}}\Big|\Phi_{H}\Big(x;\theta;D\Big)\Big|\bigg\}^{2}\textit{d}x\ll\log{X}\sum_{D\in\mathcal{D}}\,\,\frac{1}{X}\bigintssss\limits_{ X}^{2X}\Big|\Phi_{H}\Big(x;\theta;D\Big)\Big|^{2}\textit{d}x
\end{split}
\end{equation}
Fix $D\in\mathcal{D}$. Then by the same arguments as in the proof of Lemma 4.2 $\Big(\eqref{eq:4.27}\text{ through } \eqref{eq:4.29}\Big)$, we have:
\begin{equation}\label{eq:4.37}
\begin{split}
&\frac{1}{X}\bigintssss\limits_{ X}^{2X}\Big|\Phi_{H}\Big(x;\theta;D\Big)\Big|^{2}\textit{d}x=c_{q}X^{2(2q-1)}\underset{d\,\sim\, D}{\sum_{(d,h)=1}}\wp^{2}\big(h,d\big)+O\bigg(X^{2(2q-1)-2}\sum_{d\,\sim\,D}\frac{1}{d^{2q-5}}\bigg)\,.
\end{split}
\end{equation}
Thus, by \eqref{eq:4.36} :
\begin{equation}\label{eq:4.38}
\begin{split}
&\frac{1}{X}\bigintssss\limits_{ X}^{2X}\bigg\{\sum_{D\in\mathcal{D}}\Big|\Phi_{H}\Big(x;\theta;D\Big)\Big|\bigg\}^{2}\textit{d}x\ll\bigg\{\sum_{(d,h)=1}\wp^{2}\big(h,d\big)\bigg\}X^{2(2q-1)}\log{X}+X^{2(2q-1)-2}\log^{2}{X}\,.
\end{split}
\end{equation}
Using \eqref{eq:4.34}, the sum over the diagonal terms satisfies the bound:
\begin{equation}\label{eq:4.39}
\begin{split}
0\leq\sum_{(d,h)=1}\wp^{2}\big(h,d\big)&\ll\underset{\sqrt{H}\,<\,r_{2}d\,\leq\,H}{\sum_{\sqrt{H}\,<\,r_{1}d\,\leq\,H}}\frac{1}{r_{1}^{q}r_{2}^{q}}\frac{1}{d^{2q-3}}\,\underset{1\,\leq\,r_{2}^{2}dh\,\leq\,H}{\sum_{1\,\leq\,r_{1}^{2}dh\,\leq\,H}}\frac{1}{h^{3}}\ll\\
&\ll\sum_{1\,\leq\,r_{1},r_{2}\,\leq\,\sqrt{H}}\frac{1}{r_{1}^{q}r_{2}^{q}}\,\sum_{d\,>\sqrt{H/r_{1}r_{2}}}\,\,\frac{1}{d^{2q-3}}
\ll H^{2-q}\ll\frac{1}{X^{2}}
\end{split}
\end{equation}
By \eqref{eq:4.32} this concludes the proof.
\end{proof}
\subsection{Bounding the remainder terms}
We have everything we need in order to treat the remainder terms in the approximate expression for $\mathcal{E}_{q}(x)$. Before proceeding to the proof of the main results of this subsection, we need the following simple lemma.   
\begin{lem}
Let $H\geq1$. Suppose $r,d,m\in\mathbb{N}$, and that $1\leq m\leq Y$ for some $Y$. Then:
\begin{equation}\label{eq:4.40}
\begin{split}
|\mathfrak{a}^{\ast}_{H}\big(r^{2}m,rd\big)|\,\,,\,\,|\mathfrak{d}^{\ast}_{H}\big(r^{2}m,rd\big)|&\ll\frac{Y^{1/2}}{r^{1/2}m^{3/4}H}\bigg\{\underset{a\neq0\,,\,b\equiv0(d)}{\sum_{a^{2}+b^{2}=m}}1\bigg\}\mathds{1}_{r^{2}m\leq2H^{2}}\,.
\end{split}
\end{equation}
In particular, $\mathfrak{a}^{\ast}_{H}$ and $\mathfrak{d}^{\ast}_{H}$ satisfy condition \textit{(C.2)} in Lemma 4.1. 
\end{lem}
\begin{proof}
Let $r,d,m\in\mathbb{N}$ be as above. Recalling the definition of $\tau^{\ast}$ and noting that $\hat{\mathfrak{g}}(0)=0$, we have:
\begin{equation*}
\begin{split}
&\underset{n\,\equiv\,0\,(rd)}{\underset{0\,\leq\,n\,\leq\,h}{\underset{1\,\leq\, h\,\leq\, H}{\sideset{}{''}\sum_{n^{2}+h^{2}=r^{2}m}}}}\tau^{\ast}\bigg(\frac{h}{[H]+1}\bigg)\mathfrak{g}\bigg(\frac{n}{r\sqrt{m}}\bigg)+\underset{h\,\equiv\,0\,(rd)}{\underset{0\,\leq\,n\,\leq\,h}{\underset{1\,\leq\, h\,\leq\, H}{\sideset{}{''}\sum_{n^{2}+h^{2}=r^{2}m}}}}\tau^{\ast}\bigg(\frac{h}{rd[H/rd]+rd}\bigg)\hat{\mathfrak{g}}\bigg(\frac{n}{r\sqrt{m}}\bigg)=\\
&=\underset{n\,\equiv\,0\,(d)}{\underset{0\,\leq\,n\,\leq\,h}{\underset{1\,\leq\, rh\,\leq\, H}{\sideset{}{''}\sum_{n^{2}+h^{2}=m}}}}\tau^{\ast}\bigg(\frac{rh}{[H]+1}\bigg)\mathfrak{g}\bigg(\frac{n}{\sqrt{m}}\bigg)+\underset{h\,\equiv\,0\,(d)}{\underset{0\,\leq\,n\,\leq\,h}{\underset{1\,\leq\, rh\,\leq\, H}{\sideset{}{''}\sum_{n^{2}+h^{2}=m}}}}\tau^{\ast}\bigg(\frac{h}{d[H/rd]+d}\bigg)\hat{\mathfrak{g}}\bigg(\frac{n}{\sqrt{m}}\bigg)\leq\\
&\leq\frac{rY^{1/2}}{H}\bigg\{\underset{n\,\equiv\,0\,(d)}{\underset{0\,\leq\,n\,\leq\,h}{\underset{1\,\leq\, rh\,\leq\, H}{\sideset{}{''}\sum_{n^{2}+h^{2}=m}}}}\mathfrak{g}\bigg(\frac{n}{\sqrt{m}}\bigg)+\underset{h\,\equiv\,0\,(d)}{\underset{0\,\leq\,n\,\leq\,h}{\underset{1\,\leq\, rh\,\leq\, H}{\sideset{}{''}\sum_{n^{2}+h^{2}=m}}}}\hat{\mathfrak{g}}\bigg(\frac{n}{\sqrt{m}}\bigg)\bigg\}\leq\frac{rY^{1/2}}{H}\bigg\{\,\,\,\underset{a\neq0\,,\,b\equiv0(d)}{\sum_{a^{2}+b^{2}=m}}1\,\,\bigg\}\,.
\end{split}
\end{equation*}
By the construction of $\mathfrak{a}^{\ast}$ and $\mathfrak{d}^{\ast}_{H}$ we deduce \eqref{eq:4.40}.
\end{proof}
\text{ }\\
Now we can state the main results.
\begin{prop}
Let $X>0$ large, and set $H=X^{2}/2$. Then:\\\\
\textit{(1)} For $q\equiv0\,(2)$
\begin{equation}\label{eq:4.41}
\frac{1}{X}\bigintssss\limits_{ X}^{2X}\Bigg\{x^{2q-1}\sum_{1\,\leq\,d\,\leq\,\sqrt{H}}\frac{|\xi(d)|}{d^{q-3/2}}\sum_{m}\mathfrak{a}^{\ast}_{H}\big(m,d\big)\cos{\Big(-2\pi \frac{\sqrt{m}}{d}x^{2}+\frac{\pi}{4}\Big)}\Bigg\}^{2}\textit{d}x\ll\,X^{2(2q-1)-2}\log^{2}{X}\,. 
\end{equation}
\textit{(2)} For $q\equiv1\,(2)$
\begin{equation}\label{eq:4.42}
\frac{1}{X}\bigintssss\limits_{ X}^{2X}\Bigg\{x^{2q-1}\sum_{1\,\leq\,d\,\leq\,\sqrt{H}}\frac{1}{d^{q-3/2}}\sum_{m}\mathfrak{d}^{\ast}_{H}\big(m,d\big)\cos{\Big(-2\pi \frac{\sqrt{m}}{d}x^{2}+\frac{\pi}{4}\Big)}\Bigg\}^{2}\textit{d}x\ll\,X^{2(2q-1)-2}\log^{2}{X}\,. 
\end{equation}
\end{prop}
\begin{proof}
We shall appeal to Lemma 4.1 for $q>3$, and to Lemma 4.2 for $q=3$, where $\eta$ and $\beta$ are given by:
\begin{equation*}
\begin{split}
&\eta(d)=\frac{|\xi(d)|}{d^{q-3/2}}\mathds{1}_{\text{}_{d\leq\sqrt{H}}}\quad\,\,\,,\quad\beta\big(m,d\big)=\mathfrak{a}^{\ast}_{H}\big(m,d\big)\qquad;\qquad q\equiv0\,(2)\\
&\eta(d)=\frac{1}{d^{q-3/2}}\mathds{1}_{\text{}_{d\leq\sqrt{H}}}\quad\,\,\,,\quad\beta\big(m,d\big)=\mathfrak{d}^{\ast}_{H}\big(m,d\big)\qquad;\qquad q\equiv1\,(2)\,.
\end{split}
\end{equation*}
\textit{(I)} Suppose $q>3$. We appeal to Lemma 4.1 with $\nu(d),\alpha\big(m,d\big)\equiv0$. First we consider the error term in \eqref{eq:4.1}. Since for $q>3$:
\begin{equation*}
\begin{split}
&\Bigg(\sum_{1\,\leq\,d\,\leq\,\sqrt{H}}\frac{1}{d^{q-\frac{5}{2}}}\Bigg)^{2}+\Bigg(\sum_{1\,\leq\,d\,\leq\,\sqrt{H}}\frac{\varpi(d)}{d^{q-2}}\Bigg)^{2}+\Bigg(\sum_{1\,\leq\,d\,\leq\,\sqrt{H}}\frac{1}{d^{q-\frac{5}{2}}}\Bigg)\Bigg(\sum_{1\,\leq\,d\,\leq\,\sqrt{H}}\frac{\varpi(d)}{d^{q-2}}\Bigg)\ll1
\end{split}
\end{equation*}
we have by \eqref{eq:4.2} that the error term is $\ll X^{2(2q-1)-2}\log^{2}{X}$ for $q\equiv0,1\,(2)$. It remains to evaluate the sum over the diagonal terms. Let $r,d,\ell,m\in\mathbb{N}$ be integers such that $1\leq m^{2}\ell\leq2H^{2}$. Then taking $Y=m^{2}\ell$ in Lemma 4.4, we have by \eqref{eq:4.40} :
\begin{equation}\label{eq:4.43}
\begin{split}
|\mathfrak{a}^{\ast}_{H}\big((rm)^{2}\ell,rd\big)|\,\,,\,\,|\mathfrak{d}^{\ast}_{H}\big((rm)^{2}\ell,rd\big)|&\ll\frac{1}{r^{1/2}(m^{2}\ell)^{1/4}H}\bigg\{\underset{a\neq0\,,\,b\equiv0(d)}{\sum_{a^{2}+b^{2}=m^{2}\ell}}1\bigg\}\mathds{1}_{(rm)^{2}\ell\leq2H^{2}}\leq\\
&\leq\frac{r_{2}\big(m^{2}l\big)}{r^{1/2}(m^{2}\ell)^{1/4}H}\mathds{1}_{m^{2}\ell\leq2H^{2}}\,.
\end{split}
\end{equation}
Thus, the sum over the diagonal terms $\Xi_{H}\big(\eta,\beta\big)$ for $q\equiv0,1\,(2)$ satisfies the bound:
\begin{equation}\label{eq:4.44}
\begin{split}
0\leq\Xi_{H}\big(\eta,\beta\big)&=\underset{1\,\leq\,m^{2}\ell\,\leq\,2H^{2}}{\underset{(d,m)=1}{\sum_{\ell,d,m}}}|\mu(\ell)|\bigg(\sum_{r}\eta(rd)\beta\Big((rm)^{2}\ell,rd\Big)\bigg)^{2}\ll\frac{1}{H^{2}}\sum_{1\,\leq\,m^{2}\ell\,\leq\,2H^{2}}|\mu(\ell)|\frac{r^{2}_{2}\big(m^{2}\ell\big)}{(m^{2}\ell)^{1/2}}=\\
&=\frac{1}{H^{2}}\sum_{1\,\leq\,m\,\leq\,2H^{2}}\frac{r^{2}_{2}(m)}{m^{1/2}}\ll\,\frac{\log{X}}{X^{2}}.
\end{split}
\end{equation}
This settles the proof for $q>3$.
\text{ }\\\\
\textit{(II)} Suppose $q=3$. We appeal to Lemma 4.2. We only need to evaluate the sum over the diagonal terms $\Upsilon_{1}\big(\eta,\beta\big)$ and $\Upsilon_{2}\big(\eta,\beta\big)$. Appealing to Lemma 4.4 again, we have:
\begin{equation}\label{eq:4.45}
\begin{split}
&0\leq\Upsilon_{1}\big(\eta,\beta\big)=\underset{1\,\leq\,m\,\leq\,2H^{2}}{\sum_{d,m}}|\eta(d)|\beta^{2}\big(m,d\big)\ll\frac{1}{H^{2}}\sum_{1\,\leq\,m\,\leq\,2H^{2}}\frac{r^{2}_{2}(m)}{
m^{1/2}}\ll\,\frac{\log{X}}{X^{2}}
\end{split}
\end{equation}
and
\begin{equation}\label{eq:4.46}
\begin{split}
&0\leq\Upsilon_{2}\big(\eta,\beta\big)=\underset{1\,\leq\,m\,\leq\,\sqrt{2}H}{\sum_{(d,m)=1}}\bigg(\sum_{r}\eta(rd)\beta\big((rm)^{2},rd\big)\bigg)^{2}\ll\frac{1}{H^{2}}\sum_{1\,\leq\,m\,\leq\,\sqrt{2}H}\frac{r^{2}_{2}\big(m^{2}\big)}{
m}\ll\,\frac{\log{X}}{X^{2}}\,.
\end{split}
\end{equation}
This completes the proof.
\end{proof}
\begin{prop}
Let $X>0$ large, and set $H=X^{2}/2$. Then:\\\\
\begin{equation}\label{eq:4.47}
\frac{1}{X}\bigintssss\limits_{ X}^{2X}\big|\Theta^{H}_{q,\chi}(x)\big|^{2}\textit{d}x\ll\,X^{2(2q-1)-2}\log^{2}{X}
\end{equation}
\begin{equation}\label{eq:4.48}
\frac{1}{X}\bigintssss\limits_{ X}^{2X}\big|\Theta^{H,\chi}_{q}(x)\big|^{2}\textit{d}x\ll\,X^{2(2q-1)-2}\log^{2}{X}
\end{equation}
\begin{equation}\label{eq:4.49}
\frac{1}{X}\bigintssss\limits_{ X}^{2X}\big|\Theta^{H}_{q}(x)\big|^{2}\textit{d}x\ll\,X^{2(2q-1)-2}\log^{2}{X}\,.
\end{equation}
\end{prop}
\begin{proof}
Apply Lemma 4.3 to \eqref{eq:4.47} with:
\begin{equation*}
\theta\big(h,d\big)=-2\varrho_{\chi,q}\frac{2^{q-2}\chi(d)}{d^{q-3/2}}\frac{1}{h^{3/2}}\tau\bigg(\frac{h}{[H/d]+1}\bigg)\mathds{1}_{\sqrt{H}<d\leq H}\mathds{1}_{dh\leq H}
\end{equation*}
to \eqref{eq:4.48} with:
\begin{equation*}
\quad\qquad\qquad\theta\big(h,d\big)=-2\varrho_{\chi,q}\frac{(-1)^{\frac{q+1}{2}}4^{q-1}}{d^{q-3/2}}\frac{\chi(-h)}{h^{3/2}}\tau\bigg(\frac{h}{[H/d]+1}\bigg)\mathds{1}_{d\equiv0(4)}\mathds{1}_{\sqrt{H}<d\leq H}\mathds{1}_{dh\leq H}
\end{equation*}
and to \eqref{eq:4.49} with:
\begin{equation*}
\qquad\theta\big(h,d\big)=-2\varrho_{\chi,q}\frac{1}{d^{q-3/2}}\lambda^{\ast}(h,d)\frac{1}{h^{3/2}}\tau^{\ast}\bigg(\frac{h}{[H/d]+1}\bigg)\mathds{1}_{\sqrt{H}<d\leq H}\mathds{1}_{dh\leq H}\,.
\end{equation*}
\end{proof}
\subsection{Estimating in mean-square the leading terms in the approximate expression for $\mathcal{E}_{q}(x)$}
Before turning to the asymptotic estimate, we need to remove the dependency of $\mathfrak{a}_{H}$ and $\mathfrak{a}_{H,\chi}$ with respect to the parameter $H$. The following lemma shows that for integers $m\in\mathbb{N}$ which are of moderate size with respect to the parameter $H$, $m^{-3/4}\mathfrak{a}_{H}\big(m,d\big)$ and $m^{-3/4}\mathfrak{a}_{H,\chi}\big(m,d\big)$ can be approximated by $r_{2}\big(m,d;q\big)$ and $r_{2,\chi}\big(m,d;q\big)$ with a reasonable error.    
\begin{lem}
Let $H\geq1$, $r,d,m\in\mathbb{N}$. Then:\\\\
\textit{(I)}
\begin{equation}\label{eq:4.50}
\big|\mathfrak{a}_{H}\big(r^{2}m,rd\big)\big|\,\,,\,\,\big|\mathfrak{a}_{H,\chi}\big(r^{2}m,rd\big)\big|\ll\frac{r_{2}\big(m,d;q\big)}{r^{3/2}m^{3/4}}\mathds{1}_{r^{2}m\leq2H^{2}}\,.
\end{equation}
\textit{(II)} Suppose $1\leq rd\leq H$ , $1\leq r^{2}m\leq H^{2}$ and that $1\leq m\leq Y$ for some Y. Then:\\
\begin{equation}\label{eq:4.51}
\mathfrak{a}_{H}\big(r^{2}m,rd\big)=\frac{r_{2}\big(m,d;q\big)}{4\pi r^{3/2}m^{3/4}}+O\bigg(\frac{r_{2}\big(m,d;q\big)r^{1/2}Y}{m^{3/4}H^{2}}\bigg)
\end{equation}
\begin{equation}\label{eq:4.52}
\qquad\quad\,\,\mathfrak{a}_{H,\chi}\big(r^{2}m,rd\big)=-\chi(r)\frac{r_{2,\chi}\big(m,d;q\big)}{2\pi r^{3/2} m^{3/4}}+O\bigg(\frac{r_{2}\big(m,d;q\big)r^{1/2}Y}{m^{3/4}H^{2}}\bigg)\,.
\end{equation}
In particular, $\mathfrak{a}_{H}$ and $\mathfrak{a}_{H,\chi}$ satisfy condition \textit{(C.2)} in Lemma 4.1. 
\end{lem}
\begin{proof}
Let $r,d,m\in\mathbb{N}$ be integers. Clearly we may assume that $r^{2}m\leq 2H^{2}$, since otherwise $\mathfrak{a}_{H}\big(r^{2}m,rd\big),\mathfrak{a}_{H,\chi}\big(r^{2}m,rd\big)=0$. By the same arguments as in the proof of Lemma 4.4 we have:
\begin{equation*}
\begin{split}
&\underset{n\,\equiv\,0\,(rd)}{\underset{0\,\leq\,n\,\leq\,h}{\underset{1\,\leq\, h\,\leq\, H}{\sideset{}{''}\sum_{n^{2}+h^{2}=r^{2}m}}}}\tau\bigg(\frac{h}{[H]+1}\bigg)\mathfrak{g}\Big(\frac{n}{r\sqrt{m}}\bigg)+\underset{h\,\equiv\,0\,(rd)}{\underset{0\,\leq\,n\,\leq\,h}{\underset{1\,\leq\, h\,\leq\, H}{\sideset{}{''}\sum_{n^{2}+h^{2}=r^{2}m}}}}\tau\bigg(\frac{h}{rd[H/rd]+rd}\bigg)\hat{\mathfrak{g}}\bigg(\frac{n}{r\sqrt{m}}\bigg)=\\
&=\underset{n\,\equiv\,0\,(d)}{\underset{0\,\leq\,n\,\leq\,h}{\underset{1\,\leq\, rh\,\leq\, H}{\sideset{}{''}\sum_{n^{2}+h^{2}=m}}}}\tau\bigg(\frac{rh}{[H]+1}\bigg)\mathfrak{g}\bigg(\frac{n}{\sqrt{m}}\bigg)+\underset{h\,\equiv\,0\,(d)}{\underset{0\,\leq\,n\,\leq\,h}{\underset{1\,\leq\,rh\,\leq\, H}{\sideset{}{''}\sum_{n^{2}+h^{2}=m}}}}\tau\bigg(\frac{h}{d[H/rd]+d}\bigg)\hat{\mathfrak{g}}\bigg(\frac{n}{\sqrt{m}}\bigg)\,.
\end{split}
\end{equation*}
and:
\begin{equation*}
\begin{split}
&\underset{n\,\equiv\,0\,(rd)}{\underset{0\,\leq\,n\,\leq\,h}{\underset{1\,\leq\, h\,\leq\, H}{\sideset{}{''}\sum_{n^{2}+h^{2}=r^{2}m}}}}\chi(-h)\tau\bigg(\frac{h}{[H]+1}\bigg)\mathfrak{g}\bigg(\frac{n}{r\sqrt{m}}\bigg)+\underset{h\,\equiv\,0\,(rd)}{\underset{0\,\leq\,n\,\leq\,h}{\underset{1\,\leq\, h\,\leq\, H}{\sideset{}{''}\sum_{n^{2}+h^{2}=r^{2}m}}}}\chi(-n)\tau\bigg(\frac{h}{rd[H/rd]+rd}\bigg)\hat{\mathfrak{g}}\bigg(\frac{n}{r\sqrt{m}}\bigg)=\\
&=-\chi(r)\Bigg\{\,\,\underset{n\,\equiv\,0\,(d)}{\underset{0\,\leq\,n\,\leq\,h}{\underset{1\,\leq\, rh\,\leq\, H}{\sideset{}{''}\sum_{n^{2}+h^{2}=m}}}}\chi(h)\tau\bigg(\frac{rh}{[H]+1}\bigg)\mathfrak{g}\bigg(\frac{n}{\sqrt{m}}\bigg)+\underset{h\,\equiv\,0\,(d)}{\underset{0\,\leq\,n\,\leq\,h}{\underset{1\,\leq\, rh\,\leq\, H}{\sideset{}{''}\sum_{n^{2}+h^{2}=m}}}}\chi(n)\tau\bigg(\frac{h}{d[H/rd]+d}\bigg)\hat{\mathfrak{g}}\bigg(\frac{n}{\sqrt{m}}\bigg)\Bigg\}\,.
\end{split}
\end{equation*}
Note that if $m=n^{2}+h^{2}$ for some integers $n\in\mathbb{N}_{0}$ and $h\in\mathbb{N}$, then by using the simple fact that $\mathfrak{g}\Big(\frac{n}{\sqrt{m}}\Big)=\hat{\mathfrak{g}}\Big(\frac{h}{\sqrt{m}}\Big)$ we have:
\begin{equation*}
\begin{split}
&\underset{h\,\geq\, n\,;\,n\equiv0(d)}{\underset{n\,\in\mathbb{N}_{0}\,,\,h\,\in\mathbb{N}}{\sideset{}{''}\sum_{n^{2}+h^{2}=m}}}\mathfrak{g}\bigg(\frac{n}{\sqrt{m}}\bigg)+\underset{h\,\geq\, n\,;\,h\equiv0(d)}{\underset{n\,\in\mathbb{N}_{0}\,,\,h\,\in\mathbb{N}}{\sideset{}{''}\sum_{n^{2}+h^{2}=m}}}\hat{\mathfrak{g}}\bigg(\frac{n}{\sqrt{m}}\bigg)=
\frac{1}{4}\underset{b\equiv0(d)}{\sum_{a^{2}+b^{2}=m}}\hat{\mathfrak{g}}\bigg(\frac{|a|}{\sqrt{m}}\bigg)=\frac{1}{4}r_{2}\big(m,d;q\big)\\
&\underset{h\,\geq\, n\,;\,n\equiv0(d)}{\underset{n\,\in\mathbb{N}_{0}\,,\,h\,\in\mathbb{N}}{\sideset{}{''}\sum_{n^{2}+h^{2}=m}}}\chi(h)\mathfrak{g}\bigg(\frac{n}{\sqrt{m}}\bigg)+\underset{h\,\geq\, n\,;\,h\equiv0(d)}{\underset{n\,\in\mathbb{N}_{0}\,,\,h\,\in\mathbb{N}}{\sideset{}{''}\sum_{n^{2}+h^{2}=m}}}\chi(n)\hat{\mathfrak{g}}\bigg(\frac{n}{\sqrt{m}}\bigg)=\frac{1}{4}\underset{b\equiv0(d)}{\sum_{a^{2}+b^{2}=m}}\chi\big(|a|\big)\hat{\mathfrak{g}}\bigg(\frac{|a|}{\sqrt{m}}\bigg)=\frac{1}{4}r_{2,\chi}\big(m,d;q\big)\,.
\end{split}
\end{equation*}
Since  $\tau(t)=\pi^{-1}+O\big(t^{2}\big)$ uniformly on the segment $[0,1]$, we obtain using the above results:
\begin{equation}\label{eq:4.53}
\begin{split}
&\mathfrak{a}_{H}\big(r^{2}m,rd\big)=\frac{1}{4\pi r^{3/2}m^{3/4}}\Bigg\{\,\,\underset{n\,\equiv\,0\,(d)}{\underset{0\,\leq\,n\,\leq\,h}{\underset{1\,\leq\, rh\,\leq\, H}{\sideset{}{''}\sum_{n^{2}+h^{2}=m}}}}\mathfrak{g}\bigg(\frac{n}{\sqrt{m}}\bigg)+\underset{h\,\equiv\,0\,(d)}{\underset{0\,\leq\,n\,\leq\,h}{\underset{1\,\leq\,rh\,\leq\, H}{\sideset{}{''}\sum_{n^{2}+h^{2}=m}}}}\hat{\mathfrak{g}}\bigg(\frac{n}{\sqrt{m}}\bigg)+O\bigg(\frac{r_{2}\big(m,d;q\big)r^{2}m}{H^{2}}\bigg)\Bigg\}\\
&\mathfrak{a}_{H,\chi}\big(r^{2}m,rd\big)=-\frac{\chi(r)}{2\pi r^{3/2} m^{3/4}}\Bigg\{\,\,\underset{n\,\equiv\,0\,(d)}{\underset{0\,\leq\,n\,\leq\,h}{\underset{1\,\leq\, rh\,\leq\, H}{\sideset{}{''}\sum_{n^{2}+h^{2}=m}}}}\chi(h)\mathfrak{g}\bigg(\frac{n}{\sqrt{m}}\bigg)+\underset{h\,\equiv\,0\,(d)}{\underset{0\,\leq\,n\,\leq\,h}{\underset{1\,\leq\,rh\,\leq\, H}{\sideset{}{''}\sum_{n^{2}+h^{2}=m}}}}\chi(n)\hat{\mathfrak{g}}\bigg(\frac{n}{\sqrt{m}}\bigg)+O\bigg(\frac{r_{2}\big(m,d;q\big)r^{2}m}{H^{2}}\bigg)\Bigg\}\,.
\end{split}
\end{equation}
Dropping the restriction $rh\leq H$ in \eqref{eq:4.53} we obtain \eqref{eq:4.50}. Now if $r,d,m$ satisfy the assumptions in \textit{(II)} then the restriction $rh\leq H$ in\eqref{eq:4.53} is redundant, hence we obtain \eqref{eq:4.51} and \eqref{eq:4.52}.
\end{proof}
\text{ }\\
With Lemma 4.5 at our disposal, we now proceed to the evaluate the leading terms in the mean square estimate of $\mathcal{E}_{q}(x)$.
\begin{prop}
Let $X>0$ be large, and set $H=X^{2}/2$. Define:
\begin{equation*}
\begin{split}
\left.\begin{aligned}
&\nu(d)=-2\varrho_{q}\frac{\xi(d)}{d^{q-3/2}}\mathds{1}_{d\leq\sqrt{H}}\,\,\quad\qquad\quad\quad\qquad\,\,\alpha\big(m,d\big)=\mathfrak{a}_{H}\big(m,d\big)
\end{aligned}\right\}\,q\equiv0\,(2)\\
\left.\begin{aligned}
&\nu(d)=-2\varrho_{\chi,q}\frac{2^{q-1}\chi(d)}{d^{q-3/2}}\mathds{1}_{d\leq\sqrt{H}}\,\,\,\qquad\qquad\qquad\alpha\big(m,d\big)=\mathfrak{a}_{H}\big(m,d\big)\\
&\eta(d)=-2\varrho_{\chi,q}\frac{(-1)^{\frac{q+1}{2}}4^{q-1}}{d^{q-3/2}}\mathds{1}_{d\equiv0(4)}\mathds{1}_{d\leq\sqrt{H}}\quad\quad\beta\big(m,d\big)=\mathfrak{a}_{H,\chi}\big(m,d\big)
\end{aligned}\right\}\,q\equiv1\,(2)
\end{split}
\end{equation*}
Then, with the notation as in Lemma 4.1, we have:\\\\
\textit{(1)} For $q\equiv0\,(2)$
\begin{equation}\label{eq:4.54}
\begin{split}
\frac{1}{X}\bigintssss\limits_{ X}^{2X}\bigg\{\Psi^{\sin{}}_{H}\Big(x;\nu,\alpha\Big)\bigg\}^{2}\textit{d}x&=\gamma_{q}\Bigg\{\underset{\,\,(d,2m)=1}{\sum_{d,m=1}^{\infty}}\frac{r^{2}_{2}\big(m,d;q\big)}{m^{3/2}d^{2q-3}}+2^{2q}\underset{d\equiv0(4)}{\underset{\,\,(d,m)=1}{\sum_{d,m=1}^{\infty}}}\frac{r^{2}_{2}\big(m,d;q\big)}{m^{3/2}d^{2q-3}}\Bigg\}X^{2(2q-1)}+O\Big( X^{2(2q-1)-1}\log{X}\Big)\,.
\end{split}
\end{equation}
\textit{(2)} For $q\equiv1\,(2)$
\begin{equation}\label{eq:4.55}
\begin{split}
&\frac{1}{X}\bigintssss\limits_{ X}^{2X}\bigg\{\Psi^{\sin{}}_{H}\Big(x;\nu,\alpha\Big)+\Psi^{\cos{}}_{H}\Big(x;\eta,\beta\Big)\bigg\}^{2}\textit{d}x=\\
&=\gamma_{q}\Bigg\{\underset{\,\,(d,2m)=1}{\sum_{d,m=1}^{\infty}}\frac{r^{2}_{2}\big(m,d;q\big)}{m^{3/2}d^{2q-3}}+2^{2q}\underset{d\equiv0(4)}{\underset{\,\,(d,m)=1}{\sum_{d,m=1}^{\infty}}}\frac{r^{2}_{2,\chi}\big(m,d;q\big)}{m^{3/2}d^{2q-3}}\Bigg\}X^{2(2q-1)}+O\Big( X^{2(2q-1)-1}\log{X}\Big)\,.
\end{split}
\end{equation}
\end{prop}
\begin{proof}
We shall appeal to Lemma 4.1.\\\\
\textit{(1)} Suppose $q\equiv0\,(2)$. Then by \eqref{eq:4.1} with $\eta,\beta\equiv0$, we have
\begin{equation}\label{eq:4.56}
\begin{split}
\frac{1}{X}\bigintssss\limits_{ X}^{2X}\bigg\{\Psi^{\sin{}}_{H}\Big(x;\nu,\alpha\Big)\bigg\}^{2}\textit{d}x&=\frac{c_{q}}{2}\,\Xi_{H}\big(\nu,\alpha\big)X^{2(2q-1)}+X^{2(2q-1)-2}\mathscr{E}^{H}\big(\nu,\alpha;\eta,\beta\big)
\end{split}
\end{equation}
and since $q\geq4$, by \eqref{eq:4.2} we find that:
\begin{equation}\label{eq:4.57}
\begin{split}
\big|\mathscr{E}^{H}\big(\nu,\alpha;\eta,\beta\big)\big|&\ll \log^{2}{X}\,.
\end{split}
\end{equation}
We now proceed to evaluate the sum over the diagonal terms. If either $m^{2}\ell>H$, $d>H^{1/4}$ or $r>H^{1/4}$ then we appeal to \eqref{eq:4.50} in Lemma 4.5 obtaining:
\begin{equation}\label{eq:4.58}
\begin{split}
\Xi_{H}\big(\nu,\alpha\big)&=4\varrho^{2}_{q}\underset{1\,\leq\,d\,\leq\,H^{1/4}}{\underset{1\,\leq\,m^{2}\ell\,\leq\,H}{\underset{(d,m)=1}{\sum_{\ell,d,m}}}}\frac{|\mu(\ell)|}{d^{2q-3}}\Bigg\{\,\,\underset{1\,\leq\,r\,\leq\,H^{1/4}}{\sum_{1\,\leq\,rd\,\leq\,H^{1/2}}}\,\,\frac{\xi(rd)}{r^{q-3/2}}\mathfrak{a}_{H}\big((rm)^{2}\ell,rd\big)\Bigg\}^{2}+O\Big(X^{-1}\log{X}\Big)=\\
&=4\varrho^{2}_{q}\underset{1\,\leq\,d\,\leq\,H^{1/4}}{\underset{1\,\leq\,m^{2}\ell\,\leq\,H}{\underset{(d,m)=1}{\sum_{\ell,d,m}}}}\frac{|\mu(\ell)|}{d^{2q-3}}\Bigg\{\,\,\sum_{1\,\leq\,r\,\leq\,H^{1/4}}\,\,\frac{\xi(rd)}{r^{q-3/2}}\mathfrak{a}_{H}\big((rm)^{2}\ell,rd\big)\Bigg\}^{2}+O\Big(X^{-1}\log{X}\Big)
\end{split}
\end{equation}
where we have dropped the restriction $1\leq rd\leq\ H^{1/2}$ since $1\leq r,d\leq H^{1/4}$. Fix integers $d,m,\ell\in\mathbb{N}$ in the above ranges. Then for $1\leq r\leq H^{1/4}$ we may appeal to \eqref{eq:4.51} in Lemma 4.5 with $Y=H$ obtaining:
\begin{equation}\label{eq:4.59}
\begin{split}
&\sum_{1\,\leq\,r\,\leq\,H^{1/4}}\,\,\frac{\xi(rd)}{r^{q-3/2}}\mathfrak{a}_{H}\big((rm)^{2}\ell,rd\big)=\frac{r_{2}\big(m^{2}\ell,d;q\big)}{4\pi (m^{2}\ell)^{3/4}}\sum_{1\,\leq\,r\,\leq\,H^{1/4}}\frac{\xi(rd)}{r^{q}}+O\bigg(\frac{r_{2}\big(m^{2}\ell,d;q\big)\log{X}}{H(m^{2}\ell)^{3/4}}\bigg)=\\\\
&=\frac{r_{2}\big(m^{2}l,d;q\big)}{4\pi (m^{2}l)^{3/4}}\sum_{r=1}^{\infty}\frac{\xi(rd)}{r^{q}}+O\bigg(\frac{r_{2}\big(m^{2}l\big)}{X(m^{2}l)^{3/4}}\bigg)=\Big(1-2^{-q}\Big)\zeta(q)\hat{\xi}(d)\frac{r_{2}\big(m^{2}l,d;q\big)}{4\pi (m^{2}l)^{3/4}}+O\bigg(\frac{r_{2}\big(m^{2}l\big)}{X(m^{2}l)^{3/4}}\bigg)
\end{split}
\end{equation}
where $\hat{\xi}(d)=\mathds{1}_{d\equiv1(2)}+(-1)^{\frac{q}{2}}2^{q}\mathds{1}_{d\equiv0(4)}$. Extending the summation over $m^{2}l$ and $d$ all the way to infinity, we derive:
\begin{equation}\label{eq:4.60}
\begin{split}
\Xi_{H}\big(\nu,\alpha\big)&=\Bigg(\frac{\varrho_{q}(1-2^{-q})\zeta(q)}{2\pi}\Bigg)^{2}\underset{(d,m)=1}{\sum_{\ell,d,m=1}^{\infty}}\frac{|\mu(\ell)|r^{2}_{2}\big(m^{2}\ell,d;q\big)\Hat{\xi}^{2}(d)}{(m^{2}\ell)^{3/2}d^{2q-3}}+O\Big(X^{-1}\log{X}\Big)=\\
&=\bigg(\frac{\pi^{q-1}}{2\Gamma(q)}\bigg)^{2}\underset{(d,m)=1}{\sum_{\ell,d,m=1}^{\infty}}\frac{|\mu(\ell)|r^{2}_{2}\big(m^{2}\ell,d;q\big)\Hat{\xi}^{2}(d)}{(m^{2}\ell)^{3/2}d^{2q-3}}+O\Big(X^{-1}\log{X}\Big)\,.
\end{split}
\end{equation}
Finally, if $\ell,d,m\in\mathbb{N}$ are integers such that $(d,m)=1$, $|\mu(\ell)|=1$ and $m^{2}\ell=a^{2}+b^{2}$ with $b\equiv0(d)$, then necessarily $(d,\ell)=1$. Thus, in \eqref{eq:4.60} the condition $(d,m)=1$ is equivalent to $(d,m^{2}\ell)=1$. Hence:
\begin{equation}\label{eq:4.61}
\begin{split}
\Xi_{H}\big(\nu,\alpha\big)&=\bigg(\frac{\pi^{q-1}}{2\Gamma(q)}\bigg)^{2}\underset{(d,m)=1}{\sum_{d,m=1}^{\infty}}\frac{r^{2}_{2}\big(m,d;q\big)\Hat{\xi}^{2}(d)}{m^{3/2}d^{2q-3}}+O\Big(X^{-1}\log{X}\Big)=\\
&\bigg(\frac{\pi^{q-1}}{2\Gamma(q)}\bigg)^{2}\Bigg\{\underset{\,\,(d,2m)=1}{\sum_{d,m=1}^{\infty}}\frac{r^{2}_{2}\big(m,d;q\big)}{m^{3/2}d^{2q-3}}+2^{2q}\underset{d\equiv0(4)}{\underset{\,\,(d,m)=1}{\sum_{d,m=1}^{\infty}}}\frac{r^{2}_{2}\big(m,d;q\big)}{m^{3/2}d^{2q-3}}\Bigg\}+O\Big(X^{-1}\log{X}\Big)\,.
\end{split}
\end{equation}
Inserting \eqref{eq:4.61} into \eqref{eq:4.56} and using \eqref{eq:4.57} we obtain \eqref{eq:4.54}. 
\text{ }\\\\
\textit{(2)} Suppose $q\equiv1\,(2)$. Then by \eqref{eq:4.1} we have
\begin{equation}\label{eq:4.62}
\begin{split}
&\frac{1}{X}\bigintssss\limits_{ X}^{2X}\bigg\{\Psi^{\sin{}}_{H}\Big(x;\nu,\alpha\Big)+\Psi^{\cos{}}_{H}\Big(x;\eta,\beta\Big)\bigg\}^{2}\textit{d}x=\frac{c_{q}}{2}\bigg\{\Xi_{H}\big(\nu,\alpha\big)+\Xi_{H}\big(\eta,\beta\big)\bigg\}X^{2(2q-1)}+X^{2(2q-1)-2}\mathscr{E}^{H}\big(\nu,\alpha;\eta,\beta\big)
\end{split}
\end{equation}
and since $q\geq3$ we find that:
\begin{equation}\label{eq:4.63}
\begin{split}
\big|\mathscr{E}^{H}\big(\nu,\alpha;\eta,\beta\big)\big|
\ll X\log{X}
\end{split}
\end{equation}
The sum over the diagonal terms is evaluated along the same lines as in the case $q\equiv0\,(2)$, where now we also make use of \eqref{eq:4.52} in Lemma 4.5. This concludes the proof.
\end{proof}
\subsection{Proof of Theorem 2}
We have everything in place for the proof of Theorem 2.
\begin{proof}(Theorem 2). Let $X>0$ be large, and set $H=X^{2}/2$. We keep the definition of $\nu,\,\eta$ and $\alpha,\,\beta$ as in Proposition 4.3. Then with the notations as in Lemma 4.1, we have by \eqref{eq:3.36} and \eqref{eq:3.38} in Proposition 3.1 :
\begin{equation}\label{eq:4.64}
\begin{split}
&q\equiv0\,(2)\,:\quad\mathcal{E}_{q}(x)=\Psi^{\sin{}}_{H}\big(x;\nu,\alpha\big)+\mathcal{E}^{H}_{q}(x)+O\Big(x^{2q-2}\log^{2}{x}\Big)\\
&q\equiv1\,(2)\,:\quad\mathcal{E}_{q}(x)=\Psi^{\sin{}}_{H}\big(x;\nu,\alpha\big)+\Psi^{\cos{}}_{H}\big(x;\eta,\beta\big)+\mathds{1}_{\text{}_{q=3}}\bigg\{\Theta^{H}_{q,\chi}(x)+\Theta^{H,\chi}_{q}(x)\bigg\}+
\mathcal{E}^{H}_{q}(x)+O\Big(x^{2q-2}\log^{2}{x}\Big)\,.
\end{split}
\end{equation}
First we deal with the remainder terms. By \eqref{eq:3.37} in Proposition 3.1 and \eqref{eq:4.41} in Proposition 4.1 in the case where $q\equiv0\,(2)$, and by \eqref{eq:3.39} in proposition 3.1, \eqref{eq:4.42} in Proposition 4.1 and \eqref{eq:4.47}, \eqref{eq:4.48} and \eqref{eq:4.49} in Proposition 4.2 in the case where $q\equiv1\,(2)$, we obtain after applying Cauchy–Schwarz inequality to handle the cross-terms:
\begin{equation}\label{eq:4.65}
\begin{split}
&q\equiv0\,(2)\,:\quad\frac{1}{X}\bigintssss\limits_{ X}^{2X}\Bigg\{\mathcal{E}^{H}_{q}(x)+O\Big(x^{2q-2}\log^{2}{x}\Big)\Bigg\}^{2}\textit{d}x\ll\,X^{2(2q-1)-2}\log^{4}{X}\\
&q\equiv1\,(2)\,:\quad\frac{1}{X}\bigintssss\limits_{ X}^{2X}\Bigg\{\mathds{1}_{\text{}_{q=3}}\bigg\{\Theta^{H}_{q,\chi}(x)+\Theta^{H,\chi}_{q}(x)\bigg\}+
\mathcal{E}^{H}_{q}(x)+O\Big(x^{2q-2}\log^{2}{x}\Big)\Bigg\}^{2}\textit{d}x\ll\,X^{2(2q-1)-2}\log^{4}{X}
\end{split}
\end{equation}
Squaring out and applying Cauchy–Schwarz inequality once more, we derive by Proposition 4.3:\\\\
\textit{(1)} For $q\equiv0\,(2)$
\begin{equation}\label{eq:4.66}
\begin{split}
\frac{1}{X}\bigintssss\limits_{ X}^{2X}\mathcal{E}^{2}_{q}(x)\textit{d}x&=\frac{1}{X}\bigintssss\limits_{ X}^{2X}\bigg\{\Psi^{\sin{}}_{H}\Big(x;\nu,\alpha\Big)\bigg\}^{2}\textit{d}x+O\Big(X^{2(2q-1)-1}\log^{2}{X}\Big)=\\
&=\gamma_{q}\Bigg\{\underset{\,\,(d,2m)=1}{\sum_{d,m=1}^{\infty}}\frac{r^{2}_{2}\big(m,d;q\big)}{m^{3/2}d^{2q-3}}+2^{2q}\underset{d\equiv0(4)}{\underset{\,\,(d,m)=1}{\sum_{d,m=1}^{\infty}}}\frac{r^{2}_{2}\big(m,d;q\big)}{m^{3/2}d^{2q-3}}\Bigg\}X^{2(2q-1)}+O\Big( X^{2(2q-1)-1}\log^{2}{X}\Big)\,.
\end{split}
\end{equation}
\textit{(2)} For $q\equiv1\,(2)$
\begin{equation}\label{eq:4.67}
\begin{split}
\frac{1}{X}\bigintssss\limits_{ X}^{2X}\mathcal{E}^{2}_{q}(x)\textit{d}x&=\frac{1}{X}\bigintssss\limits_{ X}^{2X}\bigg\{\Psi^{\sin{}}_{H}\Big(x;\nu,\alpha\Big)+\Psi^{\cos{}}_{H}\Big(x;\eta,\beta\Big)\bigg\}^{2}\textit{d}x+O\Big(X^{2(2q-1)-1}\log^{2}{X}\Big)=\\
&=\gamma_{q}\Bigg\{\underset{\,\,(d,2m)=1}{\sum_{d,m=1}^{\infty}}\frac{r^{2}_{2}\big(m,d;q\big)}{m^{3/2}d^{2q-3}}+2^{2q}\underset{d\equiv0(4)}{\underset{\,\,(d,m)=1}{\sum_{d,m=1}^{\infty}}}\frac{r^{2}_{2,\chi}\big(m,d;q\big)}{m^{3/2}d^{2q-3}}\Bigg\}X^{2(2q-1)}+O\Big( X^{2(2q-1)-1}\log^{2}{X}\Big)\,.
\end{split}
\end{equation}
This concludes the proof.
\end{proof}
\section{Lower bound estimates and proof of Theorem 3}
We have arrived at the final section of this paper, where we shall give the proof of Theorem 3. As we shall make use Proposition 3.2, our first and most important objective is to truncate the trigonometric sums appearing in the approximate expression for $\Delta_{q}(x)$ both in the leading and remainder terms. Once this will be done, we will be able to dispose with the remaining part of $\Delta^{H}_{q}(x)$, and remove the dependency with respect to the parameter $H$ from the truncated part of the leading terms. We shall devote the first 3 subsections for this. In the final subsection we shall give the proof of Theorem 3.
\subsection{The key lemma: A convolution argument using the Fe{j\'e}r kernel}
\begin{lem}
For a real parameter $P\geq1$, Let $F_{P}$ denote the Fe{j\'e}r kernel:
\begin{equation*}
F_{P}(w)=P\bigg(\frac{\sin{\big(\pi Pw\big)}}{\pi Pw}\bigg)^{2}\,.
\end{equation*}
Then for arbitrary real $\vartheta\neq0,\gamma\in\mathbb{R}$ it holds:
\begin{equation*}
\bigintssss\limits_{-1}^{1}F_{P}(w)\cos{\big(2\pi \vartheta w+\gamma\big)}\textit{d}w=\upphi\bigg(\frac{|\vartheta|}{P}\bigg)\cos{\gamma}+O\bigg(\frac{1}{|\vartheta|}\bigg)
\end{equation*}
where
\begin{equation*}
\upphi(y)=\left\{
        \begin{array}{ll}
            1-y& ; \,0\leq y\leq\, 1\\\\
            0 & ;\, \text{otherwise}
        \end{array}
    \right.
\end{equation*}
and the implied constant is absolute.
\end{lem}
\text{ }\\
For a proof see \cite{nowak1998fractional}. With the aid of the above lemma, we obtain the following estimate: 
\begin{lem}
Let $X>0$ be large, and set $H=X/2$. Suppose the functions $\nu,\eta:\mathbb{N}\longrightarrow\mathbb{R}$ and  $\alpha,\beta:\mathbb{N}^{2}\longrightarrow\mathbb{R}$ satisfy conditions \textit{(C.1)} and \textit{(C.2)} in Lemma 4.1. Then for any $P\geq1$ we have:
\begin{equation}\label{eq:5.1}
\begin{split}
&\bigintssss\limits_{X-1}^{X+1}F_{P}\big(x-X\big)\Bigg\{\underset{m\neq\square}{\sum_{d,m}}\nu(d)\alpha\big(m,d\big)\sin{\Big(-2\pi \frac{\sqrt{m}}{d}x+\frac{\pi}{4}\Big)}+\underset{m\neq\square}{\sum_{d,m}}\eta(d)\beta\big(m,d\big)\cos{\Big(-2\pi \frac{\sqrt{m}}{d}x+\frac{\pi}{4}\Big)}\Bigg\}\textit{d}x=\\
&=\underset{m\neq\square}{\sum_{d,m}}\nu(d)\alpha\big(m,d\big)\upphi\bigg(\frac{\sqrt{m}}{dP}\bigg)\cos{\Big(2\pi \frac{\sqrt{m}}{d}X+\frac{\pi}{4}\Big)}+\underset{m\neq\square}{\sum_{d,m}}\eta(d)\beta\big(m,d\big)\upphi\bigg(\frac{\sqrt{m}}{dP}\bigg)\cos{\Big(-2\pi \frac{\sqrt{m}}{d}X+\frac{\pi}{4}\Big)}+O\big(1\big).
\end{split}
\end{equation}
\end{lem}
\begin{proof}
Denote by $\mathcal{I}(X)$ the integral appearing on the LHS of \eqref{eq:5.1}. Making the change of variables $x\rightarrow x-X$, we have y Lemma 5.1 :
\begin{equation}\label{eq:5.2}
\begin{split}
\mathcal{I}(X)=&\underset{m\neq\square}{\sum_{d,m}}\nu(d)\alpha\big(m,d\big)\upphi\bigg(\frac{\sqrt{m}}{dP}\bigg)\cos{\Big(2\pi \frac{\sqrt{m}}{d}X+\frac{\pi}{4}\Big)}+\underset{m\neq\square}{\sum_{d,m}}\eta(d)\beta\big(m,d\big)\upphi\bigg(\frac{\sqrt{m}}{dP}\bigg)\cos{\Big(-2\pi \frac{\sqrt{m}}{d}X+\frac{\pi}{4}\Big)}+\\
&+E_{H}\big(\nu,\alpha;\eta,\beta\big)
\end{split}
\end{equation}
where:
\begin{equation}\label{eq:5.3}
\begin{split}
\big|E_{H}\big(\nu,\alpha;\eta,\beta\big)\big|&\ll\underset{m\neq\square}{\sum_{d,m}}d|\nu(d)|\frac{|\alpha\big(m,d\big)|}{m^{1/2}}+\underset{m\neq\square}{\sum_{d,m}}d|\eta(d)|\frac{|\beta\big(m,d\big)|}{m^{1/2}}\ll\underset{m\neq\square}{\underset{1\,\leq\,m\,\leq\,2H^{2}}{\sum_{1\,\leq\,d\,\leq\,\sqrt{H}}}}\frac{1}{d^{q-5/2}m^{5/4}}\bigg\{\underset{a\neq0\,,\,b\equiv0(d)}{\sum_{a^{2}+b^{2}=m}}1\bigg\}=\\
&=\underset{m\neq\square}{\underset{1\,\leq\,m\,\leq\,2H^{2}}{\sum_{1\,\leq\,d\,\leq\,\sqrt{H}}}}\frac{1}{d^{q-5/2}m^{5/4}}
\bigg\{\underset{a,b\neq0\,,\,b\equiv0(d)}{\sum_{a^{2}+b^{2}=m}}1\bigg\}\ll\sum_{1\,\leq\,d\,\leq\,\sqrt{H}}\frac{1}{d^{q-3/2}}\ll1.
\end{split}
\end{equation}
\end{proof}
\subsection{Bounding the remainder terms}
\begin{lem}
Let $X>0$ be large, and set $H=X/2$. Suppose $1\leq P\leq \sqrt{X}$. Then:\\\\
\textit{(1)} For $q\equiv0\,(2)$
\begin{equation}\label{eq:5.4}
\begin{split}
&\Bigg|\bigintssss\limits_{X-1}^{X+1}F_{P}\big(x-X\big)\Bigg\{\sum_{1\,\leq\,d\,\leq\,\sqrt{H}}\frac{|\xi(d)|}{d^{q-3/2}}\sum_{m\neq\square}\mathfrak{a}^{\ast}_{H}\big(m,d\big)\cos{\Big(-2\pi \frac{\sqrt{m}}{d}x+\frac{\pi}{4}\Big)}\Bigg\}\textit{d}x\Bigg|\ll1\,.
\end{split}
\end{equation}
\text{ }\\
\textit{(2)} For $q\equiv1\,(2)$
\begin{equation}\label{eq:5.5}
\begin{split}
&\Bigg|\bigintssss\limits_{X-1}^{X+1}F_{P}\big(x-X\big)\Bigg\{\sum_{1\,\leq\,d\,\leq\,\sqrt{H}}\frac{1}{d^{q-3/2}}\sum_{m\neq\square}\mathfrak{d}^{\ast}_{H}\big(m,d\big)\cos{\Big(-2\pi \frac{\sqrt{m}}{d}x+\frac{\pi}{4}\Big)}\Bigg\}\textit{d}x\Bigg|\ll1\,.
\end{split}
\end{equation}
\end{lem}
\begin{proof}
Set:
\begin{equation*}
\begin{split}
&\mathcal{I}_{1}(X)=\bigintssss\limits_{X-1}^{X+1}F_{P}\big(x-X\big)\Bigg\{\sum_{1\,\leq\,d\,\leq\,\sqrt{H}}\frac{|\xi(d)|}{d^{q-\frac{3}{2}}}\sum_{m\neq\square}\mathfrak{a}^{\ast}_{H}\big(m,d\big)\cos{\Big(-2\pi \frac{\sqrt{m}}{d}x+\frac{\pi}{4}\Big)}\Bigg\}\textit{d}x\\\\
&\mathcal{I}_{2}(X)=\bigintssss\limits_{X-1}^{X+1}F_{P}\big(x-X\big)\Bigg\{\sum_{1\,\leq\,d\,\leq\,\sqrt{H}}\frac{1}{d^{q-\frac{3}{2}}}\sum_{m\neq\square}\mathfrak{d}^{\ast}_{H}\big(m,d\big)\cos{\Big(-2\pi \frac{\sqrt{m}}{d}x+\frac{\pi}{4}\Big)}\Bigg\}\textit{d}x\,.
\end{split}
\end{equation*}
Appealing to Lemma 5.2 with $\nu,\alpha\equiv0$, and
\begin{equation*}
\begin{split}
&\eta(d)=\frac{|\xi(d)|}{d^{q-\frac{3}{2}}}\mathds{1}_{\text{}_{d\leq\sqrt{H}}}\quad\,\,,\quad\beta\big(m,d\big)=\mathfrak{a}^{\ast}_{H}\big(m,d\big)\qquad;\qquad q\equiv0\,(2)\\
&\eta(d)=\frac{1}{d^{q-\frac{3}{2}}}\mathds{1}_{\text{}_{d\leq\sqrt{H}}}\quad\,\,\,,\quad\beta\big(m,d\big)=\mathfrak{d}^{\ast}_{H}\big(m,d\big)\qquad;\qquad q\equiv1\,(2)\
\end{split}
\end{equation*}
we have:
\begin{equation}\label{eq:5.6}
\begin{split}
&q\equiv0\,(2)\,:\quad\big|\mathcal{I}_{1}(X)\big|\ll\sum_{1\,\leq\,d\,\leq\,\sqrt{H}}\frac{1}{d^{q-\frac{3}{2}}}\,\,\underset{m\neq\square}{\sum_{1\,\leq\,\sqrt{m}\,\leq\,dP}}\big|\mathfrak{a}^{\ast}_{H}\big(m,d\big)\big|+1\\
&q\equiv0\,(2)\,:\quad\big|\mathcal{I}_{2}(X)\big|\ll\sum_{1\,\leq\,d\,\leq\,\sqrt{H}}\frac{1}{d^{q-\frac{3}{2}}}\,\,\underset{m\neq\square}{\sum_{1\,\leq\,\sqrt{m}\,\leq\,dP}}\big|\mathfrak{d}^{\ast}_{H}\big(m,d\big)\big|+1\,.
\end{split}
\end{equation}
Fix integers $d,m\in\mathbb{N}$ in the ranges $1\leq d\leq \sqrt{H}$ and $1\leq\sqrt{m}\leq dP$. Appealing to Lemma 4.4 with $Y=d^{2}P^{2}$ we obtain by \eqref{eq:4.40} :
\begin{equation}
\begin{split}
&\big|\mathfrak{a}^{\ast}_{H}\big(m,d\big)\big|,\,\big|\mathfrak{d}^{\ast}_{H}\big(m,d\big)\big|\ll\frac{dP}{Hm^{3/4}}\bigg\{\underset{a\neq0\,,\,b\equiv0(d)}{\sum_{a^{2}+b^{2}=m}}1\bigg\}\,.
\end{split}
\end{equation}
Using this bound we find that:
\begin{equation}
\begin{split}
\big|\mathcal{I}_{1}(X)\big|\,\,,\,\,\big|\mathcal{I}_{2}(X)\big|&\ll\frac{P}{H}\sum_{1\,\leq\,d\,\leq\,\sqrt{H}}\frac{1}{d^{q-5/2}}\,\,\underset{m\neq\square}{\sum_{1\,\leq\,\sqrt{m}\,\leq\,dP}}\frac{1}{m^{3/4}}\bigg\{\underset{a\neq0\,,\,b\equiv0(d)}{\sum_{a^{2}+b^{2}=m}}1\bigg\}+1=\\
&=\frac{P}{H}\sum_{1\,\leq\,d\,\leq\,\sqrt{H}}\frac{1}{d^{q-5/2}}\,\,\underset{m\neq\square}{\sum_{1\,\leq\,\sqrt{m}\,\leq\,dP}}\frac{1}{m^{3/4}}\bigg\{\underset{a,b\neq0\,,\,b\equiv0(d)}{\sum_{a^{2}+b^{2}=m}}1\bigg\}+1\ll\frac{P^{3/2}}{H}\sum_{1\,\leq\,d\,\leq\,\sqrt{H}}\frac{1}{d^{q-2}}+1\ll\\
&\ll\frac{\log{X}}{X^{1/4}}+1\ll1\,.
\end{split}
\end{equation}
\end{proof}
\subsection{Estimating $\widehat{\Delta}_{q}(X)$ from below by an extremely short trigonometric sum}
\begin{prop}
Let $X>0$ be large. Define:
\begin{equation*}
\begin{split}
& q\equiv0\,(2)\,:\quad\widehat{\Delta}_{q}(X)=4\pi\underset{x\in[X-1,X+1]}{\sup}\,\,\big|\Delta_{q}(x)\big|\\
& q\equiv1\,(2)\,:\quad\widehat{\Delta}_{q}(X)=2^{3-q}\pi\underset{x\in[X-1,X+1]}{\sup}\,\,\big|\Delta_{q}(x)\big|\,.
\end{split}
\end{equation*}
Suppose that $1\leq P\leq \sqrt{X}$. Then:\\\\
\textit{(1)} For $q\equiv0\,(2)$
\begin{equation}\label{eq:5.9}
\begin{split}
\widehat{\Delta}_{q}(X)\geq\sum_{1\,\leq\,d\,\leq\,\sqrt{P}}\frac{\xi(d)}{d^{q-3/2}}\sum_{m}\frac{r_{2}\big(m,d;q\big)}{m^{3/4}}\upphi\bigg(\frac{\sqrt{m}}{dP}\bigg)\cos{\Big(2\pi \frac{\sqrt{m}}{d}X+\frac{\pi}{4}\Big)}+O\big(1\big)\,.
\end{split}
\end{equation}
\textit{(2)} For $q\equiv1\,(2)$
\begin{equation}\label{eq:5.10}
\begin{split}
\,\,\,\,\,\qquad\widehat{\Delta}_{q}(X)&\geq\sum_{1\,\leq\,d\,\leq\,\sqrt{P}}\frac{\chi(d)}{d^{q-3/2}}\sum_{m}\frac{r_{2}\big(m,d;q\big)}{m^{3/4}}\upphi\bigg(\frac{\sqrt{m}}{dP}\bigg)\cos{\Big(2\pi \frac{\sqrt{m}}{d}X+\frac{\pi}{4}\Big)}+\\
&+(-1)^{\frac{q-1}{2}}2^{q}\underset{d\equiv0(4)}{\sum_{1\,\leq\,d\,\leq\,\sqrt{P}}}\frac{1}{d^{q-3/2}}\sum_{m}\frac{r_{2,\chi}\big(m,d;q\big)}{m^{3/4}}\upphi\bigg(\frac{\sqrt{m}}{dP}\bigg)\cos{\Big(-2\pi \frac{\sqrt{m}}{d}X+\frac{\pi}{4}\Big)}+O\big(1\big)\,.
\end{split}
\end{equation}
\end{prop}
\begin{proof}
Define
\begin{equation*}
\begin{split}
\left.\begin{aligned}
&\nu(d)=\frac{\xi(d)}{d^{q-3/2}}\mathds{1}_{d\leq\sqrt{H}}\,\,\,\quad\qquad\quad\quad\qquad\alpha\big(m,d\big)=\mathfrak{a}_{H}\big(m,d\big)
\end{aligned}\right\}\,q\equiv0\,(2)\\
\left.\begin{aligned}
&\nu(d)=\frac{2^{q-1}\chi(d)}{d^{q-3/2}}\mathds{1}_{d\leq\sqrt{H}}\,\,\,\qquad\qquad\qquad\alpha\big(m,d\big)=\mathfrak{a}_{H}\big(m,d\big)\\
&\eta(d)=\frac{(-1)^{\frac{q+1}{2}}4^{q-1}}{d^{q-3/2}}\mathds{1}_{d\equiv0(4)}\mathds{1}_{d\leq\sqrt{H}}\quad\quad\beta\big(m,d\big)=\mathfrak{a}_{H,\chi}\big(m,d\big)
\end{aligned}\right\}\,q\equiv1\,(2)
\end{split}
\end{equation*}
\text{ }\\
and for $x\in[X-1,X+1]$ set:
\begin{equation*}
\begin{split}
& q\equiv0\,(2)\,:\quad\mathfrak{D}(x)=\underset{m\neq\square}{\sum_{d,m}}\nu(d)\alpha\big(m,d\big)\sin{\Big(-2\pi \frac{\sqrt{m}}{d}x+\frac{\pi}{4}\Big)}\\
& q\equiv1\,(2)\,:\quad\mathfrak{D}(x)=\underset{m\neq\square}{\sum_{d,m}}\nu(d)\alpha\big(m,d\big)\sin{\Big(-2\pi \frac{\sqrt{m}}{d}x+\frac{\pi}{4}\Big)}+\underset{m\neq\square}{\sum_{d,m}}\eta(d)\beta\big(m,d\big)\cos{\Big(-2\pi \frac{\sqrt{m}}{d}x+\frac{\pi}{4}\Big)}\,.
\end{split}
\end{equation*}
By Proposition 3.2 we have
\begin{equation}\label{eq:5.11}
\Delta_{q}(x)=\mathfrak{D}(x)+
\Delta^{H}_{q}(x)+O\big(1\big)
\end{equation}
where $\Delta^{H}_{q}(x)$ satisfies the bound:
\begin{equation}\label{eq:5.12}
\begin{split}
& q\equiv0\,(2)\,:\quad\big|\Delta^{H}_{q}(x)\big|\leq\sum_{1\,\leq\,d\,\leq\,\sqrt{H}}\frac{|\xi(d)|}{d^{q-3/2}}\sum_{m\neq\square}\mathfrak{a}^{\ast}_{H}\big(m,d\big)\cos{\Big(-2\pi \frac{\sqrt{m}}{d}x+\frac{\pi}{4}\Big)}+O\big(1\big)\\
& q\equiv1\,(2)\,:\quad\big|\Delta^{H}_{q}(x)\big|\leq\sum_{1\,\leq\,d\,\leq\,\sqrt{H}}\frac{1}{d^{q-3/2}}\sum_{m\neq\square}\mathfrak{d}^{\ast}_{H}\big(m,d\big)\cos{\Big(-2\pi \frac{\sqrt{m}}{d}x+\frac{\pi}{4}\Big)}+O\big(1\big)\,.
\end{split}
\end{equation}
By Lemma 5.2 and Lemma 5.3 we have for $q\equiv0\,(2)$ :
\begin{equation}\label{eq:5.13}
\begin{split}
\bigintssss\limits_{X-1}^{X+1}F_{P}\big(x-X\big)\Delta_{q}(x)\textit{d}x&=\bigintssss\limits_{X-1}^{X+1}F_{P}\big(x-X\big)\Big\{\mathfrak{D}(x)+
\Delta^{H}_{q}(x)+O\big(1\big)\Big\}\textit{d}x
=\bigintssss\limits_{X-1}^{X+1}F_{P}\big(x-X\big)\mathfrak{D}(x)\textit{d}x+O\big(1\big)=\\
&=\underset{m\neq\square}{\sum_{d,m}}\nu(d)\alpha\big(m,d\big)\upphi\bigg(\frac{\sqrt{m}}{dP}\bigg)\cos{\Big(2\pi \frac{\sqrt{m}}{d}X+\frac{\pi}{4}\Big)}+O\big(1\big)=\\
&=\sum_{1\,\leq\,d\,\leq\,\sqrt{P}}\nu(d)\underset{m\neq\square}{\sum_{m}}\alpha\big(m,d\big)\upphi\bigg(\frac{\sqrt{m}}{dP}\bigg)\cos{\Big(2\pi \frac{\sqrt{m}}{d}X+\frac{\pi}{4}\Big)}+O\big(1\big)\,.
\end{split}
\end{equation}
Here we have used the bound:
\begin{equation*}
\begin{split}
&\bigg|\sum_{\sqrt{P}\,<\,d\,\leq\,\sqrt{H}}\nu(d)\underset{m\neq\square}{\sum_{m}}\alpha\big(m,d\big)\upphi\bigg(\frac{\sqrt{m}}{dP}\bigg)\bigg|\ll\sum_{\sqrt{P}\,<\,d\,\leq\,\sqrt{H}}\frac{1}{d^{q-3/2}}\,\,\underset{m\neq\square}{\sum_{1\,\leq\,\sqrt{m}\,\leq\,dP}}\frac{1}{m^{3/4}}\bigg\{\underset{a\neq0\,,\,b\equiv0(d)}{\sum_{a^{2}+b^{2}=m}}1\bigg\}=\\
&=\sum_{\sqrt{P}\,<\,d\,\leq\,\sqrt{H}}\frac{1}{d^{q-3/2}}\,\,\underset{m\neq\square}{\sum_{1\,\leq\,\sqrt{m}\,\leq\,dP}}\frac{1}{m^{3/4}}\bigg\{\underset{a,b\neq0\,,\,b\equiv0(d)}{\sum_{a^{2}+b^{2}=m}}1\bigg\}\ll\sqrt{P}\sum_{\sqrt{P}\,<\,d\,\leq\,\sqrt{H}}\frac{1}{d^{q-1}}\ll1\,.
\end{split}
\end{equation*}
The same arguments give for $q\equiv1\,(2)$ :
\begin{equation}\label{eq:5.14}
\begin{split}
\bigintssss\limits_{X-1}^{X+1}F_{P}\big(x-X\big)\Delta_{q}(x)\textit{d}x=&\sum_{1\,\leq\,d\,\leq\,\sqrt{P}}\nu(d)\underset{m\neq\square}{\sum_{m}}\alpha\big(m,d\big)\upphi\bigg(\frac{\sqrt{m}}{dP}\bigg)\cos{\bigg(2\pi \frac{\sqrt{m}}{d}X+\frac{\pi}{4}\bigg)}+\\
&+\sum_{1\,\leq\,d\,\leq\,\sqrt{P}}\eta(d)\underset{m\neq\square}{\sum_{m}}\beta\big(m,d\big)\upphi\bigg(\frac{\sqrt{m}}{dP}\bigg)\cos{\bigg(-2\pi \frac{\sqrt{m}}{d}X+\frac{\pi}{4}\bigg)}+O\big(1\big)\,.
\end{split}
\end{equation}
Note that by the support of $\upphi$, for every integer $1\leq d\leq\sqrt{P}$ the sum over $m$ in \eqref{eq:5.13} and \eqref{eq:5.14} is restricted to the range $1\leq\sqrt{m}\leq dP$. Fix integers $d,m\in\mathbb{N}$ in the above ranges. Appealing to Lemma 4.5 with $Y=d^{2}P^{2}$ we have by \eqref{eq:4.51} and \eqref{eq:4.52} :
\begin{equation}\label{eq:5.15}
\begin{split}
&\mathfrak{a}_{H}\big(m,d\big)=\frac{r_{2}\big(m,d;q\big)}{4\pi m^{3/4}}+O\Bigg(\frac{r_{2}\big(m,d;q\big)d^{2}P^{2}}{m^{3/4}H^{2}}\Bigg)\\
&\mathfrak{a}_{H,\chi}\big(m,d\big)=-\frac{r_{2,\chi}\big(m,d;q\big)}{2\pi m^{3/4}}+O\Bigg(\frac{r_{2}\big(m,d;q\big)d^{2}P^{2}}{m^{3/4}H^{2}}\Bigg)\,.
\end{split}
\end{equation}
Using the bound
\begin{equation*}
\begin{split}
&\bigg(\frac{P}{H}\bigg)^{2}\sum_{1\,\leq\,d\,\leq\,\sqrt{P}}\frac{1}{d^{q-7/2}}\,\,\underset{m\neq\square}{\sum_{1\,\leq\,\sqrt{m}\,\leq\,dP}}\frac{r_{2}\big(m,d;q\big)}{m^{3/4}}\ll\bigg(\frac{P}{H}\bigg)^{2}\sum_{1\,\leq\,d\,\leq\,\sqrt{P}}\frac{1}{d^{q-7/2}}\,\,\underset{m\neq\square}{\sum_{1\,\leq\,\sqrt{m}\,\leq\,dP}}\frac{1}{m^{3/4}}\bigg\{\underset{a\neq0\,,\,b\equiv0(d)}{\sum_{a^{2}+b^{2}=m}}1\bigg\}=\\
&=\bigg(\frac{P}{H}\bigg)^{2}\sum_{1\,\leq\,d\,\leq\,\sqrt{P}}\frac{1}{d^{q-7/2}}\,\,\underset{m\neq\square}{\sum_{1\,\leq\,\sqrt{m}\,\leq\,dP}}\frac{1}{m^{3/4}}\bigg\{\underset{a,b\neq0\,,\,b\equiv0(d)}{\sum_{a^{2}+b^{2}=m}}1\bigg\}\ll\frac{P^{5/2}}{H^{2}}\sum_{1\,\leq\,d\,\leq\,\sqrt{P}}\frac{1}{d^{q-3}}\leq\frac{P^{3}}{H^{2}}\leq1
\end{split}
\end{equation*}
we derive for $q\equiv0\,(2)\,:$
\begin{equation}\label{eq:5.16}
\begin{split}
\quad&\sum_{1\,\leq\,d\,\leq\,\sqrt{P}}\nu(d)\underset{m\neq\square}{\sum_{m}}\alpha\big(m,d\big)\upphi\bigg(\frac{\sqrt{m}}{dP}\bigg)\cos{\Big(2\pi \frac{\sqrt{m}}{d}X+\frac{\pi}{4}\Big)}=\\
&=\frac{1}{4\pi}\sum_{1\,\leq\,d\,\leq\,\sqrt{P}}\frac{\xi(d)}{d^{q-3/2}}\sum_{m}\frac{r_{2}\big(m,d;q\big)}{m^{3/4}}\upphi\bigg(\frac{\sqrt{m}}{dP}\bigg)\cos{\Big(2\pi \frac{\sqrt{m}}{d}X+\frac{\pi}{4}\Big)}+O\big(1\big)
\end{split}
\end{equation}
and for $q\equiv1\,(2)\,:$
\begin{equation}\label{eq:5.17}
\begin{split}
&\left.\textit{(I)}\qquad\qquad\quad\begin{split}
&\sum_{1\,\leq\,d\,\leq\,\sqrt{P}}\nu(d)\underset{m\neq\square}{\sum_{m}}\alpha\big(m,d\big)\upphi\bigg(\frac{\sqrt{m}}{dP}\bigg)\cos{\Big(2\pi \frac{\sqrt{m}}{d}X+\frac{\pi}{4}\Big)}=\\
&=\frac{2^{q-3}}{\pi}\sum_{1\,\leq\,d\,\leq\,\sqrt{P}}\frac{\chi(d)}{d^{q-3/2}}\sum_{m}\frac{r_{2}\big(m,d;q\big)}{m^{3/4}}\upphi\bigg(\frac{\sqrt{m}}{dP}\bigg)\cos{\Big(2\pi \frac{\sqrt{m}}{d}X+\frac{\pi}{4}\Big)}+O\big(1\big)   
\end{split}
    \right.\\\\
    &\left.\textit{(II)}\qquad\qquad\quad\begin{split}
&\sum_{1\,\leq\,d\,\leq\,\sqrt{P}}\eta(d)\underset{m\neq\square}{\sum_{m}}\beta\big(m,d\big)\upphi\bigg(\frac{\sqrt{m}}{dP}\bigg)\cos{\Big(-2\pi \frac{\sqrt{m}}{d}X+\frac{\pi}{4}\Big)}=\\
&=\frac{(-1)^{\frac{q-1}{2}}2^{2q-3}}{\pi}\underset{d\equiv0(4)}{\sum_{1\,\leq\,d\,\leq\,\sqrt{P}}}\frac{1}{d^{q-3/2}}\sum_{m}\frac{r_{2,\chi}\big(m,d;q\big)}{m^{3/4}}\upphi\bigg(\frac{\sqrt{m}}{dP}\bigg)\cos{\Big(-2\pi \frac{\sqrt{m}}{d}X+\frac{\pi}{4}\Big)}+O\big(1\big)\,.
\end{split}
    \right.
\end{split}
\end{equation}
Here, in the second lines of \eqref{eq:5.16}, \eqref{eq:5.17}-(I) and \eqref{eq:5.17}-(II), we have reinserted the sum over $m=\square$ at the cost of:
\begin{equation*}
\begin{split}
\underset{m=\square}{\underset{1\,\leq\,\sqrt{m}\,\leq\,dP}{\sum_{1\,\leq\,d\,\leq\,\sqrt{P}}}}\frac{\big|r_{2,\chi}\big(m,d;q\big)\big|}{d^{q-3/2}m^{3/4}}\leq\underset{m=\square}{\underset{1\,\leq\,\sqrt{m}\,\leq\,dP}{\sum_{1\,\leq\,d\,\leq\,\sqrt{P}}}}\frac{r_{2}\big(m,d;q\big)}{d^{q-3/2}m^{3/4}}\leq\sum_{d=1}^{\infty}\frac{1}{d^{q-3/2}}\sum_{m=1}^{\infty}\frac{r_{2}\big(m^{2}\big)}{m^{3/2}}=O\big(1\big)\,.
\end{split}
\end{equation*}
By \eqref{eq:5.13} and \eqref{eq:5.14}
we obtain:\\\\
\textit{(1)} For $q\equiv0\,(2)$
\begin{equation}\label{eq:5.18}
\begin{split}
\bigintssss\limits_{X-1}^{X+1}F_{P}\big(x-X\big)\Delta_{q}(x)\textit{d}x=\frac{1}{4\pi}\sum_{1\,\leq\,d\,\leq\,\sqrt{P}}\frac{\xi(d)}{d^{q-3/2}}\sum_{m}\frac{r_{2}\big(m,d;q\big)}{m^{3/4}}\upphi\bigg(\frac{\sqrt{m}}{dP}\bigg)\cos{\Big(2\pi \frac{\sqrt{m}}{d}X+\frac{\pi}{4}\Big)}+O\big(1\big)
\end{split}
\end{equation}
\textit{(2)} For $q\equiv1\,(2)$
\begin{equation}\label{eq:5.19}
\begin{split}
\bigintssss\limits_{X-1}^{X+1}F_{P}\big(x-X\big)\Delta_{q}(x)\textit{d}x&=\frac{2^{q-3}}{\pi}\sum_{1\,\leq\,d\,\leq\,\sqrt{P}}\frac{\chi(d)}{d^{q-3/2}}\sum_{m}\frac{r_{2}\big(m,d;q\big)}{m^{3/4}}\upphi\bigg(\frac{\sqrt{m}}{dP}\bigg)\cos{\Big(2\pi \frac{\sqrt{m}}{d}X+\frac{\pi}{4}\Big)}+\\
&+\frac{(-1)^{\frac{q-1}{2}}2^{2q-3}}{\pi}\underset{d\equiv0(4)}{\sum_{1\,\leq\,d\,\leq\,\sqrt{P}}}\frac{1}{d^{q-3/2}}\sum_{m}\frac{r_{2,\chi}\big(m,d;q\big)}{m^{3/4}}\upphi\bigg(\frac{\sqrt{m}}{dP}\bigg)\cos{\Big(-2\pi \frac{\sqrt{m}}{d}X+\frac{\pi}{4}\Big)}+O\big(1\big)\,.
\end{split}
\end{equation}
Finally, multiplying \eqref{eq:5.18} by $4\pi$ and \eqref{eq:5.19} by $2^{3-q}\pi$, and using the simple fact that:
\begin{equation*}
\underset{x\in[X-1,X+1]}{\sup}\,\,\big|\Delta_{q}(x)\big|\geq\bigg|\bigintssss\limits_{X-1}^{X+1}F_{P}\big(x-X\big)\Delta_{q}(x)\textit{d}x\bigg|
\end{equation*}
concludes the proof.
\end{proof}
\subsection{Proof of Theorem 3}
\begin{proof}
(Theorem 3). For $D\in\mathbb{N}$, define $\mathscr{L}(D)$ according to the parity of $q$:
\begin{equation*}
\begin{split}
&q\equiv0\,(2)\,:\quad\mathscr{L}(D)=\frac{1}{\sqrt{2}}\sum_{1\,\leq\,d\,\leq\,D}\frac{\xi(d)}{d^{q-1}}-\frac{2\pi}{D}\sum_{1\,\leq\,d\,\leq\,D}\frac{|\xi(d)|}{d^{q-1}}-\sum_{d\,>\,D}\frac{|\xi(d)|}{d^{q-1}}\\
&q\equiv1\,(2)\,:\quad\mathscr{L}(D)=\frac{1}{\sqrt{2}}\sum_{1\,\leq\,d\,\leq\,D}\frac{\chi(d)}{d^{q-1}}-\frac{2\pi}{D}\sum_{1\,\leq\,d\,\leq\,D}\frac{|\chi(d)|}{d^{q-1}}-\sum_{d\,>\,D}\frac{|\chi(d)|}{d^{q-1}}-\frac{2\pi}{D}\underset{d\equiv0(4)}{\sum_{1\,\leq\,d\,\leq\,D}}\frac{2^{q}}{d^{q-1}}-\underset{d\equiv0(4)}{\sum_{d\,>\,D}}\frac{2^{q}}{d^{q-1}}\,.
\end{split}
\end{equation*}
Since:
\begin{equation*}
\begin{split}
&q\equiv0\,(2)\,:\quad\lim_{D\to\infty}\mathscr{L}(D)=\frac{1}{\sqrt{2}}\sum_{d=1}^{\infty}\frac{\xi(d)}{d^{q-1}}=\Big\{1-2^{1-q}\big(1+(-1)^{\frac{q}{2}+1}\big)\Big\}\zeta(q-1)\geq\frac{3}{2^{5/2}}\zeta(q-1)>0\\
&q\equiv1\,(2)\,:\quad\lim_{D\to\infty}\mathscr{L}(D)=\frac{1}{\sqrt{2}}\sum_{d=1}^{\infty}\frac{\chi(d)}{d^{q-1}}>0
\end{split}
\end{equation*}
there exists an integer $D_{0}=D_{0,q}\geq2$ which depends on $q$ such that:
\begin{equation}\label{eq:5.20}
\mathscr{L}(D_{0})>0\,.
\end{equation}
Choose such $D_{0}$ \textbf{once and for all}.\\\\
We begin by fixing an integer $P$. At the end of the proof we are going to let $P\to\infty$, so in particular we may and we shall assume that $P$ is sufficiently large in terms of $q$ and $D_{0}$. Define:
\begin{equation*}
\begin{split}
\mathscr{A}\big(P\big)=\bigg\{\frac{\sqrt{m}}{d}\leq P\,:\,1\leq d\leq D_{0}\,,\,\exists\, a,b\in\mathbb{Z}\text{ such that }m=a^{2}+b^{2}\text{ with }b\equiv0\,(d)\bigg\}\,.
\end{split}
\end{equation*}
Note that
\begin{equation}\label{eq:5.21}
\big|\mathscr{A}\big(P\big)\big|\leq D_{0}\underset{m=\square+\square}{\sum_{1\,\leq\,m\,\leq\,D^{2}_{0}P^{2}}}1\ll_{\text{}_{D_{0}}}\frac{P^{2}}{\sqrt{\log{P}}}\,.
\end{equation}
By Dirichlet's approximation theorem, there exists an integer $X\in\mathbb{N}$ satisfying the following two conditions:
\begin{equation}\label{eq:5.22}
P^{2}\leq X\leq P^{2}D_{0}^{|\mathcal{A}(P)|}
\end{equation}
and
\begin{equation}\label{eq:5.23}
\quad\qquad\qquad\frac{\sqrt{m}}{d}\in\mathscr{A}\big(P\big)\Longrightarrow\Big|\Big|\frac{\sqrt{m}}{d}X\Big|\Big|\leq\frac{1}{D_{0}}\,.
\end{equation}
We shall appeal to Proposition 5.1, but before we do so we shall need some simple estimates. Let
\begin{equation*}
\omega_{q}=\frac{16(q+1)}{3}\bigintssss\limits_{0}^{1}t^{q-1}\big(1-t^{2}\big)^{1/2}dt\,.
\end{equation*}
Note that $\omega_{q}>0$. Applying partial summation, we have the following estimates:
\begin{equation}\label{eq:5.24}
\begin{split}
&(I)\qquad\sum_{m}\frac{r_{2}\big(m,d;q\big)}{m^{3/4}}\upphi\bigg(\frac{\sqrt{m}}{dP}\bigg)=\omega_{q}\frac{P^{1/2}}{d^{1/2}}+O(1)\\\\
&(II)\qquad\bigg|\sum_{m}\frac{r_{2,\chi}\big(m,d;q\big)}{m^{3/4}}\upphi\bigg(\frac{\sqrt{m}}{dP}\bigg)\bigg|=O(1)\\\\
&(III)\qquad\sum_{m}\frac{\big|r_{2,\chi}\big(m,d;q\big)\big|}{m^{3/4}}\upphi\bigg(\frac{\sqrt{m}}{dP}\bigg)\leq \omega_{q}\frac{P^{1/2}}{d^{1/2}}+O(1)
\end{split}
\end{equation}
valid for any integer $d\in\mathbb{N}$, where the implied constant depends only on $q$. The derivation of (III) follows from the fact that $|r_{2,\chi}\big(m,d;q\big)\big|\leq r_{2}\big(m,d;q\big)$\,.\\\\
Now we appeal to Proposition 5.1 for $\widehat{\Delta}_{q}(X)$ . Consider the sum over $d,m\in\mathbb{N}$ appearing on the RHS of \eqref{eq:5.9} and \eqref{eq:5.10} in which the variable $d$ is restricted to the range $1\leq d\leq D_{0}$. For each such $d$, the sum over $m$ is restricted by the support of $\upphi$ to the range $1\leq\sqrt{m}\leq dP$, and if $\sqrt{m}/d\not\in\mathscr{A}\big(P\big)$ then  $r_{2}\big(m,d;q\big),\,r_{2,\chi}\big(m,d;q\big)\equiv0$. Thus, by \eqref{eq:5.23}, \eqref{eq:5.24} and \eqref{eq:5.20} we obtain:\\\\
\textit{(1)} For $q\equiv0\,(2)$
\begin{equation}\label{eq:5.25}
\begin{split}
&\widehat{\Delta}_{q}(X)\geq\sum_{1\,\leq\,d\,\leq\,\sqrt{P}}\frac{\xi(d)}{d^{q-3/2}}\sum_{m}\frac{r_{2}\big(m,d;q\big)}{m^{3/4}}\upphi\bigg(\frac{\sqrt{m}}{dP}\bigg)\cos{\Big(2\pi \frac{\sqrt{m}}{d}X+\frac{\pi}{4}\Big)}+O\big(1\big)\geq\\
&\omega_{q}\bigg\{\frac{1}{\sqrt{2}}\sum_{1\,\leq\,d\,\leq\,D_{0}}\frac{\xi(d)}{d^{q-1}}-\frac{2\pi}{D_{0}}\sum_{1\,\leq\,d\,\leq\,D_{0}}\frac{|\xi(d)|}{d^{q-1}}-\sum_{d\,>\,D_{0}}\frac{|\xi(d)|}{d^{q-1}}\bigg\}P^{1/2}+O\big(1\big)=\omega_{q}\mathscr{L}(D_{0})P^{1/2}+O\big(1\big)\geq\\
&\geq\frac{1}{2}\omega_{q}\mathscr{L}(D_{0})P^{1/2}\,.
\end{split}
\end{equation}
\textit{(2)} For $q\equiv1\,(2)$
\begin{equation}\label{eq:5.26}
\begin{split}
\widehat{\Delta}_{q}(X)&\geq\sum_{1\,\leq\,d\,\leq\,\sqrt{P}}\frac{\chi(d)}{d^{q-\frac{3}{2}}}\sum_{m}\frac{r_{2}\big(m,d;q\big)}{m^{3/4}}\upphi\bigg(\frac{\sqrt{m}}{dP}\bigg)\cos{\Big(2\pi \frac{\sqrt{m}}{d}X+\frac{\pi}{4}\Big)}+\\
&+(-1)^{\frac{q-1}{2}}2^{q}\underset{d\equiv0(4)}{\sum_{1\,\leq\,d\,\leq\,\sqrt{P}}}\frac{1}{d^{q-\frac{3}{2}}}\sum_{m}\frac{r_{2,\chi}\big(m,d;q\big)}{m^{3/4}}\upphi\bigg(\frac{\sqrt{m}}{dP}\bigg)\cos{\Big(-2\pi \frac{\sqrt{m}}{d}X+\frac{\pi}{4}\Big)}+O\big(1\big)\geq\\
&\geq\omega_{q}\bigg\{\frac{1}{\sqrt{2}}\sum_{1\,\leq\,d\,\leq\,D_{0}}\frac{\chi(d)}{d^{q-1}}-\frac{2\pi}{D_{0}}\sum_{1\,\leq\,d\,\leq\,D_{0}}\frac{|\chi(d)|}{d^{q-1}}-\sum_{d\,>\,D_{0}}\frac{|\chi(d)|}{d^{q-1}}-\frac{2\pi}{D_{0}}\underset{d\equiv0(4)}{\sum_{1\,\leq\,d\,\leq\,D_{0}}}\frac{2^{q}}{d^{q-1}}-\underset{d\equiv0(4)}{\sum_{d\,>\,D_{0}}}\frac{2^{q}}{d^{q-1}}\bigg\}P^{1/2}+O\big(1\big)=\\
&=\omega_{q}\mathscr{L}(D_{0})P^{1/2}+O\big(1\big)\geq\frac{1}{2}\omega_{q}\mathscr{L}(D_{0})P^{1/2}\,.
\end{split}
\end{equation}
Note that by \eqref{eq:5.21} and \eqref{eq:5.22} we obtain:
\begin{equation}\label{eq:5.27}
P^{1/2}\geq\tilde{\omega}_{\text{}_{D_{0}}}\Big(\log{x}\Big)^{1/4}\Big(\log{\log{x}}\Big)^{1/8}\quad;\quad x\in[X-1,X+1]
\end{equation}
for some constant $\tilde{\omega}_{\text{}_{D_{0}}}>0$ which depends $D_{0}$.\\\\
Recalling the definition of $\widehat{\Delta}_{q}(X)$, we deduce by \eqref{eq:5.25} in the case where $q\equiv0\,(2)$, and by \eqref{eq:5.26} in the case where $q\equiv1(2)$, that there exists $x\in[X-1,X+1]$ such that:
\begin{equation}\label{eq:5.28}
\big|\mathcal{E}_{q}\big(\sqrt{x}\,\big)\big|\geq\alpha_{q}x^{q-\frac{1}{2}}\Big(\log{x}\Big)^{1/4}\Big(\log{\log{x}}\Big)^{1/8}
\end{equation}
for some constant $\alpha_{q}=\alpha_{q,D_{0}}>0$.\\\\
Now, as $P\to\infty$ by \eqref{eq:5.22} so does $X$, and thus there are are arbitrarily large values of $x$ for which \eqref{eq:5.28} holds. Replacing $x$ by $x^{2}$ in \eqref{eq:5.28} and adjusting the constant concludes the proof. 
\end{proof}
\phantomsection 
\addcontentsline{toc}{section}{References}
\bibliographystyle{plain}
\bibliography{refs.bib}
Department of Mathematics, Technion - Israel Institute of Technology, Haifa 3200003, Israel\\
\textit{E-mail address:} \href{mailto:yoavgat@campus.tachnion.ac.il}{yoavgat@campus.tachnion.ac.il}

\end{document}